\documentclass{pspum-l}

\usepackage[export]{adjustbox}

\usepackage[skins]{tcolorbox}

\usepackage{multirow, hhline}
\usepackage{verbatim}
\usepackage{mathtools}
\usepackage[shortlabels]{enumitem}
\usepackage{tikz-cd}
\usepackage{xcolor}
\usepackage{hyperref}

\usepackage[textsize=small, color=red]{todonotes}
\usepackage{amssymb} \usepackage{amsbsy,
  amsthm, amsmath,amstext, amsopn} \usepackage[all]{xy}
\usepackage{amsfonts} \usepackage{amscd} \usepackage{stmaryrd}
\usepackage{mathabx}

\newtheorem{thm}{Theorem}[section]

\newtheorem{corollary}[thm]{Corollary}
\newtheorem{cor}[thm]{Corollary}
\newtheorem{prop}[thm]{Proposition}

\theoremstyle{definition}

\newtheorem{defn}[thm]{Definition}

\newtheorem{conj}[thm]{Conjecture}

\newtheorem{exmp}[thm]{Example}
\newtheorem{problem}[thm]{Problem}

\newtheorem{conjecture}[thm]{Conjecture}

\theoremstyle{remark}

\newtheorem{remark}[thm]{Remark}
\newtheorem{rmk}[thm]{Remark}

\begin{document}

\numberwithin{equation}{section}

\newcommand{\hs}{\mbox{\hspace{.4em}}}
\newcommand{\ds}{\displaystyle}
\newcommand{\bd}{\begin{displaymath}}
\newcommand{\ed}{\end{displaymath}}
\newcommand{\bcd}{\begin{CD}}
\newcommand{\ecd}{\end{CD}}

\newcommand{\on}{\operatorname}
\newcommand{\proj}{\operatorname{Proj}}
\newcommand{\bproj}{\underline{\operatorname{Proj}}}
\newcommand{\spec}{\operatorname{Spec}}
\newcommand{\Spec}{\operatorname{Spec}}
\newcommand{\bspec}{\underline{\operatorname{Spec}}}
\newcommand{\pline}{{\mathbf P} ^1}
\newcommand{\aline}{{\mathbf A} ^1}
\newcommand{\pplane}{{\mathbf P}^2}
\newcommand{\cone}{\operatorname{cone}}
\newcommand{\aplane}{{\mathbf A}^2}
\newcommand{\coker}{{\operatorname{coker}}}
\newcommand{\ldb}{[[}
\newcommand{\rdb}{]]}

\newcommand{\bO}{\mathbb{O}}

\newcommand{\Sym}{\operatorname{Sym}^{\bullet}}
\newcommand{\Symp}{\operatorname{Sym}}
\newcommand{\Pic}{\bf{Pic}}
\newcommand{\Aut}{\operatorname{Aut}}
\newcommand{\codim}{\operatorname{codim}}
\newcommand{\PAut}{\operatorname{PAut}}

\newcommand{\kos}{{\underset{\clap{\scriptsize \it Kos}}{\;\sim\;}}}

\newcommand{\Fqbar}{\overline{\mathbb{F}}_q}
\newcommand{\Fq}{{\mathbb{F}_q}}

\newcommand{\too}{\twoheadrightarrow}
\newcommand{\C}{{\mathbb C}}
\newcommand{\Z}{{\mathbb Z}}
\newcommand{\Q}{{\mathbb Q}}
\newcommand{\A}{\mathbb{A}}
\newcommand{\R}{{\mathbb R}}
\newcommand{\Cx}{{\mathbb C}^{\times}}
\newcommand{\Cbar}{\overline{\C}}
\newcommand{\Cxbar}{\overline{\Cx}}
\newcommand{\cA}{{\mathcal A}}
\newcommand{\fA}{{\mathfrak A}}
\newcommand{\cS}{{\mathcal S}}
\newcommand{\cV}{{\mathcal V}}
\newcommand{\cM}{{\mathcal M}}
\newcommand{\bA}{{\mathbf A}}
\newcommand{\cB}{{\mathcal B}}
\newcommand{\cC}{{\mathcal C}}
\newcommand{\cD}{{\mathcal D}}
\newcommand{\D}{{\mathcal D}}
\newcommand{\cs}{{\mathbf C} ^*}
\newcommand{\boldc}{{\mathbf C}}
\newcommand{\cE}{{\mathcal E}}
\newcommand{\cF}{{\mathcal F}}
\newcommand{\bF}{{\mathbb F}}
\newcommand{\cG}{{\mathcal G}}
\newcommand{\G}{{\mathbb G}}
\newcommand{\cH}{{\mathcal H}}
\newcommand{\bH}{{\mathbf H}}
\newcommand{\CI}{{\mathcal I}}
\newcommand{\cJ}{{\mathcal J}}
\newcommand{\cK}{{\mathcal K}}
\newcommand{\cL}{{\mathcal L}}
\newcommand{\baL}{{\overline{\mathcal L}}}
\newcommand{\M}{{\mathcal M}}
\newcommand{\Mf}{{\mathfrak M}}
\newcommand{\bM}{{\mathbf M}}
\newcommand{\bm}{{\mathbf m}}
\newcommand{\cN}{{\mathcal N}}
\newcommand{\theo}{\mathcal{?}{O}}
\newcommand{\cP}{{\mathcal P}}
\newcommand{\cR}{{\mathcal R}}
\newcommand{\Pp}{{\mathbb P}}
\newcommand{\boldp}{{\mathbf P}}
\newcommand{\boldq}{{\mathbf Q}}
\newcommand{\bbL}{{\mathbf L}}
\newcommand{\cQ}{{\mathcal Q}}
\newcommand{\cO}{{\mathcal O}}
\newcommand{\cT}{{\mathcal T}}
\newcommand{\Oo}{{\mathcal O}}
\newcommand{\cY}{{\mathcal Y}}
\newcommand{\OX}{{\Oo_X}}
\newcommand{\OY}{{\Oo_Y}}
\newcommand{\all}{{\mathrm{all}}}
\newcommand{\cZ}{{\mathcal Z}}
\newcommand{\bD}{\mathbb{D}}
\newcommand{\DMod}{\mathcal{D}}
\newcommand{\cDMod}{\breve{\mathcal{D}}}
\newcommand{\rnD}{\cDMod}
\newcommand{\otY}{{\underset{\OY}{\ot}}}
\newcommand{\otX}{{\underset{\OX}{\ot}}}
\newcommand{\cU}{{\mathcal U}}\newcommand{\cX}{{\mathcal X}}
\newcommand{\cW}{{\mathcal W}}
\newcommand{\boldz}{{\mathbf Z}}
\newcommand{\Rees}{\operatorname{Rees}}
\newcommand{\IC}{\IndCoh}
\newcommand{\IndPerf}{\operatorname{IndPerf}}
\newcommand{\ssupp}{\operatorname{SS}}
\newcommand{\qgr}{\operatorname{q-gr}}
\newcommand{\gr}{\operatorname{gr}}
\newcommand{\rk}{\operatorname{rk}}
\newcommand{\Sh}{\mathcal{S}hv}
\newcommand{\Shv}{\Sh}
\newcommand{\SH}{{\underline{\operatorname{Sh}}}}
\newcommand{\End}{\operatorname{End}}
\newcommand{\uEnd}{\underline{\operatorname{End}}}
\newcommand{\Hom}{\operatorname{Hom}}
\newcommand{\uHom}{\underline{\operatorname{Hom}}}
\newcommand{\Sing}{\operatorname{Sing}}
\newcommand{\uHomY}{\uHom_{\OY}}
\newcommand{\uHomX}{\uHom_{\OX}}
\newcommand{\Ext}{\operatorname{Ext}}
\newcommand{\bExt}{\operatorname{\bf{Ext}}}
\newcommand{\Tor}{\operatorname{Tor}}

\newcommand*\brslash{\vcenter{\hbox{\includegraphics[height=3.6mm]{brokenslash4.png}}}}

\newcommand{\inv}{^{-1}}
\newcommand{\airtilde}{\widetilde{\hspace{.5em}}}
\newcommand{\airhat}{\widehat{\hspace{.5em}}}
\newcommand{\nt}{^{\circ}}
\newcommand{\del}{\partial}

\newcommand{\supp}{\operatorname{supp}}
\newcommand{\GK}{\operatorname{GK-dim}}
\newcommand{\hd}{\operatorname{hd}}
\newcommand{\pt}{\operatorname{pt}}
\newcommand{\id}{\operatorname{id}}
\newcommand{\res}{\operatorname{res}}
\newcommand{\lrar}{\leadsto}
\newcommand{\im}{\operatorname{Im}}
\newcommand{\hh}{HH}
\newcommand{\hn}{HN}
\newcommand{\hc}{HC}
\newcommand{\hp}{HP}
\newcommand{\Gal}{\operatorname{Gal}}
%I changed these to just be regular italics because I think it makes things like HH(QC) or HH(A-mod) more readable, i.e. inside is regular font?

\newcommand{\TF}{\operatorname{TF}}
\newcommand{\Bun}{\operatorname{Bun}}

\newcommand{\mix}{{\mathrm{m}}}

\newcommand{\F}{\mathcal{F}}
\newcommand{\Ff}{\mathbb{F}}
\newcommand{\nthord}{^{(n)}}
\newcommand{\Gr}{{\mathfrak{Gr}}}

\newcommand{\BB}{\mathbb{B}}

\newcommand{\Ting}{\mathbb{T}^*_{[\mh 1]}}

\newcommand{\bL}{\mathbb{L}}
\newcommand{\deeq}{{\mathrm{deq}}}
\newcommand{\Fr}{\operatorname{Fr}}
\newcommand{\GL}{\operatorname{GL}}
\newcommand{\Perv}{\operatorname{Perv}}
\newcommand{\gl}{\mathfrak{gl}}
\newcommand{\SL}{\operatorname{SL}}
\newcommand{\KPerf}{\operatorname{KPerf}}
\newcommand{\ff}{\footnote}
\newcommand{\ot}{\otimes}
\def\Ext{\operatorname {Ext}}
\def\Hom{\operatorname {Hom}}
\def\Ind{\operatorname {Ind}}
\newcommand{\MHM}{\operatorname{MHM}}

\def\bbZ{{\mathbb Z}}

\newcommand{\Irr}{\mathrm{Irr}}

\newcommand{\nc}{\newcommand}
\nc{\ol}{\overline} \nc{\cont}{\on{cont}} \nc{\rmod}{\on{mod}}
\nc{\Mtil}{\widetilde{M}} \nc{\wb}{\overline} \nc{\wt}{\widetilde}
\nc{\wh}{\widehat} \nc{\sm}{\setminus} \nc{\mc}{\mathcal}
\nc{\mbb}{\mathbb}  \nc{\K}{{\mc K}} \nc{\Kx}{{\mc K}^{\times}}
\nc{\Ox}{{\mc O}^{\times}} \nc{\unit}{{\bf \on{unit}}}
\nc{\boxt}{\boxtimes} \nc{\xarr}{\stackrel{\rightarrow}{x}}

\newcommand{\cLX}{{\cL}X}

\nc{\Ga}{\G_a}
 \nc{\PGL}{{\on{PGL}}}
 \nc{\PU}{{\on{PU}}}

\nc{\h}{{\mathfrak h}} \nc{\kk}{{\mathfrak k}}
 \nc{\Gm}{\G_m}
\nc{\Gabar}{\wb{\G}_a} \nc{\Gmbar}{\wb{\G}_m} \nc{\Gv}{{\check G}}
\nc{\Tv}{\check T} \nc{\Bv}{\check B} \nc{\Pv}{\check P}
\nc{\g}{{\mathfrak g}}
\nc{\gv}{\check {\mathfrak g}} \nc{\RGv}{\on{Rep}\Gv}
\nc{\RTv}{\on{Rep}\check T}
 \nc{\Flv}{\check{\mathcal B}}
 \nc{\TFlv}{T^*\Flv}
 \nc{\Fl}{{\mathfrak Fl}}
\nc{\RR}{{\mathcal R}} \nc{\Nv}{\check {\mathcal{N}}}
\nc{\St}{{\mathcal St}} \nc{\ST}{{\underline{\mathcal St}}}
\nc{\Hec}{{\bf{\mathcal H}}} \nc{\Hecblock}{{\bf{\mathcal
H_{\alpha,\beta}}}} \nc{\dualHec}{{\bf{\check \mathcal H}}}
\nc{\dualHecblock}{{\bf{\check \mathcal H_{\alpha,\beta}}}}
\newcommand{\ramBun}{{\bf{Bun}}}
\newcommand{\ramBuno}{\ramBun^{\circ}}

\newcommand{\bB}{\mathbb{B}}

\nc{\Buntheta}{{\bf Bun}_{\theta}} \nc{\Bunthetao}{{\bf
Bun}_{\theta}^{\circ}} \nc{\BunGR}{{\bf Bun}_{G_\R}}
\nc{\BunGRo}{{\bf Bun}_{G_\R}^{\circ}}
\nc{\HC}{{\mathcal{HC}}}
\nc{\risom}{\stackrel{\sim}{\to}} \nc{\Hv}{{\check H}}
\nc{\bS}{{\mathbf S}}
\nc{\bC}{\mathbf{C}}

\def\Conn{\operatorname {Conn}}

\nc{\Vect}{{\operatorname{Vect}}}
\nc{\Hecke}{{\operatorname{Hecke}}}

\nc{\twistor}{{\widetilde{\Pp}^1_\R}}

\newcommand{\ZZ}{{Z_{\bullet}}}
\nc{\HZ}{{\mc H}\ZZ} \nc{\eps}{\epsilon}

\nc{\CN}{\mathcal N} \nc{\BA}{\mathbb A}
\nc{\XYX}{X\times_Y X}

\nc{\bQ}{\mathbb{Q}}

\nc{\ul}{\underline}

\nc{\bn}{\mathbf n} \nc{\Sets}{{\on{Sets}}} \nc{\Top}{{\on{Top}}}

\nc{\Simp}{{\mathbf \Delta}} \nc{\Simpop}{{\mathbf\Delta^\circ}}

\nc{\Cyc}{{\mathbf \Lambda}} \nc{\Cycop}{{\mathbf\Lambda^\circ}}

\nc{\Mon}{{\mathbf \Lambda^{mon}}}
\nc{\Monop}{{(\mathbf\Lambda^{mon})\circ}}

\nc{\Aff}{{\on{Aff}}} \nc{\Sch}{{\on{Sch}}}

\nc{\La}{\mathcal La}

\newcommand{\minus}{\scalebox{0.5}[1.0]{$-$}}
\renewcommand{\ng}{\minus}

\newcommand{\hCoh}{\wh{\Coh}}

\nc{\bul}{\bullet}
\nc{\module}{{\operatorname{-mod}}}

\nc{\dstack}{{\mathcal D}}

\nc{\BL}{{\mathbb L}}

\nc{\BD}{{\mathbb D}}

\nc{\BR}{{\mathbb R}}

\nc{\BT}{{\mathbb T}}

\nc{\bT}{\mathbb{T}}

\nc{\SCA}{{\mc{SCA}}}
\nc{\DGA}{{\mc DGA}}

\nc{\DSt}{{DSt}}

\nc{\lotimes}{{\otimes}^{\mathbf L}}

\nc{\bs}{\backslash}

\nc{\Lhat}{\widehat{\mc L}}

\newcommand{\Coh}{{\on{Coh}}}

\newcommand{\TT}{\mathbb{T}^{*[\mh 1]}}

\nc{\QC}{\operatorname{QC}}
\nc\Perf{\on{Perf}}
\nc{\Cat}{{\on{Cat}}}
\nc{\dgCat}{{\on{dgCat}}}
\nc{\bLa}{{\mathbf \Lambda}}
\nc{\QCoh}{\QC}
\newcommand{\IndCoh}{\QC^!}

\nc{\RHom}{\mathbf{R}\hspace{-0.15em}\on{Hom}}
\nc{\REnd}{\mathbf{R}\hspace{-0.15em}\on{End}}
%\nc{\colim}{\on{colim}}
\nc{\oo}{\infty}
\nc\Mod{\on{Mod}}

\nc\fh{\mathfrak h}
\nc\al{\alpha}
\nc\la{\alpha}
\nc\BGB{B\bs G/B}
\nc\QCb{QC^\flat}
\nc\qc{\cQ}

\nc{\fg}{\mathfrak g}

\nc{\fgv}{\check\fg}
\nc{\fn}{\mathfrak n}
\nc{\Map}{\on{Map}} \nc{\fX}{\mathfrak X}

\nc{\Tate}{\operatorname{Tate}}

\nc{\ch}{\check}
\nc{\fb}{\mathfrak b} \nc{\fu}{\mathfrak u} \nc{\st}{{st}}
\nc{\fU}{\mathfrak U}
\nc{\fZ}{\mathfrak Z}

\nc\fk{\mathfrak k} \nc\fp{\mathfrak p}

\nc{\RP}{\mathbf{RP}} \nc{\rigid}{\text{rigid}}
\nc{\glob}{\text{glob}}

\nc{\cI}{\mathcal I}

\nc{\quot}{/\hspace{-.25em}/}

\nc\aff{\it{aff}}
\nc\BS{\mathbb S}

\nc\Loc{{\mc Loc}}
\nc\Ch{{\mc Ch}}

\nc\git{/\hspace{-0.2em}/}
\nc{\fc}{\mathfrak c}
\nc\BC{\mathbb C}
\nc\BZ{\mathbb Z}
\nc\bZ{\mathbb Z}

\nc\stab{\text{\it st}}
\nc\Stab{\text{\it St}}

\nc\perf{\on{-perf}}

\nc\intHom{\mathcal{H}om}
\nc\intEnd{\mathcal{E}nd}

\nc\gtil{\widetilde\fg}

\newcommand{\shear}{{\mathbin{\mkern-6mu\fatslash}}}
\newcommand{\unshear}{{\,\mathbin{\mkern-6mu\fatbslash}}}

\def\adjquot{/_{\hspace{-0.2em}ad}\hspace{0.1em}}

\nc\mon{\text{\it mon}}
\nc\bimon{\text{\it bimon}}
\nc\uG{{\underline{G}}}
\nc\uB{{\underline{B}}}
\nc\uN{{\underline{\cN}}}
\nc\uNtil{{\underline{\wt{\cN}}}}
\nc\ugtil{{\underline{\wt{\fg}}}}
\nc\uH{{\underline{H}}}
\nc\uX{{\ul{X}}}
\nc\uY{\ul{Y}}
\nc\upi{\ul{\pi}}
\nc{\uZ}{\ul{Z}}
\nc{\ucZ}{\ul{\cZ}}
\nc{\ucH}{\ul{\cH}}
\nc{\ucS}{\ul{\cS}}
\nc{\ahat}{{\wh{a}}}
\nc{\shat}{{\wh{s}}}
\nc{\DCoh}{{\rm DCoh}}

\newcommand{\Lie}{\operatorname{Lie}}

\newcommand{\cind}{\operatorname{cInd}}
\newcommand{\LL}{\cL}

\newcommand{\Tot}{\operatorname{Tot}}

\newcommand{\mf}{\mathfrak}

\newcommand{\mantodo}[1]{\textbf{\textcolor{red}{todo: #1}}}
\newcommand{\help}[1]{\todo[color=green]{HELP: #1}}
\newcommand{\comm}[1]{\todo[color=green]{HC: #1}}
\newcommand{\chen}[1]{\todo[color=green]{HC: #1}}
\newcommand{\dbz}[1]{\todo[color=orange]{DBZ: #1}}
\newcommand{\helm}[1]{\todo[color=yellow]{DH: #1}}
\newcommand{\nadler}[1]{\todo[color=blue]{DN: #1}}

\newcommand{\gwt}{\operatorname{wt}_{\mathbb{G}_m}}

%harrison's junk
\newcommand{\cat}{\mathbf}
\newcommand{\OO}{\cO}
\newcommand{\dmod}{\operatorname{-mod}}
\newcommand{\DMOD}{\textbf{-mod}}
\newcommand{\dperf}{\operatorname{-perf}}
\newcommand{\dcmod}{\operatorname{-cmod}}
\newcommand{\dtors}{\operatorname{-tors}}

\newcommand{\dcomod}{\operatorname{-comod}}
\newcommand{\dcoh}{\operatorname{-coh}}
\newcommand{\ICoh}{\operatorname{ICoh}}
\newcommand{\dR}{\operatorname{dR}}
\newcommand{\LLf}{\widehat{\cL}}
%this is G x Gm
\newcommand{\GG}{G_{\mathrm{gr}}}
\newcommand{\ggr}{{\mathrm{gr}}}
\newcommand{\Hgr}{{H_{\mathrm{gr}}}}
\newcommand{\Bgr}{{B_{\mathrm{gr}}}}

\newcommand{\Fun}{\operatorname{Fun}}

\newcommand{\NN}{\widetilde{\mathcal{N}}}

\newcommand{\actson}{\circlearrowright}

%affine Hecke algebra shortcut
\newcommand{\Haff}{\mathcal{H}}
\newcommand{\grH}{\bar{\mathcal{H}}}

%affine Hecke category shortcut
\newcommand{\Hcat}{\mathbf{H}}

\newcommand{\ubQ}{\underline{\mathbb{Q}}}

\newcommand{\sboxed}[1]{%
  \,\raisebox{-1ex}[\height][0pt]{\fbox{#1}}%
}
\newcommand{\uboxed}[1]{%
  \makebox[0pt]{\fbox{$\scriptstyle\mathstrut#1$}}%
}

\newcommand{\Rep}{\mathrm{Rep}}

%traces
\nc\Tr{{\mathbf{Tr}}} %"universal trace functor"
\newcommand{\tr}{\mathcal{T}r} %categorical trace
\newcommand{\chern}{\operatorname{ch}}

\newcommand{\tilN}{\widetilde{\mathcal{N}}}

\newcommand{\un}{{un}}

%tate shearing functor (so we cna change later)
\newcommand{\TS}{\mathbf{T}}

%nilpotent singular support: capital nu?
\newcommand{\DLnil}{\mathcal{V}}

\newcommand{\bQl}{\overline{\mathbb{Q}}_\ell}

\newcommand{\tens}[1]{%
  \mathbin{\mathop{\otimes}\limits_{#1}}%
}

\newcommand{\utimes}[1]{%
  \mathbin{\mathop{\times}\limits_{#1}}%
}

\mathchardef\md="2D
\mathchardef\mh="2D

\newcommand{\BG}{\cC}

%naive hodge
\newcommand{\NH}{\mathbf{H}}

\newcommand{\bfI}{{\widecheck{\mathbf{I}}}}
\newcommand{\bfG}{{\widecheck{\mathbf{G}}}}

\newcommand{\LS}{\mathcal{LS}}
\newcommand{\APerf}{\operatorname{APerf}}

\newcommand{\bG}{\mathbb{G}}

%%%%%%%%%%%%%%%%%%%%%%%%%%%%%%%%%%%%%%%%%%%
%%%%%%%%%%%%%%%%%%%%%%%%%%%%%%%%%%%%%%%%%%%
%%%%%%%%%%%%%%%%%%%%%%%%%%%%%%%%%%%%%%%%%%%
%%%%%%%%%%%%%%%%%%%%%%%%%%%%%%%%%%%%%%%%%%%

\title[Circle Actions and Local Langlands]{Between Coherent and Constructible Local Langlands Correspondences}

\author{David Ben-Zvi} \address{Department of Mathematics\\University
  of Texas\\Austin, TX 78712-0257} \email{benzvi@math.utexas.edu}
\author{Harrison Chen} \address{Institute of Mathematics\\Academia Sinica\\Taipei 106319, Taiwan} \email{chenhi.math@gmail.com}
  
  \author{David Helm} \address{Department of Mathematics\\Imperial College\\  London SW7 2BU, United Kingdom} \email{dhelm@math.utexas.edu}
\author{David Nadler} \address{Department of Mathematics\\University
  of California\\Berkeley, CA 94720-3840}
\email{nadler@math.berkeley.edu}

\subjclass[2020]{Primary 20G25}

\date{}

\begin{abstract}
Refined forms of the local Langlands correspondence seek to relate representations of reductive groups over local fields with sheaves on stacks of Langlands parameters. But what kind of sheaves? Conjectures in the spirit of Kazhdan-Lusztig theory (due to Vogan and Soergel) describe representations of a group and its pure inner forms with fixed central character in terms of constructible sheaves. Conjectures in the spirit of geometric Langlands (due to Fargues, Zhu and Hellmann) describe representations with varying central character of a large family of groups associated to isocrystals in terms of coherent sheaves. The latter conjectures also take place on a larger parameter space, in which Frobenius (or complex conjugation) is allowed a unipotent part. 

In this article we propose a general mechanism that interpolates between these two settings. This mechanism derives from the theory of cyclic homology, as interpreted through circle actions in derived algebraic geometry. We apply this perspective to categorical forms of the local Langlands conjectures for both archimedean and non-archimedean local fields. In the nonarchimedean case, we describe how circle actions relate coherent and constructible realizations of affine Hecke algebras and of all smooth representations of $GL_n$, and propose a mechanism to relate the two settings in general. In the archimedean case, we explain how to use circle actions to derive the constructible local Langlands correspondence (in the form due to Adams-Barbasch-Vogan and Soergel) from a coherent form (a real counterpart to Fargues' conjecture): the tamely ramified geometric Langlands conjecture on the twistor line, which we survey.
\end{abstract}

\maketitle

%%%%%%%%%%%%%%%%%%%%%%%%%%%%%%%%%%%%%%%%%%%
%%%%%%%%%%%%%%%%%%%%%%%%%%%%%%%%%%%%%%%%%%%
%%%%%%%%%%%%%%%%%%%%%%%%%%%%%%%%%%%%%%%%%%%
%%%%%%%%%%%%%%%%%%%%%%%%%%%%%%%%%%%%%%%%%%%

\tableofcontents

\section{Overview: Sheaves on Langlands Parameters}
The fundamental theme of spectral decomposition in representation theory seeks to describe representations of a group $G$ in terms of a dual object $\wh{G}$, which parametrizes (suitable) irreducible representations. Harmonic analysis asks for the spectral description of large representations of $G$ (such as functions on $G$-spaces) as families of vector spaces (the multiplicity spaces) over $\wh{G}$. We might also seek to describe indecomposable or standard modules in terms of irreducibles and solve extension problems, or more generally describe families of representations, in terms of the geometry of $\wh{G}$. 

\medskip

 This theme underlies much of geometric representation theory, where one seeks to describe representations as sheaves on a parameter space, and utilize the geometry of sheaves to solve representation-theoretic problems. There are two main types of sheaves that play this role: {\em constructible} (taken in the broad sense to include perverse sheaves as well as $\D$-modules) and {\em coherent} (used informally to include quasicoherent and ind-coherent sheaves).  

\medskip

{\bf N.B.:} We restrict our attention completely to characteristic zero representation theory in this article, and work in a derived ($\infty$-categorical) rather than abelian setting, so that all categories will be $k$-linear dg categories (or $k$-linear stable $\infty$-categories) for a field $k\supset \mathbb Q$.

\subsection{Constructible sheaves in representation theory.}
Let us recall two classic instances of the role of constructible sheaves, Kazhdan-Lusztig theory and Springer theory, which serve as prototypes for 
the local Langlands correspondence over archimedean and non-archimedean fields, respectively. In these and other situations 
of interest in geometric representation theory, we are dealing with equivariance for group actions with finitely many orbits, and can work equivalently with either Betti or de Rham versions of equivariant derived categories, i.e., with either constructible sheaves or $\D$-modules. 

\medskip

Kazhdan-Lusztig theory identifies category $\cO_0$ of highest-weight modules (with trivial infinitesimal character) for a complex Lie algebra $\fg$ as the category of $N$-equivariant perverse sheaves on the flag variety $\cB=G/B$. This description sets up a bijection between irreducibles in category $\cO_0$ and the set of $N$-orbits on the flag variety (which are simply connected hence carry a unique local system). Much more significantly it lets one describe composition series of indecomposable modules in terms of the equivariant geometry of the flag variety, as captured by its category of sheaves (the subject of the Kazhdan-Lusztig conjecture). 

\medskip

In Springer theory, we consider complex representations of the Chevalley group $G(\Ff_q)$, 
and in particular those representations that admit a $B(\Ff_q)$-fixed vector, the unipotent principal series representations. Springer theory realizes these representations inside the category of $G$-equivariant perverse sheaves on the nilpotent cone $\cN\subset \fg$. This goes as follows: unipotent principal series representations are identified with modules for the finite Hecke algebra 
$$\Haff^f=\C[B(\Ff_q) \bs G(\Ff_q) / B(\Ff_q)] = \End_{G(\Ff_q)}(\C[G(\Ff_q)/B(\Ff_q)]),$$ which (after choosing a square-root of $q$) may be identified with the Weyl group $W$ of $G$ (though that identification is in a sense incidental to our story). 
Let $\mu:\tilN=T^{\ast} \cB\to \cN$ denote the Springer resolution, and consider the Springer sheaf $$\bS=\mu_* \C_{\tilN/G}[\dim(G/B)]\in \Perv(\cN/G).$$
There is an isomorphism $$\Haff^f \simeq \End_{\cN/G}(\bS)$$ between the finite Hecke algebra and the endomorphisms of the Springer sheaf. The Springer sheaf is projective (in fact the whole category $\Perv(\cN/G)$ is semisimple), whence we obtain a full embedding 
$$\{\mbox{unipotent principal series of $G(\Ff_q)$}\}\simeq \Haff^f\module\stackrel{\sim}{\longrightarrow} \langle \bS\rangle \subset \Perv(\cN/G)$$
 of Hecke-modules into perverse sheaves as the subcategory generated by the Springer sheaf. As a result, irreducible unipotent principal series representations are classified by nilpotent orbits together with {\em certain} equivariant local systems on them, or equivalently representations of the component group of the centralizer of a nilpotent. (The missing local systems are accounted for by Lusztig's generalized Springer correspondence.)

\subsection{Constructible sheaves in the local Langlands program.}\label{local Langlands intro}
The local Langlands correspondence in Vogan's formulation~\cite{vogan} proposes an analogous classification of representations of reductive groups over local fields. Let us fix a local field $K$ and a field $k$ of characteristic zero. We also fix for simplicity a connected split reductive group $G$ over $K$ with Langlands dual group $\Gv$, an algebraic group over $k$.
The local Langlands correspondence
provides a map
$$\{\mbox{Irreducible smooth reps. of }G(K)\}\longrightarrow \LL_\Gv(K):= \{ \cL_K\longrightarrow \Gv\mbox{ continuous}\}/\Gv$$ from representations to Langlands parameters, [continuous] representations of the Weil(-Deligne) group of $K$ into $G$. The fibers of this map, the {\em L-packets}, are expected to be parametrized by {\em certain} irreducible representations of the component group of the centralizer of the corresponding Langlands parameter -- i.e., by certain irreducible equivariant local systems. To see the missing local systems, one replaces the single group $G(K)$ by its collection of pure inner forms of (which arises naturally when thinking sheaf-theoretically about representations of $G(K)$, see e.g.~\cite{bernstein stacks} and Section~\ref{automorphic real} below).

\medskip

Thus, we obtain a conjectural bijection between equivariant local systems on the orbit of a Langlands parameter and a collection of irreducible representations of pure inner forms of $G(K)$. This lead to Vogan's conjectures~\cite{vogan}, providing a complete description on the level of Grothendieck groups of the representation theory of $G(K)$ and its pure inner forms with fixed central character in terms of constructible sheaves on suitable spaces of Langlands parameters, in the spirit of Kazhdan-Lusztig theory and Springer theory. The relevant spaces of Langlands parameters involve a fixed infinitesimal character (as described in~\cite{vogan}), which means in particular fixing a semisimple conjugacy class in $\Gv$ in the non-archimedean setting (the class of a semisimple Frobenius element) or a semisimple conjugacy class in $\fgv$ in the archimedean setting (coming from the derivative of the action of the Weil group). 

\medskip
In the archimedean setting this picture was established by Adams-Barbasch-Vogan~\cite{ABV} (cf. Mason-Brown's lecture~\cite{MB}). Moreover Soergel~\cite{soergel} conjectured a categorical enhancement of these results, asserting that categories of Harish-Chandra modules (with fixed infinitesimal character) for pure inner forms of a real reductive group are Koszul dual to equivariant derived categories of constructible sheaves on the ABV spaces of Langlands parameters. Soergel's conjecture was proved in~\cite{soergel} for complex groups and $SL_2(\R)$ and in~\cite{RomaKari} for quasisplit blocks. 

\medskip
In the nonarchimedean setting, Kazhdan and Lusztig~\cite{KL} and Ginzburg \cite{CG} established an affine counterpart to Springer theory, identifying the representation theory of the affine Hecke algebra (or equivalently unramified principal series representations of $G(K)$) in terms of (certain) equivariant constructible sheaves on the spaces of unipotent Langlands parameters. 
Lusztig~\cite{lusztig unipotent} (following~\cite{cuspidal1,cuspidal2}) completed this picture to an ``affine generalized Springer correspondence", establishing the local Langlands correspondence between all unipotent representations of pure inner forms of $G(K)$ for $G$ adjoint, simple and unramified and all equivariant constructible sheaves on unipotent Langlands parameters.  
This work has recently been extended by Solleveld~\cite{soluni,soluni2} first to all unipotent representations of connected $G$ and then~\cite{sol1,sol2} to prove Vogan's $p$-adic Kazhdan-Lusztig conjecture for a much broader class of representations.
The main input in all these developments is an identification Hecke algebras associated to blocks of representations with Ext-algebras of suitable perverse sheaves. When combined with formality theorems as in~\cite{rider, RiRu1,RiRu2,sol2} this results in derived equivalences between categories of representations and categories of constructible sheaves on Langlands parameters. 

\medskip
One disadvantage of this constructible picture is that it appears inadequate to describe variation of the continuous parameters -- indeed the spaces of Langlands parameters appearing in Vogan's conjectures vary discontinuously as a function of infinitesimal characters. 

\subsection{Coherent sheaves in representation theory.} 
The origin of the role of coherent sheaves is the identification of all modules for a commutative algebra $Z$ as quasicoherent sheaves on $\Spec(Z)$. If $Z$ arises as the center of an associative algebra $A$ then we obtain a localization functor from $A$-modules to quasicoherent sheaves on $\Spec(Z)$. More abstractly a category $\cC$ sheafifies over the spectrum of its Bernstein center
(or Hochschild cohomology), with Hom spaces localizing to quasicoherent sheaves, providing a spectral decomposition of $\cC$. Thus coherent sheaves appear naturally in contexts where it is important to study families of representations, in particular with varying central character. Among the many motivations in the setting of the local Langlands correspondence we mention connections with harmonic analysis (see e.g.~\cite{SV}), $K$-theory and the Baum-Connes conjecture~\cite{ABPS}, and modular and integral representation theory~\cite{emerton helm, curtis, HM converse}.

\medskip
The Langlands correspondence encodes ambitious algebraic spectral decompositions of this flavor, with spaces of Langlands parameters playing the role of central parameters. 
On one hand, the theory of (unramified) Hecke operators identifies a large commutative algebra of symmetries on automorphic forms (in classical Langlands) or automorphic sheaves (in geometric Langlands) for a reductive group $G$. On the other hand, the spectrum of this commutative algebra is related (via the Satake correspondence) with a space of Langlands parameters into the dual group $\Gv$. We then seek to spectrally decompose the automorphic side in terms of algebraic geometry of Langlands parameters. This leads to the geometric Langlands conjecture (in any of its variants) for curves over algebraically closed fields, in which a category of automorphic sheaves is identified with a category of (ind-)coherent sheaves on a stack of local systems on the curve. 

\medskip
Recently a new ``coherent" categorical form of the local Langlands correspondence has emerged, inspired by the geometric Langlands program. in which categories of representations of groups over local fields are described via coherent sheaves on stacks of Langlands parameters (for details see the lectures of Fargues-Scholze, Emerton-Hellmann-Gee and Zhu in this summer school). In the archimedean setting such a coherent formulation was proposed as the tamely geometric Langlands correspondence over the twistor line in~\cite{loops and parameters, loops and reps}, and will be discussed in this article. In the non-archimedean setting such a picture began to emerge from several different directions (see in particular~\cite{lafforgue,lafforgue ICM,lafforgue zhu,genestier,dennis shtuka, GKRV}), leading to conjectures of Zhu~\cite{xinwen survey}, Hellmann~\cite{hellmann}) and Fargues~\cite{fargues}. This perspective was profoundly developed in the monumental work of Fargues and Scholze~\cite{farguesscholze}, which in particular establishes the ``automorphic\hspace{0ex}-to-Galois" direction, showing that the automorphic category (and thus its subcategory of representations of $G(K)$) sheafifies over the stack of Langlands parameters.
Upcoming work~\cite{hemo zhu} of Hemo and Zhu applies the theory of categorical traces (as in~\cite{xinwen trace}) and Bezrukavnikov's tamely ramified local geometric Langlands correspondence~\cite{roma ICM,roma hecke} to establish a coherent local Langlands correspondence for unipotent representations (the principal series part of which is proved in~\cite{BCHN}). 

\medskip
Categories of coherent sheaves on stacks are much larger than those of constructible sheaves. For example, at a point with stabilizer group $H$, quasicoherent sheaves are indexed by all algebraic representations of $H$ while constructible sheaves correspond only to representations of the component group of $H$ (or in the derived setting to modules for chains on $H$). 
Thus we must consider a much larger representation theoretic side to match the entire categories of coherent sheaves on Langlands parameters (as opposed to distinguished full subcategories). 
Indeed, rather than the finite collection of pure inner forms appearing in Vogan's conjectures, the conjectures of~\cite{xinwen survey, fargues, farguesscholze} identify coherent sheaves on Langlands parameters with representations of an infinite family of groups. These groups, which are all (groups of points of) inner forms of Levis of $G$, arise as the automorphism groups of $G$-isocrystals (or of points in a suitable stack of $G$-bundles) and are indexed by Kottwitz' set $B(G)$.

\subsection{Relating algebra and topology}
Broadly speaking, we have described the following two flavors of the local Langlands correspondence:
\medskip

 {\small 
 %\begin{center}
%\centering
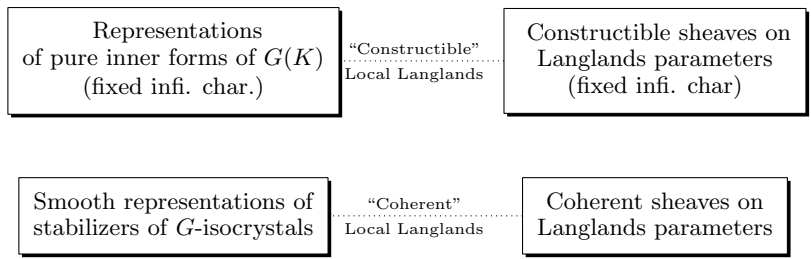
\xymatrix@C=1cm{
    % Second row
%  &    Automorphic && Spectral \\
   & *++[F-,]\txt{ Representations \\ 
of pure inner forms of $G(K)$  \\ (fixed infi. char.) 
    }     \ar@{.}[rr]^{ \medskip \textrm{``Constructible"}}_{\medskip \textrm{Local Langlands}} && 
    *++[F-,]\txt{  Constructible sheaves on  \\
   Langlands parameters \\(fixed infi. char)}   
    & 
    \\ 
    % Third row
   & *++[F-,]\txt{Smooth representations of    \\ 
 stabilizers of $G$-isocrystals  } \ar@{.}[rr]^{\medskip \textrm{``Coherent"}}_{ \medskip \textrm{Local Langlands}} &&
    *++[F-,]\txt{Coherent sheaves on\\ Langlands parameters }
   &
 }
%\end{center} 
 }
 
 \medskip

Our main goal in this lecture is to describe a mechanism that interpolates between the coherent and constructible settings for the local Langlands program -- specifically, we propose that the latter appears as the generic point in a deformation of a full subcategory of the former. In particular we propose a form of the constructible categories which does vary well with infinitesimal character, and an operation which (conjecturally) cuts down the representation theory from all $G$-isocrystals to pure inner forms. 

\medskip

% we review the well-known connection between constructible sheaves and $\D$-modules provided by the  Riemann-Hilbert correspondence.  
For $X$ be a smooth complex variety, the algebraic de Rham theorem identifies the cohomology of the constant sheaf $H^*(X,\C_X) \simeq\Ext^*_{\Sh(X)}(\C_X,\C_X)$, i.e., the derived endomorphisms of the local system,  with de Rham cohomology $H^*(X, \Omega^\bullet, d)\simeq\Ext^*_{\D_X}(\cO_X,\cO_X)$, i.e., the derived endomorphisms of the $\D$-module $\cO_X$ equipped with its natural flat connection.  The Riemann-Hilbert correspondence generalizes this to an equivalence between the abelian category of local systems on $X$ and the abelian category of vector bundles with regular flat connections on $X$, and may be extended to an equivalence of categories between the constructible derived category $\Sh_c(X; \C)$ of $X$ and a full subcategory $\cD_{rh}(X)$ (consisting of regular holonomic objects) of the derived category $\cD(X)$ of $\D_X$-modules.

\medskip

This equivalence extends to the equivariant setting via descent, i.e., to smooth stacks $X$. Moreover in the situations of greatest interest in geometric representation theory, where in particular we are dealing with equivariant sheaves on schemes acted on with finitely many orbits, we can drop the modifiers ``regular holonomic" and identify the equivariant constructible derived category with the derived category $\cD_c(X)$ of all coherent strongly equivariant $\D$-modules.  Thus for example in the settings of Kazhdan-Lusztig or Springer theory as recounted above we can simply replace the use of constructible sheaves by $\D$-modules.

\medskip

$\D$-modules themselves can be described as deformation quantizations of quasicoherent sheaves on the cotangent bundle $\bT^*_X$. The sheaf $\D_X$ of differential operators is filtered by the order of operators, with associated graded $\gr(\D_X)\simeq \cO_{\bT^*_X}$ identified with symbols, i.e., functions on the cotangent bundle of $X$. As a result the sheaf $\D_X$ can be described as a deformation quantization of $\cO_{\bT^*_X}$ with its standard symplectic form: the Rees construction on $\D_X$ produces a $\Gm$-equivariant family of sheaves of algebras $\D_{\hbar,X}$ for $\hbar\in \BA^1/\G_m$, recovering $\D_X$ at $\hbar \ne 0$ and $\cO_{\bT^*_X}$ with Poisson bracket from the $\hbar$-linear term in the commutator at $\hbar=0$. Passing to modules we see the category of $\D_X$-modules as the fiber at $\hbar \ne 0$ of a $\Gm$-equivariant family of categories over $\BA^1$, degenerating to $\QCoh(\bT^*_X)^{\bG_m}$ at $\hbar=0$.

\medskip

The theory of cyclic homology due to Connes and Feigin-Tsygan, as seen through the eyes of derived algebraic geometry, provides a topological interpretation of the de Rham complex, whence of the category of $\D$-modules. Namely, consider the simplicial (or ``animated") circle $S^1= B\Z$ and the {\em derived loop space} 
$$\cL X=\Map(S^1,X)=X\times_{X\times X} X.$$ On a scheme, the Hochschild-Kostant-Rosenberg theorem identifies $\cLX$ with the spectrum of the algebra of differential forms (in negative cohomological degrees), while the de Rham differential has the pleasant interpretation as the linearization of the loop-rotation $S^1$-action on $\cL X$, and is encoded by the Connes $B$-operator of degree $-1$.  The Rees parameter $\hbar\in \BA^1$ gets interpreted (up to a cohomological shift) as the $S^1$-equivariant parameter $k[u]\simeq H^*_{S^1}(\pt)\simeq H^*(\BC\Pp^{\infty})$. Thus (up to cohomological shift and renormalization) the $\hbar$-deformation from $\QC(\bT^*_X)$ to $\cD$-modules on $X$ is identified with the $k[u]$-linear category of {\em cyclic sheaves} $\QC^!(\cL X)^{S^1}$ on $\cL X$: $S^1$-equivariant (ind-)coherent sheaves on the loop space, i.e. a categorification of the theory of cyclic homology. Passing to {\em periodic cyclic sheaves} -- inverting the equivariant parameter $u$ --  we recover the base-change $\cD_X\module \ot_k k[u,u\inv]$ of the category of $\cD_X$-modules.

\medskip
This cyclic perspective on $\D$-modules may be considered a formal manipulation in the setting of smooth schemes, but takes on a completely different flavor in the setting of smooth stacks $X$. The derived loop space $\cL X$ is identified with the (derived) {\em inertia stack} of $X$, i.e., the stack of pairs of points and automorphisms
$$\cL X\simeq \{x\in X, g\in \Aut(x)\},$$ 
which in the special case of finite orbit stacks is in fact an {\em ordinary} (underived) stack.  An extension of the Koszul duality pattern described above relates the category of $\D$-modules on a smooth stack $X$ to the circle action on the {\em formal completion} of the loop space (where we replace $\Aut(x)$ by its formal group). But the full category of coherent sheaves on $\cL X$ (with its circle action) gives a new and richer category, which enhances $\D$-modules on $X$ by adding new continuous parameters (the ``central characters", appearing as eigenvalues of the loop $g$). 

\medskip
The message of this lecture is that such loop spaces and their variants arise naturally as stacks of Langlands parameters, and their categories of coherent sheaves provide natural spectral sides for categorical forms of the local Langlands correspondence. Crucially, this description endows these categories with circle actions.
In the archimedean setting, the circle action appears from the geometric action of $S^1$ on the twistor line. In the unipotent non-archimedean setting, the circle action arises from interpolating between the trace of Frobenius and the trace of the identity (loop space) by varying $q$ (informally, it rotates the ``circle" $\Spec({\mathbb F}_q)$ with fundamental group generated by Frobenius). In general we only speculate about a general construction of circle actions reflecting a subtle aspect of the theory of traces of Frobenius. 

\medskip

The utility of the circle actions is that they provide the desired mechanism relating coherent sheaves to $\D$-modules on Langlands parameter spaces with fixed central characters, i.e., interpolating between coherent and constructible Langlands correspondences. Namely the category of cyclic (i.e., $S^1$-equivariant) sheaves forms a family over the equivariant parameter $u$. Specializing to $u=0$ we obtain a full subcategory of ($S^1$-invariant) coherent sheaves, the setting of the ``geometric-Langlands-type" local Langlands correspondence. On the other hand, for nonzero $u$ we obtain the categories of $\D$-modules that parametrize representations of pure inner forms via Vogan's local Langlands correspondence.

\subsection{Jordan decomposition}
The key mechanism relating cyclic sheaves with representation theory is the ``Jordan decomposition for loops", a stacky generalization of the Jordan decomposition for conjugacy classes in a reductive group which allows us to interpolate different central characters. Let us illustrate the most basic example. Consider the stack $\Gv/\Gv$ (which we recall is the inertia or derived loops $\Gv/\Gv=\cL(B\Gv)$ in the stack $B\Gv=\pt/\Gv$). Taking invariant polynomials (i.e., passing to the affinization) defines a map
$$\chi:\Gv/\Gv\longrightarrow \Gv\quot\Gv=\Spec(\cO(\Gv)^{\Gv})\simeq \Hv\quot W$$
to the space of continuous parameters. 

\medskip
If we interpret $\Gv\quot \Gv$ as semisimple conjugacy classes in $\Gv$, as in the constructible form of the local Langlands correspondence, we encounter an apparent discontinuity: 
picking a semisimple representative $\alpha\in (\Gv)^{ss}$ of a class $[\alpha]\in \Gv\quot \Gv$, the corresponding centralizers $\Gv(\alpha)$ depend discontinuously on the eigenvalues. Likewise the classifying stacks $B\Gv(\alpha)$ don't form a nice family over $\Hv\quot W$. 

\medskip
However the Jordan decomposition in $\Gv$ corrects this by adding commuting unipotents: the fiber $\chi\inv(\wh{[\alpha]})$ of the 
formal neighborhood of $[\alpha]$ is identified with $\Gv(\alpha)^u/\Gv(\alpha)$, the stack of unipotent conjugacy classes in the centralizer $\Gv(\alpha)$. 
This is a special case of a general construction, the {\em unipotent loop space}, a variant of the inertia defined for any stack in which we only allow {\em unipotent} automorphisms:
 $\cL^u(B\Gv(\alpha))=\Gv(\alpha)^u/\Gv(\alpha)$. 
In other words, if we augment $B\Gv(\alpha)$ to allow nontrivial unipotent classes in $\Gv(\alpha)$ (replacing it by its unipotent loops) then we obtain a good family, with total space the loops in $B\Gv$. 

\medskip
A crucial feature of unipotent automorphisms (or abstractly of unipotent loops) is that they carry a $\G_m$ action contracting them to the identity. This action allows us to relate coherent sheaves on the full fiber $\chi\inv(\wh{[\alpha]})=\Gv(\alpha)^u/\Gv(\alpha)$ with those on its formal completion. Finally Koszul duality relates cyclic sheaves on this formal completion with $\D$-modules on the stack $B\Gv(\alpha)$ itself. As a result the family of categories of $\D$-modules on $B\Gv(\alpha)$ obtain a uniform expression in terms of the circle action on the category of coherent sheaves on $\Gv/\Gv$. 

\medskip
We will see similar phenomena occurring in both the archimedean and non-\hspace{0ex}archimedean settings: by dropping the condition that the action of Frobenius or the archimedean Weil group be semisimple, we obtain versions of Langlands parameter spaces which smoothly interpolate different central characters, but still capture the same categories of constructible sheaves thanks to the application of contracting $\G_m$ actions combined with $S^1$-equivariant localization. Moreover the categories of coherent sheaves on these larger Langlands parameter spaces turn out to be the ones needed for the geometric-Langlands-flavored formulations of the local Langlands correspondence.

\subsection{Outline}
The rest of this paper is organized as follows: Section~\ref{loopspaces} is a brief and gentle introduction to derived loop spaces.  In Sections~\ref{unipotent} and \ref{real} we describe the archmidean and non-archimedean local Langlands correspondences, respectively, from the perspective of loop spaces and circle actions, pushing back to Section~\ref{koszul} an overview of the underlying technical mechanisms from derived algebraic geometry. In more detail:

\medskip

In Section~\ref{loopspaces} we review the notion of derived loop spaces and circle actions and give an overview of the general pattern from cyclic homology and derived algebraic geometry that we will apply, postponing all technical details to the Appendix.

\medskip
In Section~\ref{unipotent} we describe the role of loop spaces in the unipotent non-archimedean local Langlands correspondence. We first describe some recent results in the categorical local Langlands program, realizing categories of representations as full subcategories of coherent sheaves. We then describe results and conjectures concerning the role of circle actions in the unipotent local Langlands correspondence as interpolating coherent and constructible forms of the correspondence. We then briefly speculate on the appearance of circle actions in the general non-archimedean local Langlands correspondence. 

\medskip
In Section~\ref{real} we describe the role of loop spaces in the archimedean local Langlands correspondence. In particular we give a simple stacky description of the ABV geometric parameter spaces~\cite{ABV} and a corresponding statement of Soergel's conjecture, following~\cite{loops and reps}.
We then introduce the Langlands parameter stack associated to the twistor line $\twistor$, which smoothly interpolates between the ABV parameter spaces for varying infinitesimal character and extend Soergel's conjecture to this setting in terms of cyclic sheaves. On the automorphic side, we recover representations of pure inner forms from the stack of real parabolic bundles on the twistor line. We formulate the tamely ramified geometric Langlands conjecture for $\twistor$ (following~\cite{loops and parameters, betti}) and explain how it recovers Soergel's conjecture as its periodic cyclic deformation.

\medskip

In the Appendix, Section~\ref{koszul}, we describe the key technical mechanisms underlying the paper. We relate categories of $S^1$-equivariant coherent sheaves on loop spaces (among which we find stacks of Langlands parameters) to categories of $\D$-modules (in particular on Vogan varieties) in three steps. The first is the Jordan decomposition of loops, a general pattern relating loop spaces of quotient stacks to unipotent loop spaces. Next, the contracting $\Gm$-action on unipotent loop spaces relates their sheaf theory to that of formal loops. Finally, we apply the Koszul duality between $S^1$-equivariant coherent sheaves on formal loop spaces and $\D$-modules.

%and the associated contracting $\G_m$-action relating sheaves on unipotent and formal loops. 

%!TEX root = IHESwriteup.tex

\section{Equivariant Sheaves on Loop Spaces}\label{loopspaces}

%As we have seen above, the category of sheaves which appears when studying Langlands parameters in families is the derived category of coherent sheaves on certain derived loop spaces, while the category of sheaves which appears Langlands parameters at fixed semisimple parameters is the derived category of constructible sheaves on certain fixed point varieties.  In this section we discuss a precise way in which one can relate these two categories of sheaves.  In particular, we will describe a procedure by which we may:
%\begin{enumerate}
%\item Complete the coherent Springer sheaf at a semisimple Frobenius-parameter to obtain a filtered equivariant constructible sheaf, whose endomorphisms may be related to the graded Hecke algebra by recent work of Solleveld.\cite{sol1, sol2}
%\item Realize admissible objects in the derived category of modules for the graded Hecke algebra as sheaves on a category of Langlands parameters.
%\end{enumerate}

In this section, we summarize key statements about equivariant sheaves on loop spaces in derived algebraic geometry, the setting for the representation theoretic constructions in the following sections. We defer further details of the general theory to Section~\ref{koszul}.

\medskip
We will work with a broadly applicable set-up: over a fixed  field $k$ of characteristic zero with a smooth Artin stack $X$ with affine diagonal.  In all cases of present interest, we may assume $X = Y/G$ where $Y$ is a smooth quasi-projective scheme with the action of an affine algebraic group $G$.

\subsection{Variations on the theme of loops}\label{loop spaces}

%As before we fix $k$ to be a field of characteristic zero. Let $X$ be a (derived) geometric stack, i.e. a derived Artin 1-stack with affine diagonal.  

We view the topological group $S^1 = B\bZ = K(1, \bZ)$ as a locally constant group object in prestacks.  We may view it as the suspension of  $S^0 = \Spec k\coprod  \Spec k$:
$$S^1 \simeq \Sigma S^0 = \Spec k  \coprod_{\Spec k \coprod \Spec k} \Spec k.$$
This presentation of $S^1$ as a colimit leads to a presentation of the mapping stack out of $S^1$, i.e., 
the (derived) \emph{loop space},  as a limit
$$\cL X := \Map(S^1, X) = \Map( \Spec k  \coprod_{\Spec k \coprod \Spec k} \Spec k, X) = X \times_{X \times X} X.$$
Here the (derived) fiber product is the self-intersection of the diagonal  of $X$. 
Equivalently,  the loop space is the (derived) inertia stack with $k$-points $\cL X(k)$  given by pairs $(x, \gamma)$ where $x \in X(k)$ and $\gamma \in \Aut(x) = (\{x\} \times_X \{x\})(k)$ is an automorphism of the point $x$, modulo $\Aut(x)$-conjugacy.

\begin{exmp}\label{quotient stack exmp}
We consider the following examples.  \\
\indent(1) For $X$ a scheme, letting $\cB_X$ denote the sheafified cyclic bar complex on $X$ viewed as a sheaf of dg algebras under the shuffle product, we have $\cL X = \Spec_X \cB_X$. \\
\indent(2) For $X = BG$ a classifying stack, the diagonal is smooth, and in particular flat.  Thus, $\cL X$ is a underived stack, i.e. is the classical inertia stack $G/G$.\\
\indent(3) For $X = Y/G$ a global quotient stack, the loop space has the following presentation as a derived fiber product of stacks along representable maps:
$$\begin{tikzcd}
\cL(Y/G) \arrow[r] \arrow[d] & (Y \times G)/G \arrow[d, "{(a, p)}"] \\
Y/G \arrow[r, "\Delta"] & (Y \times Y)/G
\end{tikzcd}$$
where $G$ acts diagonally on $Y \times Y$ and $Y \times G$, $a$ is the action map, $p$ is the projection map, and $\Delta$ the diagonal.  This stack has $k$-points
$$\cL(Y/G)(k) = \{(y, g) \in Y(k) \times G(k) \mid g \cdot y = y\}/G(k).$$
\indent(4) For a smooth finite orbit stack $X$ (i.e. a stack with finitely many $k$-points), the derived loop space $\cL X$ is in fact an {\em ordinary} (underived) stack, and is identified with the classical inertia stack of $X$.  This absence of derived structure does not depend on flatness of any of the morphisms in the diagram of Example \ref{quotient stack exmp}, but rather a dimension count, i.e. if $X$ is a finite orbit stack then $\cL X$ is a union of irreducible components of (stacky) dimension 0, which is equal to the expected dimension.\footnote{Though this is well known, let us outline the argument for the uninitiated since it is fundamental.  First, any fiber product $X \times_Y Z$ may be written as the intersection of $X \times Z$ and $X \times Z$ (under two different mappings) in $X \times Y \times Z$, so we may assume our derived fiber product is a derived intersection $X \cap_Y Z$.  Locally, $X, Z$ are cut out by regular sequences since everything is smooth, and since $\codim(X \cap_Y Z) = \codim(X) + \codim(Z)$, the union of the sequence is regular as well, thus the higher $\Tor$'s vanish.}
\end{exmp}

\medskip

\subsubsection{Structures on loops} 
The loop space $\cL X = \Map(S^1, X)$ inherits many familiar algebraic structures from topology, of which we
will focus on the following: 

\begin{enumerate}

\item Restriction to a point $pt \in S^1$ gives {\em  projection to the base-point} $\cL X \to X$. 

\item Pullback along $S^1 \to pt$ gives {\em constant loops} $ X \subset \cL X$.

\item The $S^1$-action on itself by rotation gives  an $S^1$-action  on  $\cL X $ by {\em loop rotation} (see Section \ref{circle actions sec}).

\end{enumerate}

It will not play a role in our development, but it is also worth mentioning the the collapse map $ S^1 \to S^1 \wedge S^1$,  inverse map $ S^1 \to S^1 $, and their compatibilities give a {\em relative group structure on} $ \cL X \to X$.

One of our main goals is to explicate  the loop rotation action concretely. Among a coherent collection of structures, the first order part of such an action is an automorphism of the identity self-map of $\cL X$, i.e., a canonical automorphism of each point of $\cL X$.  In a discrete setting, this structure is concrete and familiar:

\begin{exmp}
For $X = BG$ with $G$ finite, we can identify the adjoint quotient $\cL X = G/G$ with the disjoint union $\coprod_{[\alpha]} BG(\alpha)$ of classifying stacks of centralizers $G(\alpha) \subset G$ of group elements $\alpha$  ranging over conjugacy classes $[\alpha]$.  For each component $BG(\alpha)$, its space of automorphisms is the homotopy quotient $\Aut(G(\alpha))/G$, where $G$ acts by composition with conjugation, and the identity automorphism has stabilizer $Z(G(\alpha))$.  The $S^1$-action on the component $BG(\alpha)$ is then given by the central element $\alpha\in Z(G(\alpha))$, considered as an automorphism of the identity map of $BG(\alpha)$, i.e.  a canonical automorphism of any $G(\alpha)$-bundle.
\end{exmp}
\medskip

\subsubsection{Odd tangent bundles}\label{odd tangent} For typical reasons, we may attempt to understand complicated non-linear objects via their linearizations.  The relevant linearization of the derived loop space is the \emph{odd tangent bundle}
$$\bT_X[-1] = \Spec_X \Sym_X \Omega^1_X$$
where $\Omega^1_X$ is the \emph{cotangent complex} of $X$, i.e., the derived Kahler differentials. The odd tangent bundle may be identified with the normal bundle of the constant loops $X \subset \cL X$. It is a derived vector bundle over $X$, and thus equipped with a natural contracting $\bG_m$-action.

%\medskip
For the linearization of $S^1=B\Z$ itself, we may take 
its affinization $B\bG_a = \Spec \cO(S^1)$, the functor on connective $k$-algebras represented by functions $\cO(S^1) \simeq C^\bullet(S^1; k)$ on $S^1$, i.e.,  $k$-valued cochains on $S^1$.\footnote{Note that since $\cO(S^1)$ is coconnective (and not connective), its ``spectrum'' is not an affine derived scheme but an \emph{affine stack} via  T\"{o}en (or \emph{coaffine stack} via Lurie).}  
The group $\bG_a$ has a canonical contracting action of $\bG_m$ via group homomorphisms inducing a contracting $\bG_m$-action on $B\bG_a$.  In more down-to-earth terms, $\cO(S^1) \simeq C^\bullet(S^1; k)$ is formal as a dg algebra, thus has a canonical internal weight grading agreeing with the cohomological grading.

We will also be interested in  the {\em formal odd tangent bundle} $\wh{\bT}_X[-1]$,
the formal completion of $\bT_X[-1] $ along its zero section. 
The $S^1$-action on $\cL X$ induces $S^1$-actions on $\bT_X[-1]$ and $\wh{\bT}_X[-1]$, and the latter factors through the  affinization map $S^1 \to B\bG_a$ \cite[Prop. 3.2.3]{koszul}. We may organize the  actions on $\wh{\bT}_X[-1]$ by a single action of the semidirect product $B\bG_a \rtimes \bG_m$.

\begin{exmp}
We continue with the examples from Example \ref{quotient stack exmp}.  \\
\indent(1) For $X$ a scheme, the Hochschild-Kostant-Rosenberg (or exponential) map defines an $S^1$-equivariant (factoring through $B\bG_a$-equivariant) equivalence (see Proposition 4.4 of \cite{loops and conns}):
$$\exp: \begin{tikzcd} \bT_X[- 1] \arrow[r, "\simeq"] & \cL X.\end{tikzcd}$$
In particular, when $X$ is a scheme, $\Omega^1_X[1]$ has negative Tor-amplitude, so $X \subset \bT_X[-1]$ is a nilthickening.  
Thus we have $\cL X \simeq \bT_X[-1] \simeq \wh{\bT}_X[-1] $. 
%and $\wh{\cL}X = \cL^uX = \cL X$ and all are identified via the exponential map.  
So for $X$ a scheme, the derived loop space $\cL X$ is no more complicated than its linearization.
% (or Lie algebroid). 
\\
\indent(2) For $X = BG$ a classifying stack, the odd tangent bundle $\bT_X[-1] = \mf{g}/G$ is the adjoint-equivariant  Lie algebra, and  the formal odd tangent bundle $\wh{\bT}_X[-1] = \widehat{\mathfrak g} /G$  is the adjoint-equivariant formal Lie algebra.
\end{exmp}

\medskip

\subsubsection{Formal and unipotent loops} We saw above that when $X$ is a scheme, the exponential map identified the loop space with the odd tangent bundle.  When $X$ is a stack, the exponential map is not even well-defined: for example, there is no map of algebraic stacks $\exp: \mf{g}/G \rightarrow G/G$.  

To write down an exponential map we need to restrict to the  \emph{formal loop space} $\wh{\cL}X$,  the completion of $\cL X$ along constant loops $X\subset \cL X$.  By Theorem 6.9 of \cite{loops and conns}, there is an $S^1$-equivariant equivalence between the \emph{formal} odd tangent bundle and \emph{formal} loop space:
$$\exp: \begin{tikzcd} \wh{\bT}_X[-1] \arrow[r, "\simeq"] & \wh{\cL} X.\end{tikzcd}$$
%Likewise, the $S^1$-action on $\cL X$ does not factor through $B\bG_a$.

\medskip

It is convenient to introduce another, intermediate, variant of the loop space $\cL X$, the \emph{unipotent loop space}:
$$\cL^u X := \Map(B\bG_a, X)$$
Note  $B\bG_a$-action on itself equips  $\cL^u X$ with a canonical $B\bG_a$-action.  The affinization map $S^1 \rightarrow B\bG_a$ induces an $S^1$-equivariant map $\cL^u X \rightarrow \cL X$; in our setting where $X$ is an Artin 1-stack with affine diagonal, this map is a monomorphism (see Proposition 2.1.24 in \cite{Ch}).  

Altogether, we have a sequence of $S^1$-equivariant monomorphisms:
$$\wh{\cL} X \hookrightarrow \cL^u X \hookrightarrow \cL X.$$

\begin{exmp}
We continue with the examples from Example \ref{quotient stack exmp}. \\
\indent(1) If $X$ is a scheme, then the inclusions above are all identities, i.e.
$$\wh{\cL} X = \cL^u X = \cL X$$
That is, the closed inclusion $X \hookrightarrow \cL X$ is a derived thickening. \\
\indent(2) When $X = BG$, the formal loop space is $\wh{\cL}(BG) = \wh{G}_e/G$, i.e. the adjoint-quotient of the formal group, and the unipotent loop space is $\cL^u(BG) = \wh{G}_{\cU}/G$, i.e. the adjoint-quotient of the formal completion along the unipotent cone. \\
\indent(3) The map $Y/G \rightarrow BG$ gives rise to a map $\cL(Y/G) \rightarrow G/G$.  By Propositions 2.1.20 and 2.1.25 of \cite{Ch}, the formal loop space is the completion along the inverse image of $\{e\}/G$, and the unipotent loop space is the completion along the inverse image of $\cU/G$.
\end{exmp}

We summarize the different loop spaces in the following picture:

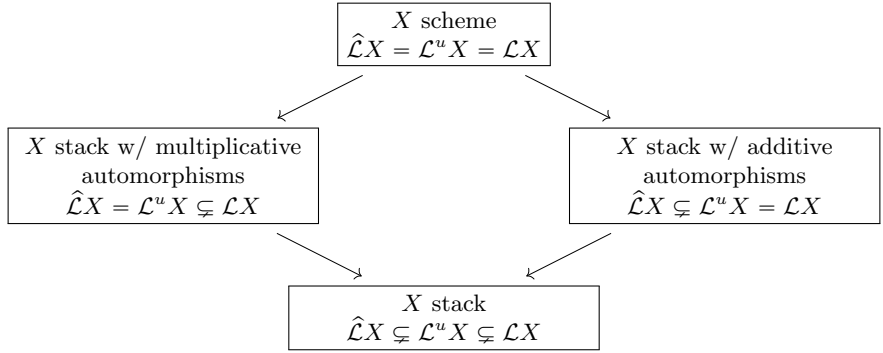
\small
\begin{equation}\label{loops picture}\begin{tikzcd}[column sep=-5ex]
& \fbox{\begin{minipage}{8em}\begin{center} $X$ scheme  \\ $\wh{\cL} X = \cL^u X = \cL X$
\end{center}\end{minipage}} \arrow[dr] \arrow[dl]  & \\
\fbox{\begin{minipage}{12em}\begin{center} $X$ stack w/ multiplicative automorphisms \\ $\wh{\cL} X = \cL^u X \subsetneq \cL X$ \end{center}\end{minipage}} \arrow[dr] & & \fbox{\begin{minipage}{12em}\begin{center} $X$ stack w/ additive automorphisms \\ $\wh{\cL} X \subsetneq \cL^u X = \cL X$ \end{center}\end{minipage}} \arrow[dl] \\    
&  \fbox{\begin{minipage}{12em}\begin{center}$X$ stack \\ $\wh{\cL} X \subsetneq \cL^u X \subsetneq \cL X$\end{center}\end{minipage}}   &
\end{tikzcd}
\end{equation}
\normalsize

%%%%%%

\subsection{Koszul dual description of equivariant sheaves}\label{equiv sheaves}

The main object we would like to study is the category of $S^1$-equivariant coherent sheaves on the loop space $\cL X$.
We defer all technical discussion until Section~\ref{koszul} and aim here simply to state the main results to be invoked. In particular, we refer to Section~\ref{circle actions sec} for the precise definition of 
the category
 (compactly renormalized) $S^1$-equivariant coherent sheaves 
$\IndCoh(\cL X)^{\omega S^1}$  we will work with.

\medskip

\subsubsection{Case of schemes}

Recall from Section \ref{odd tangent} that when $X$ is a scheme, 
the Hochschild-Kostant-Rosenberg (or exponential) map defines an  equivalence
$$\exp: \begin{tikzcd} \bT_X[- 1] \arrow[r, "\simeq"] & \cL X.\end{tikzcd}$$
compatible with the factorization of the $S^1$-action through the $B\bG_a$-action.
Thus we have an equivalence
 $$\begin{tikzcd} \IndCoh(\cL X)^{\omega S^1} \arrow[r, "\simeq"] & \IndCoh(\bT_X[-1])^{\omega B\bG_a}\end{tikzcd}$$ 
 
Recall as well the $B\bG_a$-action on
the right-hand side lifts to a  $B\bG_a\rtimes \bG_m$-action, so  we may likewise consider the graded enhancement $\IndCoh(\bT_X[-1])^{\omega B\bG_a \rtimes \bG_m}$.
Kapranov \cite{kapranov} proved the following version of Koszul duality, which was interpreted in terms of loop spaces in Corollary 5.2 of \cite{loops and conns}.

\begin{thm}[Koszul duality for loop spaces of schemes]
Let $X$ be a smooth scheme, $\cD(X)$ the category of 
 \emph{$\cD$-modules} on $X$,
and
 $F\cD(X)$ the category of \emph{filtered $\cD$-modules} on $X$, i.e., modules for the Rees algebra attached to the order filtration on differential operators.
  
 Then, we have  compatible equivalences:
$$\begin{tikzcd}
{\IndCoh}(\bT_X[-1])^{\Tate \rtimes \bG_m}  \arrow[d, "\simeq"'] & {\IndCoh}(\bT_X[-1])^{\omega B\bG_a \rtimes \bG_m} \arrow[r] \arrow[d, "\simeq"'] \arrow[l] & {\IndCoh}(\bT_X[-1])^{\G_m} \arrow[d, "\simeq"'] \\
{\cD}(X) & F{\cD}(X)  \arrow[r] \arrow[l] & \IndCoh(\mathbb{T}_{X}^*)^{\G_m} 
\end{tikzcd}$$

\end{thm}

\medskip

\subsubsection{Case of stacks} When $X$ is a stack,  Koszul duality alone is only sufficient to describe 
$S^1$-equivariant coherent sheaves on the formal loop space $ \wh \cL X$ or  the unipotent loop space $\cL^u X$. (Recall when $X$ is a scheme, the inclusions are equivalences $\wh \cL X \simeq  \cL^u X \simeq \cL X$.)  The following Koszul duality is a theorem of  \cite{koszul}; see Section~\ref{sec ren} for the  ``renormalized" categories involved. In  particular, the category $$\wh{\IndCoh}(\wh \cL X) = \Ind(\wh{\Coh}(\wh{\cL} X))$$ is a  suitably renormalized version of ind-coherent sheaves. 
\begin{thm}[Koszul duality for loop spaces of stacks]\label{kd stacks thm}
Let $X$ be a smooth geometric stack,
$\breve\cD(X)$ the suitably renormalized category of 
 \emph{$\cD$-modules} on $X$,
and
 $F\cD(X)$ the suitably renormalized category of \emph{filtered $\cD$-modules} on $X$.

   We have compatible equivalences:
$$\begin{tikzcd}
\wh{\IndCoh}(\wh \cL X)^{\Tate \rtimes \bG_m}  \arrow[d, "\simeq"'] & \wh{\IndCoh}(\wh \cL X)^{\omega B\bG_a \rtimes \bG_m} \arrow[r] \arrow[d, "\simeq"'] \arrow[l] & \wh{\IndCoh}(\wh \cL X)^{\G_m} \arrow[d, "\simeq"'] \\
\breve{\cD}(X) & F\breve{\cD}(X)  \arrow[r] \arrow[l] & \IndCoh(\mathbb{T}_{X}^*)^{\G_m} 
\end{tikzcd}$$
  In particular, we have a Koszul duality (in the sense of Section~\ref{graded lifts}):
$$\wh{\IndCoh}(\wh \cL X)^{\Tate} \kos \breve{\cD}(X) \otimes \QCoh(\bG_m).$$
%In the non-graded setting, we have compatible equivalences:
%$$\begin{tikzcd}
%\wh{\IndCoh}(\wh{\bT}_X[-1])^{\Tate} \arrow[d, "\simeq"'] & \wh{\IndCoh}(\wh{\bT}_X[-1])^{\omega B\bG_a} \arrow[r] \arrow[d, "\simeq"'] \arrow[l] & \wh{\IndCoh}(\wh{\bT}_X[-1]) \arrow[d,"\simeq"'] \\
%F\breve{\cD}(X) \otimes_k k((u)) & F\breve{\cD}^{[2]}(X)  \arrow[r] \arrow[l] & \QCoh(\bT_{X}^*[2]).
%\end{tikzcd}$$
%functorial with respect to smooth pullback and proper pushforward.
\end{thm}

Going further, we do not have to pass to formal loops  $\wh \cL X$, but can work more directly with unipotent loops $\cL^u X$. This is justified by the following theorem to appear in \cite{BCHN2}; see Example \ref{exmp BGa} for a worked example.

\begin{thm}\label{formal is unipotent}
Let $X$ be a geometric stack.  The pullback functor induces an equivalence
$$
%\IndCoh(\cL^uX)^{\Tate} \simeq \IndCoh(\wh{\cL}X)^{\Tate}, \;\;\;\;\;\;\; 
\wh{\IndCoh}(\cL^uX)^{\Tate  \rtimes \bG_m} \simeq \wh{\IndCoh}(\wh{\cL}X)^{\Tate \rtimes \bG_m}$$
and hence we have an equivalence
$$
\wh{\IndCoh}(\cL^uX)^{\Tate  \rtimes \bG_m} \simeq \breve\cD(X)$$
\end{thm}

%%%%%%

\subsection{Jordan decomposition and equivariant localization}
With the preceding theory for $S^1$-equivariant coherent sheaves on the unipotent loop space $\cL^u X$ in hand, we would like to go further and describe $S^1$-equivariant coherent sheaves on the entire loop space $\cL X$. Our approach will involve factoring loops into semisimple and unipotent parts and then applying the prior theory to the unipotent part as twisted by the semisimple part. 
%
%In this section, we specialize to the case of global quotient stacks $X/G$, where $G$ is a reductive group $G$ (or more generally, an affine algebraic group $G$ with reductive neutral component).  See Section 3 of \cite{Ch} for an extended discussion.   

\medskip

\subsubsection{Jordan decomposition of loops}

In this section, we specialize to the case of global quotient stacks $X/G$, where $G$ is a reductive group $G$ (or more generally, an affine algebraic group $G$ with reductive neutral component).  See Section 3 of \cite{Ch} for an extended discussion.   
%Recall that the exponential map
%$$\exp: \wh{\bT}_{X/G}[-1] \rightarrow \wh{\cL}(X/G)$$
%is an equivalence on \emph{formal completions}.  Thus, Koszul duality gives us a way to understand coherent sheaves on formal loop spaces $\wh{\cL}(X/G)$, but not the full loop space $\cL(X/G)$.  To access the rest of $\cL(X/G)$, we describe an equivariant localization pattern which describes formal or unipotent loops on $\cL(X/G)$ over a point in $G//G$.  

Consider the ``characteristic polynomial" map 
$$\mu: \LL(BG)=G/G\longrightarrow G//G$$   from the adoint-quotient stack  to the affine quotient scheme, i.e., to the variety parametrizing semisimple conjugacy classes.  Note the pre-image of the class of the identity is precisely
the unipotent cone $\cU_G \subset G$.
%The closed points in $G//G$ parameterize semisimple conjugacy classes in $G$, i.e., closed adjoint $G$-orbits in $G$.

For any semisimple $\alpha \in G$, set $[\alpha] := \mu(\alpha) \in G//G $ and write $\bO_\alpha = G \cdot \{\alpha\} \simeq G/G(\alpha) \subset G$ for the conjugacy class of $\alpha$, where $G(\alpha) \subset G$ is the centralizer of $\alpha$. Then we have the Jordan decomposition of group elements viewed as loops
$$\mu^{-1}([\alpha]) \simeq G \times^{G(\alpha)} \{su \in G \mid s \in \bO_\alpha,  u \in \cU_G, su = us\}$$

Now for a scheme $X$ with a $G$-action, by taking fibers of  the natural maps 
$$\LL(X/G)\longrightarrow \LL(pt/G) \simeq G/G\longrightarrow G//G$$ 
we can  define  variants of loop space. In particular,  taking completions over points in $G//G$ gives rise to generalizations of unipotent loop spaces, while taking completions over closed orbits of $G/G$ gives rise to generalizations of formal loop spaces.

\begin{defn}\label{def loop spaces}
Let $G$ be a linear algebraic group with reductive neutral component acting on a prestack $X$.  Let $\bO \subset G$ denote a semisimple adjoint orbit, defining a closed substack $\bO/G \subset G/G \simeq \cL(BG)$, and denote by $[\alpha] =: \mu(\bO) \in G//G$ its class in the affine quotient.  

We define the following loop spaces:
\begin{enumerate}
\item  The \emph{$\alpha$-unipotent loop space} $\LL_\alpha^u(X/G)$ is the completion of $\LL(X/G)$ along the inverse image of $[\alpha] \in G//G$, or equivalently along the saturation\footnote{I.e. the maximal closed substack of $G/G$ containing $\bO$ as its unique closed orbit.} of the orbit $\bO/G$.
\item The \emph{$\alpha$-formal loop space} $\LLf_\alpha(X/G)$ is the completion of $\LL(X/G)$ along the inverse image of the orbit $\bO/G$.
\item The \emph{$\alpha$-specialized loop space} $\LL'_\alpha(X/G) = \cL(X/G) \times_{\cL(BG)} \bO/G$ is the (derived) fiber of $\LL(X/G)$ over $\bO/G$.
\end{enumerate}

Though the notation suggests otherwise, the above constructions only depend on the orbit $\bO$ (equivalently, $[\alpha] \in G//G$) and not a choice of representative (or lift) $\alpha \in \bO$.  If we make such a choice of representative, then we obtain an equivalence $\bO/G \simeq BG(\alpha)$ with the classifying stack of the centralizer of $\alpha$, and we may define the following (which depends on $\alpha$):
\begin{enumerate}
\item[(4)] The \emph{$\alpha$-twisted loop space}\footnote{Also known as the derived fixed points for the self-map of $X$ defined by the action of $\alpha$ on $X$.  In particular, the $\alpha$-twisted loop space on $X$ only requires this self-map as an input and not the entire $G$-action on $X$.} $\cL_\alpha(X)$ is the (derived) fiber of $\cL(X/G)$ along $\Spec k = \{\alpha\} \rightarrow BG(\alpha) \simeq \bO/G \subset G/G$.
\end{enumerate}
\end{defn}

The $S^1$-action on $\cL (X/G)$ restricts to $S^1$-equivariant inclusions 
$$\cL'_\alpha(X/G) \subset \LLf_\alpha(X/G) \subset \cL^u_\alpha(X/G) \subset \cL(X/G).
$$ The twisted loop space $\cL_\alpha(X)$ does not admit an $S^1$-action, but there is a map $\cL_\alpha(X) \rightarrow \cL'_\alpha(X/G)$.

\medskip

\subsubsection{Sheaves via equivariant localization for loop spaces}\label{shvs via eq loc}

We  state here the equivariant localization theorem of~\cite{Ch} describing loops in the quotient stack $X/G$ with given semisimple part $\alpha$ in terms of unipotent loops on the quotient stack $X(\alpha)/G(\alpha)$.  Here 
$X(\alpha)$ denotes the homotopy fixed-points $X^A$ with respect to the closed subgroup $A = \langle \alpha \rangle \subset G$ generated by $\alpha$. We refer to Section~\ref{sec homotopy fixed} below for a more detailed discussion.

Consider the natural map 
$$\ell_\alpha: \cL_\alpha(X(\alpha)/G(\alpha)) \rightarrow \cL_\alpha(X/G)$$
 obtained by applying $\cL_\alpha$ to the natural map $X(\alpha)/G(\alpha) \hookrightarrow X/G(\alpha) \rightarrow X/G$.  
% While the $\alpha$-loop spaces do not depend on the choice of $\alpha \in \bO$, the localization map $\ell_\alpha$ does.

\begin{thm}[Equivariant localization for derived loop spaces]\label{LocThm}
Let $G$ be a linear algebraic group with reductive neutral component acting on a locally Noetherian derived scheme $X$, and $\alpha \in G$ a semisimple element.  Then, the unipotent $\alpha$-localization map $$\ell_\alpha^u: \LL_\alpha^u(X(\alpha)/G(\alpha)) \rightarrow \LL_\alpha^u(X/G)$$ is an $S^1$-equivariant equivalence.  The same is true for the corresponding localization maps on formal and specialized loops
$$\hat{\ell}_\alpha: \LLf_\alpha(X(\alpha)/G(\alpha)) \rightarrow \LLf_\alpha(X/G), \hskip.3in
 \ell'_\alpha: \LL_\alpha'(X(\alpha)/G(\alpha)) \rightarrow \LL_\alpha'(X/G).$$
\end{thm}
\begin{proof}
When $X$ is smooth, the statement is Theorem A in \cite{Ch}.  The general case follows from the following: by the reduction in \emph{loc. cit.} we may assume that $\alpha$ is central, and let $A = \langle \alpha \rangle$.  We first claim that every derived scheme may locally be written as an interated fiber product of smooth affine $A$-schemes.  Since $A$ has neutral component a multiplicative torus, every $A$-scheme has a Zariski cover by $A$-closed opens, so we may assume that $X$ is affine.  Since $X$ is locally Noetherian, we may choose a $A$-equivariant semi-free resolution of $\cO(X)$, i.e. a finite-dimensional graded $A$-subrepresentation $V_0^* \subset \cO(X)$ which generates as a dg algebra, and likewise $V_1^* \subset \cO(V_0)$ generating relations, $V_2 \subset \cO(V_0 \times V_1)$ generating relations between relations, et cetera.  These assemble into a diagram of smooth affine $A$-schemes where the vertical arrows are zero-inclusions and the horizontal arrows are defined by the differentials in the semi-free resolution:
$$\begin{tikzcd}
 & V_0 \arrow[d] & V_0 \times V_1 \arrow[d] & \cdots \arrow[d] \\
V_0 \arrow[r] & V_0 \times V_1 \arrow[r] & V_0 \times V_1 \times V_2 \arrow[r]  & \dots \\
\end{tikzcd}$$
The limit of this diagram is $X$, and since loop spaces and homotopy fixed points commute with limits we may deduce the theorem from the smooth case.
\end{proof}

\medskip

%\subsubsection{Sheaves via equivariant localization}

%
%We note that if $\alpha$ is in the neutral component of $G(\alpha)$, then every $\cD$-module on $X(\alpha)/G(\alpha)$ is $\alpha$-trivial, and that this condition is automatically satisfied for a semisimple element in a reductive group.  

Finally, combining equivariant localization with Koszul duality, we have the following. As above, we refer to Section~\ref{sec ren} for details of the various renormalized categories, and Sections \ref{sec central twist} and \ref{sec triv block} for the notion of twisted $S^1$-actions and $\alpha$-trivial blocks.

\begin{thm}
Let $\alpha \in G$ be a semisimple element, and $X/G$ a global quotient stack.   We have compatible equivalences
$$\begin{tikzcd}[column sep=small]
\wh{\IndCoh}(\wh{\cL}_\alpha (X/G))^{\Tate \rtimes \bG_m}_\alpha \arrow[d, "\simeq"'] & \wh{\IndCoh}(\wh{\cL}_\alpha (X/G))^{\omega B\bG_a \rtimes \bG_m}_\alpha \arrow[r] \arrow[d, "\simeq"'] \arrow[l] & \wh{\IndCoh}(\wh{\cL}_\alpha (X/G))_\alpha^{\bG_m} \arrow[d,"\simeq"'] \\
\breve{\cD}(X(\alpha)/G(\alpha))_\alpha & F\breve{\cD}(X(\alpha)/G(\alpha))_\alpha  \arrow[r] \arrow[l] & \QCoh(\bT_{X(\alpha)/G(\alpha)}^*)_\alpha
\end{tikzcd}$$
functorial with respect to smooth pullback and proper pushforward, as well as graded versions.  

In particular, passing through Theorem \ref{formal is unipotent} we have a Koszul duality:
\begin{equation}\label{kd equation}
\wh{\IndCoh}(\cL^u_\alpha(X/G))^{\Tate}_\alpha \simeq \wh{\IndCoh}(\wh{\cL} (X/G))^{\Tate}_\alpha \kos \breve{\cD}(X(\alpha)/G(\alpha))_\alpha.
\end{equation}
\end{thm}

\medskip

\subsubsection{Twisted Springer sheaves and graded lifts}

We now consider the following situation, and example of the theory we have developed so far.  Let $\mu: \wt{X}/G \rightarrow X/G$ be a proper map of smooth stacks, and let $\cS = \cL\mu_* \cO_{\cL(\wt{X}/G)} \in \Coh(\cL(X/G)$.  In Section \ref{unipotent} this will be the \emph{coherent Springer sheaf}. Letting $\hat{\ell}_\alpha$ denote the localization map, we let 
$$\cS(\alpha) := \hat{\ell}_\alpha^{\,!} \cS \in \wh{\Coh}(\wh{\cL}_\alpha(X(\alpha)/G(\alpha)) \subset \IndCoh(\wh{\cL}_\alpha(X(\alpha)/G(\alpha))$$
denote its $!$-restriction to $\alpha$-formal loops, which is equipped with
\begin{enumerate}[leftmargin=5ex]
\item a canonical $S^1$-equivariant structure coming from the canonical $S^1$-equivariant structure on the dualizing sheaf $\cO_{\cL(\wt{X}/G)} \simeq \omega_{\cL(\wt{X}/G)}$ (see Section \ref{circle actions sec}),
\item a canonical $\alpha$-trivialization, since $\alpha$ acts trivially on the structure sheaf, and therefore a factorization of the $S^1$-equivariant structure through a $B\bG_a$-\hspace{0ex}equivariant structure (see Section \ref{sec triv block}),
\item a canonical graded lift $\wt{\cS}(\alpha)$ coming from the graded lift on $\omega_{\wh{\cL}(\wt{X}(\alpha)/G(\alpha))}$, which comes from the identification of $\wh{\cL}_\alpha(\wt{X}(\alpha)/G(\alpha))$ as an odd tangent bundle (see Section \ref{graded lifts}).\footnote{Though the graded lifts are canonical, they cannot be made globally, only parameter-by-parameter.}
\end{enumerate}

\medskip

Koszul dually, we consider the induced map $\mu^\alpha: \wt{X}(\alpha)/G(\alpha) \rightarrow X(\alpha)/G(\alpha)$.  The object corresponding to the $S^1$-equivariant sheaf $\wt{\cS}(\alpha)$ is the filtered $\cD$-module $\wt{\mathbf{S}}(\alpha) \simeq \mu^\alpha_* \cO_{\wt{X}(\alpha)/G(\alpha)}$ where $\cO_{\wt{X}(\alpha)/G(\alpha)}$ is equipped with its canonical filtration.  We denote the corresponding unfiltered $\cD$-module by  $\mathbf{S}(\alpha)$.  
%We consider the following composition:
%$$
%\wh{\iota}_\alpha: \begin{tikzcd}
%\wh{\cL}(X(\alpha)/G(\alpha)) \arrow[r, "sh_\alpha"] &  \wh{\cL}_\alpha(X(\alpha)/G(\alpha)) \arrow[r, "\wh{\ell}_\alpha"] & \wh{\LL}_\alpha(X/G) \arrow[r] & \cL(X/G)
%\end{tikzcd}$$
These objects have (derived) endomomorphism algebras:
%$$\cH = \End_{\wh{\cL}(X/G)}(\cS)^{S^1}, 
$$\cH(\alpha) = \End_{\wh{\cL}_\alpha(X/G)}(\cS(\alpha))^{S^1}, \;\;\;\;\;\;\; \mathbf{H}(\alpha) = \End_{X(\alpha)/G(\alpha)}(\mathbf{S}(\alpha))$$
and likewise $\wt{\cH}(\alpha), \wt{\mathbf{H}}(\alpha)$ for the corresponding graded lifts.  
%For example, taking $\mu: \wt{\cN}/G \times \bG_m \rightarrow \mf{g}/G \times \bG_m$ to be the adjoint and scaling equivariant Springer resolution, we have by Theorem \ref{BCHN theorem} that $\cH$ is the affine Hecke algebra tensored with $k[u]$. On the other hand, $\mathbf{H}(\alpha)$ is the graded Hecke algebra (where the internal grading is the same as the cohomological grading).  
We have the following immediate corollary of the Koszul duality equivalences.
\begin{cor}\label{aha gha}
View the sheaves $\cS(\alpha), \wt{\cS}(\alpha)$ in the $S^1$-equivariant categories.  We have an equivalence of graded algebras $\wt{\cH}(\alpha)^\shear \simeq \wt{\mathbf{H}}(\alpha)$ and of $\bG_m$-equivariant categories:
$$\langle \wt{\cS}(\alpha) \rangle^{\bG_m, \shear} \simeq \Mod^{\bG_m}(\wt{\cH}(\alpha))^{\shear} \simeq \Mod^{\bG_m}(\wt{\mathbf{H}}(\alpha)) \simeq \langle \wt{\mathbf{S}}(\alpha) \rangle^{\bG_m}$$
and therefore a Koszul duality:
$$\langle \cS(\alpha) \rangle \simeq \Mod(\cH(\alpha)) \kos \Mod(\mathbf{H}(\alpha)) \simeq  \langle {\mathbf{S}}(\alpha) \rangle.$$
\end{cor}

\section{Nonarchimedean Local Langlands and Circle Actions}\label{unipotent}

\subsection{Unipotent representations}\label{unipotent reps}

Fix a non-archimedean local field $F$ with residue field $\bF_q$, and a connected split reductive group $G$ over $F$ (though much of the discussion carries over to unramified groups).  Lusztig~\cite{lusztig unipotent} and Solleveld~\cite{soluni}, \cite{soluni2} established a Langlands correspondence for the unipotent representations of such groups, as well as their pure inner forms.  (Recall that a representation $\pi$ of an inner form $G^*$ of $G$ is called unipotent if there exists a parahoric subgroup $P$ of $G$, and a unipotent representation $\tau$ of the $\bF_q$-points of the reductive quotient of $P$, such that the restriction of $\pi$ to $P$ contains the representation $\tau$.)  The unipotent representations coming from a given $G^*$, $P$, and $\pi$ are in bijection with the irreducible representations of the Hecke algebra $\cH_{G^*,P,\pi}$, which is a so-called ``affine Hecke algebra with unequal parameters''.

Lusztig's construction yields a bijection of pairs $(G^*,\pi)$, where $G^*$ is a pure inner form of $G$ and $\pi$ is a unipotent representation, with the set of {\em unipotent extended Langlands parameters} for $G$; that is, the set of $\Gv$-conjugacy classes of triples $(\sigma,n,\chi)$, where $\sigma$ is a semisimple element in $\Gv$, $n \in \cN$ such that $\sigma \cdot n = qn$, and $\chi$ is an irreducible representation of the component group of the $\Gv$-centralizer $\Gv^{\sigma,n}$ of $(\sigma,n)$.

If we fix a semisimple $\sigma$ in $\Gv$, we can define the Vogan variety $\cN^{(\sigma,q)}/\Gv^{\sigma}$ to be the space of $n \in \CN$ such that $\sigma \cdot n = qn$; then a triple $(\sigma,n,\chi)$ determines a local system on the orbit of $n$ in $\cN^{(\sigma,q)}/\Gv^{\sigma}$, and thus a perverse sheaf on $\cN^{(\sigma,q)}/\Gv^{\sigma}$.  Lusztig's argument relates such perverse sheaves to representations of Hecke algebras by identifying the completions of each Hecke algebra $\cH_{G^*,P,\pi}$ at maximal ideals of their centers with the derived endomorphisms of certain perverse sheaves on Vogan varieties $\cN^{\sigma,q}/\Gv^{\sigma}.$

We make this construction explicit in the case of the principal block of $G$; that is, the full subcategory of representations of $G$ generated by their Iwahori fixed vectors.  The irreducible such representations are precisely the irreducible representations of the affine Hecke algebra $\cH_q := \cH_{G, B, \mathrm{triv}}$. The maximal ideals of the center of this Hecke algebra are in bijection with semisimple conjugacy classes in $\Gv$.

Assuming $q$ is not a root of unity, Lusztig's construction realizes completions of the affine Hecke algebra $\cH_q(\sigma)$ at the maximal ideal corresponding to $\sigma \in \Gv$ as Koszul equivalent to the derived endomorphisms of a certain constructible sheaf on $\cN^{(\sigma,q)}/\Gv^{\sigma}$.\footnote{In particular, the former is in cohomological degree 0, and the latter has generators in cohomological degree 2.  See Section \ref{sec graded hecke} and Definition \ref{koszul duality def} for a discussion.}  Concretely, we can define the $\sigma$-Springer sheaf $\bS_q(\sigma)$ to be the pushforward $\mu^{(\sigma,q)}_* \bC_{\wt{\cN}^{(\sigma,q)}}$ along the ``$\sigma$-fixed'' Springer resolution:
$$\mu^{(\sigma,q)}: \wt{\cN}^{(\sigma, q)}/\Gv^\sigma \rightarrow \cN^{(\sigma, q)}/\Gv^\sigma.$$
The BBD decomposition theorem implies a direct sum decomposition
$$\bS_q(\sigma) := \bigoplus_{(\cO, \cL)} E_{\cO,\cL} \otimes \mathbf{IC}(\cO, \cL)[d_{\cO, \cL}].$$
The intersection cohomology sheaves (extended from such local systems) which appear in the $\sigma$-Springer sheaf $\bS_q(\sigma)$ are the principal series representations that appear in the given L-packet.  Furthermore, we have that $\End(\bS_q(\sigma))$ is an algebra whose degree 0 part is the specialization $\cH_{q, [\sigma]}$ of the affine Hecke algebra at the central character given by $[\sigma] \in \Gv//\Gv$, and thus $\cH_{q, [\sigma]}$ acts on simple summands of $\bS_q(\sigma)$.  The Deligne-Langlands correspondence is given by this decomposition, i.e. $E_{\cO, \cL}$ are the Iwahori-invariants of the $G(F)$-representation corresponding to the Langlands parameter given by the $(\cO, \cL)$.

\medskip

\subsection{The stack of unipotent Langlands parameters}

\medskip

An ``algebraic'' or ``families'' perspective on the unramified principal series was developed in~\cite{BCHN}. Namely we consider the stack of \emph{unipotent Langlands parameters} which may be defined in the following equivalent ways.  We fix $q$ to be a prime power, in particular a nonzero integer, morally the order of the residue field of $F$.

\medskip

(1) Explicitly, it is the stack
$$\bL^u_{F,G} = \bL^u_{q,G} := \{n \in \cN_{\Gv}, g\in \Gv \mid  \mathrm{Ad}(g)\cdot n = qn \}/\Gv$$ 
Note that we have dropped the condition on Deligne-Langlands parameter that the image of Frobenius $g$ be semisimple.

\medskip

(2) It is the $q$-twisted loop space\footnote{The first identification uses the fact that $q$ is not a root of unity.}
$$\cL_q(\wh{\cN}_{\Gv}/\Gv) = \cL_q(\check{\mf{g}}/\Gv) = \cL(\check{\mf{g}}/(\Gv \times \bG_m)) \times_{\cL(B\bG_m)} \{q\}.$$
Via the discussion in Definition \ref{def loop spaces}, there is no $S^1$-rotation on the $q$-twisted loop space.  However, there is an $S^1$-action coming from the action on $\wh{\cN}_{\Gv}/\Gv = \cL^u(B\Gv)$, which we do not consider.
Equivalently, we may write $\cL_q(\cN_{\Gv}/\Gv) = \cL_q(\check{U}/\Gv)$, where $q$ acts by exponentiation.

\medskip

(3) It is a substack of the stack of of $\Gv$-local systems $\Loc_{\Gv}(T_q)$ on the $q$-twisted topological torus $T_q$ where we require the monodromy around the non-twisted loop to be unipotent, where $T_q$ is the torus obtained by gluing the two boundaries of the cylinder $S^1 \times [0, 1]$ along the degree $q$ map.  The space $T_q$ has an $S^1$-action in the meridian direction (by ``speeding up the loop'') corresponding to the tame monodromy (which we do not consider), but not in the longitudinal (Frobenius) direction.  Equivalently, we can define $T_q^u$ as the gluing of $B\bG_a \times [0, 1]$ along the scaling by $q$ map, so that $\bL^u_{q,G} = \Loc_{\Gv}(T_q^u)$ (here, $q$ is not required to be an integer).

%compare with Section \ref{sec twistor spectral}?

\medskip

There is a natural map $\mu: \bL^u_{q,B} \rightarrow \bL^u_{q,G}$ and we define the \emph{coherent Springer sheaf} by
$$\cS_q := \mu_* \cO_{\bL^u_{q,B}} \in \Coh(\bL^u_{q,G}).$$
We have the following main theorem of \cite{BCHN}.
\begin{thm}\label{BCHN theorem}
The affine Hecke algebra is naturally isomorphic to the $\Ext$ algebra of the coherent Springer sheaf $\cS_q$, which vanishes in non-zero cohomological degrees.  Therefore we have a full embedding 
$$\begin{tikzcd}[column sep=large]
\cH_q\module \arrow[hook, r, "- \otimes_{\cH_q} \cS_q"] &  \langle \cS_q \rangle \arrow[r, hook] & \QC^!(\bL_{q,G}^u)\end{tikzcd}$$ 
from the unramified principal series of $G(F)$ to ind-coherent sheaves on the stack of unipotent Langlands parameters.
\end{thm}

It is thus a natural question to relate our families, or coherent, or algebraic realization of unipotent principal series representations to the specialized, or constructible, or topological realization at parameters given by $\sigma$-Springer theory described above.

\medskip

In the case $G=GL_n$, the Bushnell-Kutzko theory of types, combined with the local Langlands theorem of Harris-Taylor and Henniart, allows us to reduce the entire smooth representation theory of $G$ to the unramified principal series. As a result we deduce that the entire category of smooth representations of $GL_n$ embeds into ind-coherent sheaves on the stack of all Langlands parameters $\bL_{F,GL_n}$, as constructed in~\cite{curtis,DHKM}:

\medskip

\begin{thm}[\cite{BCHN}]\label{BCHN GLn theorem}
There is a full embedding $$D(GL_n(F))\hookrightarrow \QC^!(\bL_{F,GL_n})$$ from the derived category of smooth representations of $GL_n(F)$ into ind-coherent sheaves on the stack of Langlands parameters. The embedding is characterized by taking irreducible cuspidals to skyscrapers at the corresponding Langlands parameters and compatibility with parabolic induction.
\end{thm}

\medskip

\subsection{Loop spaces and cyclic deformation}

An intriguing feature of the stack $\bL^u_q$ of unipotent Langlands parameters is its proximity to the derived loop space of the equivariant nilpotent cone,
$$\cL(\check{\cN}/\Gv\times \Gm)\simeq  \{n \in \check \cN, g\in \Gv, q\in \Gm\; : \; g\cdot n = qn \}/\Gv\times \Gm,$$
and Theorem~\ref{BCHN theorem} has a variant defined over the entire loop space and recovering the affine Hecke algebra with $q\in \Gm$ as a parameter. More significantly, while the stack $\bL^u_q$ -- the fiber over $\{q\}\to \Gm/\Gm$ -- does not inherit the circle action, the stack of {\em graded unipotent Langlands parameters}
$$\bL^u_{\ul{q}}:=\bL^u_q/\Gm$$
--  the fiber over $\{q\}/\Gm\in \Gm/\Gm$ -- does. Moreover Theorem~\ref{BCHN theorem} carries over unchanged over $\bL^u_{\ul{q}}$, and in fact has an $S^1$-equivariant enhancement in which the coherent Springer block deforms {\em trivially} over the equivariant parameter $k[u]=H^*(BS^1)$.

\medskip

\begin{thm}[\cite{BCHN}]\label{cyclic BCHN theorem}
We have a full embedding $$\cH_q[u]\module \simeq \langle \cS_q \rangle \subset \QC^!(\bL^u_{\ul{q}})^{\omega S^1}$$ from the $k[u]$-basechange of the unramified principal series of $G(K)$ to (renormalized) cyclic sheaves on the stack of graded unipotent Langlands parameters.
\end{thm}

\subsection{Relation to constructible sheaves and graded Hecke algebras}\label{sec graded hecke}

The unipotent loop space construction of section~\ref{loop spaces} gives us a link between this uniform coherent description of the representation theory of $G$ and the perverse description for fixed infinitesimal characters.  Indeed, if one fixes a semisimple element $\sigma \in \Gv$, one can consider the formal completion $\bL^u_{\underline{q}, [\sigma]}$ of $\bL^u_{\underline{q}}$ along the locus of pairs $(\sigma',n)$ on which the semisimple part of $\sigma'$ lies in the conjugacy class $[\sigma]$ of $\sigma$.  One can then identify this formal completion with the $q$-specialized unipotent loop space $\LL^u_{\underline{q}}(\check{\cN}^{\alpha}/\Gv^{\alpha})$ of the Vogan variety attached to the pair $\alpha := (\sigma, q)$.

Theorem~\ref{kd stacks thm} thus provides us a link between the ($S^1$-equivariant) coherent sheaves on the stack of unipotent Langlands parameters and $\Gv^\alpha$-equivariant\footnote{Since we specialized at $q \in \bG_m$ rather than completed, the resulting $\cD$-modules are $\bG_m$-constructible rather than $\bG_m$-equivariant.  However, $\bG_m$-constructibility is automatic given $\Gv^\alpha$-equivariance; see Remark 3.3.13 in \cite{koszul}.} $\cD$-modules on $\check{\cN}^{\alpha}/\Gv^{\alpha}$; this equivalence carries the $!$-restriction to $\bL^u_{\underline{q}, [\sigma]}$ of the coherent Springer sheaf (henceforth denoted $\cS(\alpha)$) to the $\cD$-module that corresponds to the $\alpha$-Springer sheaf $\bS(\alpha)$. We have the following direct application of Corollary \ref{aha gha} to this setting.
\begin{thm}
%The sheaf $\cS(\alpha)$ has a canonical graded lift, which gives rise to an isomorphism of graded algebras
%$$\End_{\cL_\alpha(\check{\mf{g}}/\Gv_{\gr})}(\cS(\alpha))^\shear \simeq \End_{\mf{g}(\alpha)/\Gv_{\gr}(\alpha)}(\mathbf{S}(\alpha)).$$
Let $\mathbf{H}(\alpha) = \End_{\check{\cN}^\alpha/\Gv^\alpha}(\mathbf{S}(\alpha))$ and $\cH(\alpha)$ denote the completion of the affine Hecke algebra at $\alpha$.  We have a Koszul duality:
$$\wh{\IndCoh}(\bL^u_{\underline{q}, [\sigma]})^{\omega S^1} \supset \langle \cS(\alpha) \rangle \simeq \Mod(\cH(\alpha))  \kos \Mod(\mathbf{H}(\alpha)) \simeq \langle \mathbf{S}(\alpha) \rangle \subset \cD(\check{\cN}^\alpha/\Gv^\alpha).$$
Working 2-periodically, we have an identification of the full subcategories
$$\wh{\IndCoh}(\bL^u_{\underline{q}, [\sigma]})^{\Tate} \supset  \Mod_{k[u,u^{-1}]}(\cH(\alpha)^{per})  \subset \cD(\check{\cN}^\alpha/\Gv^\alpha)^{per}.$$
\end{thm}

The algebras $\mathbf{H}(\alpha)$ are shearings of completions of Lusztig's graded Hecke algebras \cite{cuspidal1, gradedAHA}.  We note that the equivariantization with respect to graded lifts, shearing, and de-equivariantization that appears in Koszul duality (as discussed in \ref{graded lifts}) is compatible with the corresponding gradings in statements in \cite{soluni, soluni2}, and that one can avoid such complications by working 2-periodically.

\medskip
\subsection{Periodic cyclic sheaves and pure inner forms}

\subsubsection{The case $G = GL_n$} In the case of $GL_n$, the unramified principal series exhausts all the unipotent representations (and $GL_n$ has only one pure inner form).  Also in the case of $GL_n$, Springer theory simplifies considerably: the Springer sheaf generates all equivariant $\D$-modules sheaves on the nilpotent cone, i.e. the only local systems that appear in the decomposition of the Springer sheaf are trivial, and all such local systems appear. This gives a derived equivalence between the category of $D$-modules $\cD(\cN_{GL_n}/GL_n)$ and modules for the graded Hecke algebra.  Moreover the same is true with arbitrary central character, i.e., for the $(\sigma,q)$-variants of the nilpotent cone. In other words, the image of the coherent Springer sheaf generates the categories of periodic cyclic sheaves when completed at arbitrary parameters.

In~\cite{BCHN2} we show the following:

\medskip

\begin{thm}
For $G=GL_n$, the embedding 
$$D(GL_n) \otimes_k k[u] \subset \QC^!(\bL_{F,GL_n})^{\omega S^1}$$
 becomes an equivalence after inverting $u$,
$$D(GL_n)\ot_k k[u,u\inv] \simeq \QC^!(\bL_{F, GL_n})^{\Tate}.$$
\end{thm}

\medskip

As with Theorem~\ref{BCHN GLn theorem}, standard techniques reduce the assertion to the corresponding statement for the Iwahori block: namely, under the deformation from all coherent sheaves to periodic cyclic sheaves only the coherent Springer block survives. There we establish that the Springer resolution and Hecke algebras enjoy strong enough finiteness properties that we can deduce the generation property of the Springer sheaf $$\langle \cS_q \rangle\simeq \QC^!(\bL^u_{\ul{q}})^{\Tate}$$ from the corresponding generation assertions one completion at a time (a sort of ``fracture theorem"). 

\medskip

\subsubsection{The case of general $G$}

For a general reductive group, there are more perverse sheaves on the nilpotent cone and its $(\sigma,q)$-fixed loci, coming from {\em cuspidal local systems} on Levi subgroups of $\Gv$, as classified by Lusztig (see e.g.~\cite{cuspidal1,cuspidal2}). On the other hand there are also more unipotent representations of $G$ and of its pure inner forms than unramified principal series. As described in Section~\ref{unipotent reps}, the two are matched by the unipotent local Langlands correspondence proved by Lusztig~\cite{lusztig unipotent} for $G$ adjoint, simple and unramified and extended by Solleveld~\cite{soluni, soluni2} to all connected groups.

 One expects (following variants of conjectures of~\cite{xinwen survey,hellmann,fargues}, see forthcoming work~\cite{hemo zhu} of Hemo and Zhu) that the entire category of coherent sheaves on $\bL^u_{q}$ parametrizes unipotent representations of the family of stabilizer groups of $G$-isocrystals.
 %\todo{write something about expectation that pure sheaves lift ``unaffected" to graded version??}

Inside of this large category one might try to identify the {\em pure} subcategory, parametrizing representations of the family of pure inner forms of $G$ itself. For $G$ semisimple these coincide with the groups (extended pure inner forms) associated to {\em basic} isocrystals, which are those isocrystals whose stabilizers are actually inner forms of $G$. For $G$ general (for example $G=GL_n$) pure inner forms correspond only to a subset of basic isocrystals.  

We propose in \cite{BCHN2} that the cyclic deformation interpolates between all isocrystals and pure inner forms: i.e.,  the category of periodic cyclic sheaves on $\bL_q^u$ is identified with the category of all unipotent representations of pure inner forms of $\Gv$. 

\medskip

\begin{conjecture}\label{unipotent conjecture}
There is a full embedding 
$$\bigoplus_\eta D_{f.g.}(\check{G}_\eta)^u \ot_k k[u] \subset \Coh(\bL^u_{\ul{q}})^{S^1}$$ from the trivial $u$-deformation of the sum of unipotent representation categories of pure inner forms $G_\eta$ of $G$ into cyclic sheaves on graded unipotent Langlands parameters.
Moreover this embedding becomes an equivalence after inverting $u$,
$$\bigoplus_\eta D(\check{G}_\eta)^u\ot_k k[u,u\inv] \simeq \QC^!(\bL^u_{\ul{q}})^{\Tate}.$$
\end{conjecture}

\medskip

In other words, under the deformation from all coherent sheaves to periodic cyclic sheaves we expect only the representations of pure inner forms survive and moreover the deformation is trivial on this subcategory.

\begin{rmk}
The embedding of $S^1$-equivariant sheaves in the above conjecture is not expected to be an equivalence, which can be seen even when $G = T$ is a torus in Example \ref{exmp BT}.  That is, the $S^1$-equivariant category has torsion which is killed by applying inverting $u \in C^\bullet(BS^1; k)$.
\end{rmk}

\begin{problem}
Give an automorphic characterization of the full subcategory of $S^1$-invariant sheaves
$$\IndCoh(\bL^u_{\underline{q}})^{\omega S^1} \otimes_{k[u]} k \subset \IndCoh(\bL^u_{\underline{q}}).$$
\end{problem}

\medskip

\subsection{Whence circle actions?}

It is tempting to propose a cyclic mechanism to relate coherent and constructible forms of the local Langlands correspondence in general, extending the results discussed in the unipotent nonarchimedean setting and paralleling those to be discussed in the archimedean setting.  To do so, at the very least, would require a circle action on stacks of Langlands parameters (thus its category of coherent sheaves).  

\medskip

In the non-Archimedian case the stack of unipotent arithmetic Langlands parameters $\bL^u_q$ (representations of the Weil-Deligne group of $F$) is realized by taking the derived fixed points of a non-trivial automorphism (corresponding to Frobenius on the automorphic side) on the stack of unipotent geometric Langlands parameters $\cN_{\check{G}}/\check{G}$ (Weil-Deligne representations of the inertia subgroup), which does not carry a circle action.  We are able to define a circle action by instead taking an $S^1$-invariant substack of the derived loop space, i.e., the derived fixed points of the identity map, of $\cN_{\check{G}}/(\check{G} \times \bG_m)$, which does carry such an action.  On the level of categories, the trace of the identity operator (Hochschild homology) carries a circle action, while the trace of Frobenius does not. 

\medskip

Let us re-examine the origin of the circle action on the stack of unipotent Langlands parameters, which we first identified from the explicit form of the equations. The moduli space $\bL_q^u$ is ``polynomial in $q$", in the sense that it is the fiber over $q\in \Gm$ of a natural algebraic family over $\Gm$. When we work $\Gm$-equivariantly, the total space of this family became a loop space -- i.e., we recover not $\bL_q^u$ itself but its quotient $\bL_{\ul{q}}^u$ by $\Gm$ by setting the $S^1$-invariant equation ``Frobenius acts by $q$" on a loop space. From the perspective of realizing representations of the affine Hecke algebra, this graded version served equally well, as the coherent Springer sheaf and all of its endomorphisms are $\Gm$-invariant.

%We now propose an alternative approach to producing circle actions on modifications of stacks of Langlands parameters, by taking both {\em fixed points for} and {\em quotient by} the action of Frobenius on geometric Langlands parameters (see Definition \ref{degree 1 loops}).

%In the unipotent case, rather than take the quotient $\LL_{\ul{q}}^u=\LL_q^u/\Gm$ by $\Gm$ we can instead take the prestack quotient by the subgroup generated by the Frobenius operator $q^{\Z}\subset \Gm$:
%$$\LL^u_{q^\Z}:=\LL_q^u/q^\Z.$$ This quotient inherits the circle action from $\LL_{\ul{q}}^u$: as with the $\Gm$-quotient, $\LL^u_{q^\Z}$ is the fiber of a derived loop space $\cL(\cN/\Gv\times q^\Z)$ over the $S^1$-invariant embedding $\{q\}\in \cL B\Z\simeq \Z/\Z$. 

\medskip

\subsubsection{Graded unipotent Langlands parameters}  Another way to express the origin of the graded lift in the setting of unipotent Langlands parameters is through the theory of categorical traces, as in~\cite{BNP, xinwen trace, hemo zhu}. Namely the entire category of ind-coherent sheaves on $\bL_q^u$ arises by~\cite{BNP} as the categorical trace of Frobenius acting on the spectral form of the affine Hecke category (the monoidal category of coherent sheaves on the Steinberg stack). Taking $K = \overline{\bF}_q((t))$, Bezrukavnikov's tamely ramified local geometric Langlands correspondence~\cite{roma ICM,roma hecke} gives a monoidal equivalence
$$\cD(I \bs G_F / I) \simeq \IndCoh(\wt{\cN}_{\Gv} \times_{\check{\mf{g}}} \wt{\cN}_{\Gv})$$
between the automorphic and spectral affine Hecke categories, intertwining the pullback by geometric Frobenius with the pushforward by scaling by $q$, whence the trace of Frobenius on the two monoidal categories is identified as well. On the automorphic side Hemo and Zhu~\cite{hemo zhu} relate this trace to unipotent representations of pure inner forms of $G(K)$ (among the more general groups associated to isocrystals).

\medskip

However Bezrukavnikov's equivalence has an expected graded lift (announced in~\cite{HoLi}). On the spectral side this simply involves incorporating $\Gm$-equivariance on Steinberg, whence the trace of Frobenius becomes $\IndCoh(\bL_{\ul{q}}^u)$. On the automorphic side one obtains the ``mixed" affine Hecke category~\cite{BY}, where we replace perverse sheaves by ``Tate" Weil sheaves on the affine flag variety, or more conceptually the graded category of $\ell$-adic sheaves as defined in great generality by Ho and Li~\cite{HoLi}. 

\medskip

Thus one can consider the trace of the identity on the graded affine Hecke category, calculated automorphically and spectrally, as providing a mechanism to extend the results of~\cite{hemo zhu} to a proof of Conjecture~\ref{unipotent conjecture}.

\medskip

\subsubsection{Beyond the unipotent setting}
Our experience in the unipotent setting suggests we should look for $S^1$-actions on categories of $\Gm$-equivariant sheaves on stacks of Langlands parameters. Indeed in the case of $GL_n$, we obtained such an action by a standard reduction to the Iwahori block. In general stacks of Weil-Deligne representations carry natural $\Gm$-actions rescaling the nilpotent endomorphism, and one can construct $S^1$ actions on the quotient by reduction to unipotent cases. 
We summarize our hopes in the following broad list of problems: 

%Beyond the unipotent setting (or $GL_n$) it is unclear if we should expect such ``polynomial in $q$" (or ``Tate") behavior (though on the automorphic side in the equal-characteristic setting of~\cite{xinwen survey} it is tempting to consider the mixed version of $\ell$-adic sheaves on $G(K)/_{\!q\,} G(K)$ following~\cite{HoLi}).

\begin{problem} Relate the coherent and constructible non-archimedean local Langlands correspondences as follows:
\begin{itemize}[leftmargin=5ex]
\item Identify a natural source for $S^1$-actions on the $\Gm$-quotients of stacks of Weil-Deligne representations.
\item Construct an $S^1$-action on the graded form (in the sense of~\cite{HoLi}) of $\ell$-adic sheaves on stacks of $G$-isocrystals.
\item Formulate an $S^1$-equivariant lift of the mixed form of the conjectures of Fargues and Zhu, i.e., an $S^1$-equivariant identification of these graded categories -- see Remark 4.6.8 in~\cite{xinwen survey}. 
\item Compare the full subcategory of $S^1$-equivariantizable sheaves on the automorphic side to the subcategory corresponding to representations of pure inner forms. 
\item Show that the cyclic deformation of the coherent local Langlands correspondence on this subcategory specializes, for fixed infinitesimal character, to (a categorical form) of Vogan's constructible local Langlands correspondence. 
\end{itemize}
\end{problem}

\begin{rmk}
It may be tempting to forget $\bG_m$-equivariance and introduce a $S^1$-action on the trace of an automorphism in general by imposing $\bZ$-equivariance with respect to only that automorphism.  Namely, let $X$ be a stack with an automorphism $\phi$; we may form the stack $X/\phi^{\bZ}$ which is a stack over $B\bZ = S^1$.  Then, its loop space $\cL(X/\phi^{\bZ})$ lives over $\cL(B\bZ) \simeq \bZ \times S^1$.  Taking the fiber over $\{1\} \times S^1$ (put another way, the substack of $\Map(S^1, X/\phi^{\bZ})$ whose composition along $X/\phi^{\bZ} \rightarrow B\bZ = S^1$ is a degree 1 map), we evidently obtain a space with an $S^1$-action.  Unfortunately, the resulting $S^1$-action is free, thus the resulting categories will be entirely $u$-torsion and thus be killed by the $S^1$-localization.  Roughly speaking, the difference between $\bZ$-equivariance and $\bG_m$-equivariance can be seen via the failure of the functor $\Rep(\bG_m) \rightarrow \Rep(\bZ)$ (i.e. pullback along the map $\bZ \rightarrow \bG_m$ which sends $1$ to our chosen automorphism) to be fully faithful; the latter has nontrivial $\Ext^1$ while the former does not.  %Informally, imposing $S^1$-equivariance amounts to imposing an equation; the equation in the former setting can be trivial, while the equation in the latter setting cannot.
\end{rmk}

%!TEX root = IHESwriteup.tex

\section{From Archimedean Local Langlands to Twistor Geometric Langlands}\label{real}
In this section we discuss how to realize the basic paradigm of this paper, relating coherent and constructible forms of the local Langlands correspondence via equivariant localization for circle actions on derived loop spaces, in the archimedean setting, following~\cite{loops and parameters, loops and reps, betti}. To do so we must first describe these two forms of the correspondence. The constructible form we take is a variant on Soergel's conjecture, while the coherent form is given by the tamely ramified geometric Langlands correspondence on the twistor line $\twistor$, which can be viewed as an archimedean counterpart to Fargues' conjecture. 

We begin with the more familiar constructible archimedean local Langlands correspondence (for fixed infinitesimal character), which is a theorem of Adams-Barbasch-Vogan (ABV) on the level of Grothendieck groups and a conjecture of Soergel on the level of derived categories.
In Section~\ref{automorphic real} we describe the automorphic side, realizing the relevant categories of representations of pure inner forms as categories of equivariant constructible sheaves on flag varieties using Kashiwara-Schmid's variant of Beilinson-Bernstein localization. In Section~\ref{spectral real} we present (following~\cite{loops and reps}) a simple stacky description of the spectral side, the ABV geometric parameter spaces~\cite{ABV} (see Mason-Brown's article~\cite[2.2]{MB} in these proceedings). The comparison of the two, Soergel's conjecture, is described in Section~\ref{soergel section}. 

We then proceed to present a ``families" version of archimedean local Langlands, in which we allow the infinitesimal character to vary. 
In Section~\ref{soergel in families} we describe the automorphic (representation theory) side, and in Section~\ref{sec twistor spectral} the spectral side. Namely, we introduce a stack $\bL^\eta$ of Langlands parameters which is the archimedean counterpart of the stacks encountered in the nonarchimedean setting. This stack has an elementary and explicit description reminiscent of the stack of unipotent Langlands parameters, but can also be realized as a stack of local systems on the twistor lime $\twistor$ with tame ramification at infinity, which provides it with a natural circle action as a twisted loop space. We explain how this stack smoothly interpolates between the ABV parameter spaces for varying infinitesimal character, by a form of Jordan decomposition of loops. We then formulate a ``families version" of Soergel's conjecture, identifying smooth representations of pure inner forms (with arbitrary infinitesimal character) with cyclic sheaves on $\bL^\eta$ (i.e., the Tate construction on $S^1$-equivariant coherent sheaves). This conjecture is closely parallel to the nonarchimedean picture described in the previous chapter. 

Finally in Section~\ref{sec twistor automorphic} we explain the automorphic counterpart to the full category of coherent sheaves on $\bL^\eta$.
We formulate the tamely ramified geometric Langlands conjecture for $\twistor$ (following~\cite{loops and parameters, betti}) and explain how it recovers Soergel's conjecture as its periodic cyclic deformation. The underlying geometric mechanism on the automorphic side is the realization of the equivariant flag varieties appearing as the automorphic side of Soergel's conjecture (via the Kashiwara-Schmid description of representations) as the semistable locus (and $S^1$-fixed points) of the stack of parabolic bundles on $\twistor$.

\subsection{Archimedean Local Langlands: Automorphic side}\label{automorphic real}
Let $(G,\theta)$ be a complex reductive group equipped with a quasisplit real form.  The real form $\theta$ gives rise to a collection $\Theta$ of pure inner forms, which arise geometrically in the following way~\cite{loops and reps, bernstein stacks}: consider the Galois-fixed points $BG^\Gamma$ of $BG$ where $\Gamma$ acts by the conjugation $\theta$: we have a decomposition
$$BG^\Gamma\simeq \coprod_{\tau\in \Theta} BG_\tau$$ where $\Theta$ is the set of equivalence classes of pure inner forms of $\theta$ and $G_\tau$ is the corresponding real form (which may appear with multiplicities).

The local Langlands correspondence parametrizes representations for the groups $G_\tau$ as $\tau$ varies over $\Theta$. 
Thus an ultimate goal of the real local Langlands program is to describe the entire dg category
$$\HC_{\Theta}=\bigoplus_{\tau\in \Theta} \HC_{\tau}$$ 
of Harish-Chandra modules for real groups in the pure inner class $\Theta$
in Langlands dual terms. For each $[\lambda]\in \h^*/W$, we write
$\HC_{\tau, [\lambda]}$ for the  dg category of Harish Chandra modules
for the real form $G_{\tau}$ with pro-completed generalized infinitesimal
character~$[\lambda]$, and 
$$\HC_{\Theta,[\lambda]}=\bigoplus_{\tau\in \Theta} \HC_{\tau,[\lambda]}.$$

We now assume that $\lambda$ is regular (see Remark~\ref{singular characters} for comments on the singular case), and let $K_\tau$ denote the complexification of a maximal compact subgroup of $G_\tau$. Then we can apply Beilinson-Bernstein localization to describe the category $\HC_{\tau,[\lambda]}$ of Harish-Chandra $(\fg,K_\tau)$-modules for the real groups $G_\tau$ in terms of $\lambda$-twisted $\D$-modules on the flag variety $G/B$, equivariant for the corresponding complex symmetric subgroup $K_\tau$:
$$\HC_{\tau,[\lambda]}\simeq \cD_\lambda(K_\tau \backslash G/B).$$ 
Applying the Riemann-Hilbert correspondence, this de Rham realization has a Betti counterpart as $\alpha=\exp(2\pi i\lambda)$-twisted constructible sheaves on $K_\tau \backslash G/B$ (see~\cite{kashiwara BB article}).

Kashiwara and Schmid~\cite{KS,kashiwara real} introduced another identification of derived categories of representations with equivariant derived categories on flag varieties, which is more directly related to admissible representations of the groups $G_\tau$ (for example by natural globalization functors):
$$\HC_{\tau,[\lambda]}\simeq \Shv_{\lambda}(G_\tau\backslash G/B).$$
In this realization we replace $K_\tau$-equivariance by $G_\tau$-equivariance, and the identification of the two realizations is provided by the Matsuki correspondence for sheaves~\cite{MUV}. 
%\todo{somewhere be careful about fixed vs generalized infi character as twisted sheaves on $G/B$ etc}

We can describe all of these categories for varying $\tau$ simultaneously using homotopy fixed points: we have a natural identification
$$\coprod_{\tau\in \Theta} G_\tau\backslash G/B \simeq (B\backslash G/B)^{\Gamma}$$ where $\Gamma$ acts on $B\backslash G/B\simeq G\backslash(G/B\times G/B)$ by switching the factors composed with the real form $\theta$. In other words, representations of the entire pure inner class $\Theta$ are naturally realized as
$$\HC_{\Theta,[\lambda]}\simeq \Shv_{\lambda}((B\backslash G/B)^{\Gamma}).$$

\subsection{Archimedean Local Langlands: Spectral side}\label{spectral real}

Let $(\Gv,\eta)$ be the Langlands dual group to $(G,\theta)$ with its dual algebraic involution. Introduce the geometric $L$-group $G^L$ associated to $G$ and the conjugation $\theta$ via the Galois-equivariant Satake equivalence~\cite{xinwen ramified, xinwen Satake}. It is an extension
$\Gv \to G^L\to \Gamma$ by the Galois group $\Gamma = \Gal(\C/\R)$ (though not necessarily a semi-direct product).
To avoid subtleties with the definition of the L-group we will assume $G$ is of adjoint type.

\medskip

On the spectral side, the involution $\eta$ likewise gives rise to a collection of involutions in the inner class of $\eta$ parameterized by the set $\Sigma(\eta) := \{\sigma \in \Gv \mid \sigma \eta(\sigma) = e\}/\Gv$, and likewise a decomposition
$$B\Gv^{\Gamma} \simeq \Hom(\Gamma, \check{G} \rtimes \Gamma)/(\check{G} \rtimes \Gamma) \simeq  \{\sigma \in \Gv \mid \sigma \eta(\sigma) = e\}/\Gv \simeq \coprod_{\sigma\in \Sigma(\eta)} BK_\sigma$$
where $K_\sigma:=\Gv^{\iota} \subset \Gv$ is the corresponding symmetric subgroup for the involution $\iota(g)= \tilde{\sigma} \eta(g) \tilde{\sigma}^{-1}$ attached to a representative $[\tilde{\sigma}] = \sigma \in \Sigma(\eta)$.  
\medskip 
 
Let $X(G^L)$ denote the ABV space of geometric parameters. Since the ABV local Langlands correspondence only concerns {\em equivariant} sheaves on $X(G^L)$, it is natural to consider instead the quotient stack $\cX(G^L):=X(G^L)/\Gv$.  The stack $\cX(G^L)$ is a disjoint union of stacks $\cX(\mathbb{O}_\lambda, G^L)$ over semisimple orbits ${\mathbb O}_\lambda =\Gv\cdot \lambda\subset \fgv$.  We write $e(\mathbb O_\lambda)=\exp(2\pi i \mathbb O_\lambda)=\Gv\cdot \alpha \subset \Gv$ with $\alpha=\exp(2\pi i \lambda)$, and $\Pv(\lambda) \subset \Gv(\alpha)$ the associated parabolic in the centralizer.  We write the symmetric subgroup $K_\sigma(\alpha) := P(\lambda)^\sigma$ for
%\todo{is this ok}
\begin{equation}\label{sigma alpha}
\sigma \in \Sigma(\eta, \alpha) := \{\sigma \in \Gv \mid \sigma \eta(\sigma) = \alpha\}/\Gv(\alpha).
\end{equation}

\medskip

The description of the $\Gv$-orbits~\cite{MB} combined with simple observations about homotopy fixed points gives rise to the following succinct description of the equivariant ABV spaces from~\cite{loops and reps}:

\medskip
\begin{prop}\label{prop abv}
The ABV stack $\cX(\mathbb O_\lambda, G^L)$ for fixed $\lambda$ is identified with the disjoint union
$$\coprod_{\sigma\in \Sigma(\eta, \alpha)} K_\sigma(\alpha)\backslash \Gv(\alpha)/\Pv(\lambda)$$
which in turn is identified with the fixed point stack 
$$(\Pv(\lambda)\backslash \Gv(\alpha)/\Pv(\lambda) )^{\Gamma}$$
where $\Gamma$ is acting by exchanging the factors composed with $\eta$.
\end{prop}
\medskip

\subsection{Soergel's conjecture}\label{soergel section}
We will now assume that $\mathbb O$ is a {\em regular} semisimple orbit, i.e., that we are parametrizing representations with regular infinitesimal character; see Remark~\ref{singular characters} for discussion of singular characters.
%\todo{not sure this is right, rho-shift, dominant, blah blah (H: i think there is a rho shift, but don't have to say dominant because $\lambda$ only depends on its orbit?  i think it's ok to just say regular.)} 
In particular the associated parabolic $P(\lambda)$ in the centralizer $\Gv(\alpha)$ of $\alpha$ is a Borel $\Bv(\alpha)$.
In this case the ABV parametrization is succinctly expressed as describing the Grothendieck group of the category of sheaves on $(\Bv(\alpha)\backslash \Gv(\alpha)/\Bv(\alpha))^\Gamma.$

\medskip

\begin{conj}[Soergel's conjecture,~\cite{loops and reps} formulation]\label{soergel conj}
There is a Koszul duality 
$$\Shv_{\alpha}((B\backslash G/B)^{\Gamma}) \kos \Shv((\Bv(\alpha)\backslash \Gv(\alpha)/\Bv(\alpha))^\Gamma)$$ between the category of Harish-Chandra modules for the pure inner class $\Theta$ with regular infinitesimal character $\lambda$ with exponential $\alpha$ and the category of sheaves on the ABV geometric parameter stacks.
\end{conj}

\medskip

By {\em Koszul duality} $\cC\kos \cD$ we mean that the categories $\cC$ and $\cD$ admit graded lifts  $\cC_{gr},\cD_{gr}$ (also known as ``mixed versions"), and that we have an equivalence of graded categories $$\cC_{gr}\simeq\cD_{gr}^\shear$$ after a cohomological shear (See Definition \ref{koszul duality def}). Because of the shear the underlying categories $\cC, \cD$ need not be equivalent, but their 2-periodic versions (periodic localizations) are identified,
$$\cC\ot_k k[u,u\inv]\simeq \cD\ot_k k[u,u\inv].$$ Such a Koszul duality induces a duality between the Grothendieck groups of $\cC$ and $\cD$ (see e.g. the introduction of~\cite{RomaKari}), so that Soergel's conjecture produces Vogan's character duality~\cite{vogan duality}. 

\medskip

A stronger form of Conjecture~\ref{soergel conj} prescribes that the equivalence respect natural actions of Hecke categories on both sides, categorifying the Hecke algebra symmetries of the Grothendieck groups of representations (``coherent continuation representations"). Indeed, for $\alpha=1$ the Koszul duality theorem of~\cite{BY} provides a monoidal equivalence
$$\Shv(B\backslash G/B)\kos \Shv(\Bv\backslash \Gv/\Bv)$$ --- i.e., a strong monoidal form of the Koszul duality of~\cite{BGS} or the complex case of Soergel's conjecture (proved in~\cite{soergel}) for trivial monodromy. 
More recently the endoscopic Koszul duality theorem of~\cite{lusztig yun} provides a monoidal equivalence\footnote{For $G$ with connected center, e.g., adjoint type - otherwise need to restrict to the principal block on the left.}
$$\Shv_{\alpha}(B\backslash G/B)\kos \Shv(\Bv(\alpha)\backslash \Gv(\alpha)/\Bv(\alpha)),$$ so a monoidal form of the general (regular) complex case of Soergel's conjecture. The two sides of Conjecture~\ref{soergel conj} are naturally module categories for the two monoidal categories identified here, and the equivalence is expected to respect this structure. 

\begin{remark}[Singular infinitesimal character]\label{singular characters}
For $\lambda$ singular, we need to modify the statement of Conjecture~\ref{soergel conj} as follows. On the spectral side, we replace the Borel $\Bv(\alpha)$ by the parabolic $\Pv(\lambda)$ determined by $\lambda$ as in~\cite{MB}. On the automorphic side, the category of Harish-Chandra modules for $\lambda$ singular is a quotient of the corresponding category of equivariant sheaves on the flag variety, described e.g. in~\cite{kashiwara BB article}. Equivalently, there is a natural projector (an idempotent monad) acting on the category of sheaves whose modules are identified with representations. On the level of Hecke categories (or in the case of a complex group) these matching modifications are the parabolic-singular Koszul duality of~\cite{BGS,BY}.
\end{remark}

\subsection{Representations of real groups in families}\label{soergel in families}
We now recast the automorphic side of Soergel's conjecture, i.e., the category of Harish-Chandra modules
$\HC_{\Theta,[\lambda]}$ for real forms $G_\tau$ in the pure inner class of $\theta$ with regular infinitesimal character $\lambda$. As we described, this is realized by Kashiwara-Schmid localization~\cite{KS,kashiwara real} as $\alpha$-twisted sheaves (for $\alpha=\exp(2\pi i \lambda)$) on
$$\coprod_{\tau\in \Theta} G_\tau\backslash G/B \simeq (B\backslash G/B)^{\Gamma}.$$
In order to incorporate variation with $\alpha$ we first reformulate $\alpha$-twisted sheaves on $G/B$ as sheaves on the torus bundle $G/N\to G/B$ which are locally constant with generalized monodromy $\alpha$ along the fibers. To let $\alpha$ vary we simply drop the monodromicity condition and allow arbitrary sheaves which are locally constant along the fibers. This local constancy can be reformulated as not allowing semisimple directions in the singular support of our sheaves -- i.e., we are considering $\Shv_\cN(G/N)$, sheaves on $G/N$ with nilpotent singular support. 

Thus a families version of the automorphic side of Conjecture~\ref{soergel conj} is provided by a direct sum over the pure inner class  $\tau\in \Theta$ of categories of nilpotent sheaves $\Shv_\cN(G_\tau\backslash G/N)$. 

\subsubsection{From sheaves to representations}\label{sheaves to reps}
It is natural to wonder how the categories $\Shv_\cN(G_\tau\backslash G/N)$ relate to our original motivation, namely, representations of the real groups $G_\tau$. For this we need a form of Kashiwara-Schmid localization with varying central character. 

In the de Rham setting of $\D$-modules, the article~\cite{BZN BB} introduced a version of Beilinson-Bernstein localization for varying infinitesimal characters, replacing twisted $\D$-modules $\D_\alpha(G/B)$ on $G/B$ by weakly $H$-equivariant $\D$-modules $\cD_H(G/N)$ on $G/N$.  The result was an identification 
$$U\fg\module\simeq \cD_H(G/N)^{\mathbb D}$$
of $U\fg$-modules with modules for the Weyl-Demazure monad $\mathbb{D}$, an explicit algebra acting on $\cD_H(G/N)$ enforcing Weyl-group invariance -- i.e., accounting for the descent data from $\D$-modules, which depend on $\lambda\in \fh^*$, to representations,  which depend on $[\lambda]\in \fh^*/W$. Specializing to singular infinitesimal character this recovers the quotient from $\D$-modules to representations discussed in Remark~\ref{singular characters}. 
If we introduce equivariance for symmetric subgroups, we obtain a similar description
$$\HC_{\Theta}\simeq \bigoplus_{\tau\in \Theta} \D_H(K_\tau\backslash G/N)^{\mathbb D}$$
for the entire category of Harish-Chandra modules for the pure inner class $\Theta$. 

\medskip

The Betti category $\Shv_\cN(G/N)$ differs from its de Rham counterpart $\cD_H(G/N)$ by discarding the choice of logarithm $\lambda\in\fh^*$ of the monodromy $\alpha\in \check{H}$ -- this accounts for the difference between 
local systems and flat connections on a torus
$$\Loc(H)\simeq \QC(\check{H}) \;\;\;\;\;\;\mbox{vs.}\;\;\;\;\;\mathcal{C}onn(H) \simeq \QC(\fh^*/X_\bullet(\check{H})).$$
%local systems on a torus and weakly equivariant $\D$-modules
%$$\Loc(H)\simeq \QC(H^\vee) \;\mbox{vs.}\;\D_H(H)\simeq QC(\fh^*).$$
However this distinction can be removed by keeping track of an extra lattice grading, so that in particular we can recover
$U\fg\module$ as algebras over an ``affine Weyl" monad $\mathbb D_{\mathrm{aff}}$ on $\Shv_\cN(G/N)$. Likewise we can recover $$\HC_{\Theta}\simeq \bigoplus_{\tau\in \Theta} \Shv_\cN(G_\tau\backslash G/N)^{\mathbb D_{\mathrm{aff}}},$$
providing a direct link between the automorphic side of Conjecture~\ref{families soergel} and representations of real groups.

\subsection{Twistor Geometric Langlands: Spectral side}\label{sec twistor spectral}
A natural families version of the spectral side of Soergel's conjecture comes from considering Langlands parameters on the twistor line.  We introduced the \emph{twistor Langlands parameter stack} $\bL^\eta$ in the following equivalent ways.
\medskip

(1) Letting $\cB := \Gv/\Bv$ denote the flag variety, it is the stack
$$\bL^\eta := \{(g, \Bv') \in \Gv \times \cB \mid g \eta(g) \in \Bv'\}/\Gv$$
where $\Gv$ acts by $\eta$-twisted conjugation, i.e. $h \cdot (g, \Bv') = (hg\eta(h^{-1}), h\Bv'h^{-1})$.
% of a group element $g\in \Gv$ and a Borel $\Bv' \in \cB:=\Gv/\check B$ containing its $\eta$-twisted square $g\eta(g)$. 
In other words, $\bL^\eta$ is the fiber product
$$\xymatrix{\bL^\eta \ar[r]\ar[d]& (G^L - \Gv)/\Gv\ar[d]^-{(-)^2}\\
\Bv/\Bv \ar[r] & \Gv/\Gv}$$
of the square map on the second component of $G^L$ with the Grothendieck-Springer resolution  of $\Gv/\Gv$.  There is a monodromy map to the universal Cartan 
$$\chi: \bL^\eta \rightarrow \Bv/[\Bv,\Bv] = \Hv, \;\;\;\;\;\;\;\;g \mapsto g \eta(g) \pmod{[\Bv, \Bv]}.$$

\medskip

(2) We have a description in terms of loop spaces (see Theorem 4.6 of \cite{loops and reps}
$$\bL^\eta \simeq \cL((\Bv \bs \Gv / \Bv)^\Gamma) \simeq \cL(\Bv \bs \Gv / \Bv)^\Gamma$$
where $\eta \in \Gamma = \bZ/2\bZ$ acts by $\eta$ and swapping the two factors, and an $S^1$-action on $\bL^\eta$ from this description.  The monodromy map is encoded by the $S^1$-equivariant map
%\todo{check}
$$\chi: \cL(\Bv \bs \Gv / \Bv)^\Gamma \rightarrow \cL(B\Bv \times B\Bv)^\Gamma \simeq \Bv/\Bv \rightarrow \Bv//\Bv \simeq \Hv.$$
%To see this, we pass through equivariant localization (Theorem \ref{LocThm})\todo{dont have it for stacks}
We also have a description as the $\eta$-twisted loop space
$$\cL((\Bv \bs \Gv / \Bv)^\Gamma) \simeq \cL_\eta((\cB \times \cB)/\check{G}) = \cL((\cB \times \cB)/G^L) \times_{\cL(B\Gamma)} \{\eta\}$$
%\footnote{We also note that $\cL((\cB \times \cB)/G^L) = \cL(\check{B}\bs\check{G}/\check{B})/\Gamma \coprod \LL^\eta/\Gamma$, i.e. the other connected component of the loop space is the usual Steinberg stack, with an additional $\Gamma$-equivariance.}  
though this presentation does not a priori inherit an $S^1$-action from the loop space $\cL((\cB \times \cB)/G^L)$ (see Definition \ref{def loop spaces}).\footnote{If we impose $\Gamma$-equivariance, i.e. take the specialized loop space $\cL_\eta'$, there is a circle action, but this action is ``half'' the degree of the action we consider. The $\Gamma$-equivariance allows for a well-defined notion of half-degree.}

\medskip

(3) We also recall (see e.g.~\cite[3.2]{betti}) that $\bL^\eta$ may be identified with $\eta$-twisted parabolic $\Gv$-local systems on the twistor $\twistor$, i.e. the real form defined by $\Pp^1_{\C}$ modulo the $\Gamma$-action by antipodal map $z \mapsto -1/\bar{z}$. Namely, we consider $G^L$-local systems on $\twistor\setminus \infty$ whose induced $\Gamma$-local system is identified with the orientation double cover, and moreover are equipped with a Borel containing the monodromy. Equivalently, these are $\Gv$-local systems on $\Pp^1_{\C} \setminus 0,\infty$ with invariant flags at the two poles, and equipped with a $\Gamma$-equivariant structure respecting the flags.    The description of $\bL^\eta$ makes evident an action of $S^1$, coming from the geometric action of $U(1)$ as symmetries of $(\twistor,\infty)$ (equivalently, rotation of $\Pp^1_{\C}$ along the axis through the poles, which commutes with the anitipodal map) -- in fact the stack of parabolic local systems on $(\Pp^1_{\C},0,\infty)$, the group version of the Steinberg stack, is naturally identified as the loop space $\cL(\Bv\backslash \Gv/\Bv)$.   The monodromy map is given by exactly by the monodromy around $\infty$.

%We have an $S^1$-invariant map $\LL^\eta \to \Hv$ which records the diagonal part of $g\eta(g)$ in $\Hv=\Bv/[\Bv,\Bv]$.   

\medskip

For an element $\alpha\in \Hv$ of the universal Cartan, we define the
monodromic twistor parameter space $\bL^\eta_\alpha$ to be the formal completion of $\bL^\eta$ along the fiber over $\alpha$.  A key observation of~\cite{loops and reps} is an identification of $\bL^\eta_\alpha$ as the unipotent loop space of the corresponding ABV space (recall Proposition \ref{prop abv}):\footnote{Recall from (\ref{sigma alpha}) the meaning of $\Sigma(\eta, \alpha)$.}
$$
\bL^\eta_\alpha \simeq  \coprod_{\sigma\in \Sigma({\eta,\alpha})}
\cL^u(K_{\sigma}(\alpha)\backslash \Gv(\alpha)/\Bv(\alpha)) \simeq \cL^u((\Bv(\alpha) \bs \Gv(\alpha) / \Bv(\alpha))^\Gamma)
$$
This identification may be viewed as a special case of equivariant localization for the group $B \times B$, discussed in Example \ref{loc BB}, i.e. it obtained from the loop space $\cL((B \bs G/B)^\Gamma)$ by completing at $\alpha \in B//B = H$.  In particular, $\bL^\eta_\alpha$ comes equipped with a $\Gm$-action contracting it to the ABV space $\cX_\alpha$.  This identification respects circle actions up to an explicit central twist (see Section \ref{sec central twist}), and as a result the categories of cyclic sheaves are identified.
Moreover, applying Theorem \ref{formal is unipotent} identifying Tate sheaves on formal and unipotent loops, as a result of the contracting $\Gm$-action, one can identify cyclic sheaves on $\bL^\eta_\alpha$ directly with filtered $\D$-modules (and, by Riemann-Hilbert on finite orbit stacks, sheaves) on the ABV spaces, i.e., with the spectral side of Soergel's conjecture:\footnote{See Section \ref{circle actions sec} for a discussion of the renormalized Tate construction.}
\begin{thm}
There is a $k[u, u\inv]$-linear equivalence of categories
$$\IndCoh(\bL^\eta_\alpha)^{\Tate} \simeq \Shv((\Bv(\alpha)\backslash \Gv(\alpha)/\Bv(\alpha))^\Gamma)\ot_k k[u,u\inv]$$
between periodic cyclic sheaves on the monodromic twistor Langlands parameter space and the category of sheaves on the $\alpha$-ABV parameter spaces (base changed to $k[u,u\inv]$).
\end{thm}

\medskip

%\todo{Say something about graded lift / $\G_m$ actions?}

Thus we have a families version of the spectral side of Soergel's Conjecture to compare with the automorphic counterpart given by nilpotent sheaves on $G/N$. We encode this expectation in the following conjecture (postponing for the moment natural compatibilities with Hecke actions):

\medskip

\begin{conj}[Families Soergel Conjecture]\label{families soergel}
There is an equivalence of categories 
$$\Shv_\cN((N\backslash G/N)^{\Gamma})\ot_k k[u,u\inv] \simeq  \QC^!(\bL^\eta)^{\Tate}.$$
\end{conj}

\medskip

\subsection{Twistor Geometric Langlands: Automorphic side}\label{sec twistor automorphic}
The stack $(B\backslash G/B)^{\Gamma}$ appearing on the automorphic side of Conjecture~\ref{soergel conj} (or its $N$-version appearing in Conjecture~\ref{families soergel}) has a natural geometric interpretation in terms of the twistor line, discovered in~\cite{loops and parameters}.

\begin{defn} The topological stack of real $G$-bundles on $\twistor$ 
$$\Bun_{G,\theta}(\twistor;\infty):=\Bun_G(\Pp^1;0,\infty)^\Gamma$$  is the fixed point stack of $\Gamma$ acting on (the Betti realization of) the stack of $G$-bundles on $\Pp^1$ equipped with decorated flags (i.e., $N$-reductions) at $0,\infty$ by composition of the antipodal map and the involution $\theta$ on $G$. 
\end{defn}

\medskip
The Betti form of the geometric Langlands correspondence~\cite{betti} seeks to describe the categories of nilpotent sheaves on stacks of $G$-bundles $\Shv_\cN(\Bun_G)$ in terms of algebraic geometry of stacks of local systems. In the case of parabolic bundles on the twistor line, we have already identified the corresponding space of local systems with $\bL^\eta$, so that we have the following:

\medskip
\begin{conj}[Twistor Geometric Langlands~\cite{loops and parameters, betti}]\label{twistor conjecture}
There is an equivalence $$\Shv_\cN(\Bun_{G,\theta}(\twistor,\infty))\simeq \QC^!(\bL^\eta)$$ intertwining natural affine Hecke symmetries.
\end{conj}

\begin{remark}
There are two variants (``standard" and ``renormalized") of the large categories of sheaves on stacks, corresponding to the two Koszul dual variants of sheaves on $BG$ (modules for chains on $G$ vs. modules for cochains on $BG$). These correspond to imposing or not imposing nilpotent singular support of coherent sheaves on the spectral side.
\end{remark}

The open locus in $\Bun_{G,\theta}(\twistor;\infty)$ where the underlying $G$-bundle on $\Pp^1$ is trivial is identified with the fixed point stack $(N\backslash G/N)^{\Gamma}$ of the Galois group on the stack $N\backslash G/N$, which parametrizes pairs of decorated flags on the trivial bundle. Thus the category $\Shv_\cN((N\backslash G/N)^{\Gamma})$ and its variants appearing in Soergel's conjecture is identified with sheaves on an open locus in $\Bun_{G,\theta}(\twistor;\infty)$, and thus fits in a semi-orthogonal decomposition of the category. The circle action respects this subcategory and is trivialized on it; indeed we have the following description of the periodic cyclic form of the twistor automorphic category, which follows from the equivariant localization theorem of~\cite{GKM} applied to Hom spaces of $S^1$-equivariant sheaves on $\Bun_{G,\theta}(\twistor,\infty)$:

\medskip
\begin{prop}~\cite{loops and parameters}
The periodic cyclic category of automorphic sheaves is identified as
$$\Shv_\cN(\Bun_{G,\theta}(\twistor,\infty))^{\Tate}\simeq \Shv_\cN((N\backslash G/N)^{\Gamma})\ot_k k[u,u\inv].$$
Thus the families Soergel conjecture, Conjecture~\ref{families soergel}, is identified with the periodic cyclic form of twistor geometric Langlands, Conjecture~\ref{twistor conjecture}.
\end{prop}
\medskip

In other words, the $u$-deformation of the twistor geometric Langlands conjecture (the coherent local Langlands correspondence over $\R$) picks out only the subcategory associated to representations of pure inner forms of $G$, and on this subcategory produces the constructible form of the categorical local Langlands correspondence.

\begin{remark}
For non-semisimple groups the trivial bundle (or $S^1$-fixed) locus -- whose sheaves participate in the cyclic deformation to representations of pure inner forms -- can be strictly smaller than the semistable locus, whose sheaves index representations of groups associated to basic isocrystals. 
%\todo{[INSERT DISCUSSION OF $GL_2$, quaternion algebra]}
\end{remark}

\appendix

\section{Foundations}\label{koszul}

%As we have seen above, the category of sheaves which appears when studying Langlands parameters in families is the derived category of coherent sheaves on certain derived loop spaces, while the category of sheaves which appears Langlands parameters at fixed semisimple parameters is the derived category of constructible sheaves on certain fixed point varieties.  In this section we discuss a precise way in which one can relate these two categories of sheaves.  In particular, we will describe a procedure by which we may:
%\begin{enumerate}
%\item Complete the coherent Springer sheaf at a semisimple Frobenius-parameter to obtain a filtered equivariant constructible sheaf, whose endomorphisms may be related to the graded Hecke algebra by recent work of Solleveld.\cite{sol1, sol2}
%\item Realize admissible objects in the derived category of modules for the graded Hecke algebra as sheaves on a category of Langlands parameters.
%\end{enumerate}

In this section we establish technical foundations for the representation theoretic applications discussed in previous sections.   We will begin with a general set-up: let $X$ be a smooth Artin stack with affine diagonal over a field $k$ of characteristic zero. % In all cases of present interest we may assume $X = Y/G$ where $Y$ is a smooth quasi-projective scheme with the action of an affine algebraic group $G$.  

\subsection{Circle and $B\bG_a$-actions}\label{circle actions sec}

We make a digression on categorical $G$-actions for the non-affine groups stacks $G = S^1, B\bG_a, B\bG_a \rtimes \bG_m$, since they may be alien to a reader accustomed to working with actions by affine algebraic groups.  Extensive discussion on this subject may be found in Section 6 of \cite{toly}; we also refer the reader to Section 3.1 of \cite{koszul} for the case $G = B\bG_a$.  We wish to draw attention to the following two phenomena which do not present when $G$ is an affine algebraic group.

\medskip

\noindent \textbf{Lack of (de)-equivariantization correspondence:} Since $BG = BS^1$ is not 1-affine, not all objects in a $\QCoh(G)$-module category $\cat{C}$ can be generated by invariant ones, i.e. the natural functor
$$\begin{tikzcd} \cat{C}^{\QCoh(G)} \otimes_{\QCoh(BS^1)} \cat{Vect} \arrow[r, hook] & \cat{C}\end{tikzcd}$$
is fully faithful (Proposition \ref{equivariant ff}) but no longer an equivalence.  On the other hand, when $G = B\bG_a$ (and simialrly for $G = B\bG_a \rtimes \bG_m$), by Theorem 2.5.7 of \cite{1affine} we have that $BG = B^2\bG_a$ is 1-affine, thus the functor is an equivalence:
$$\begin{tikzcd} \cat{C}^{B\bG_a} \otimes_{\QCoh(B^2\bG_a)} \cat{Vect} \arrow[r, "\simeq"]& \cat{C}.\end{tikzcd}$$
This is discussed in Section \ref{deeq sec}.

\medskip

\noindent \textbf{Renormalized invariants:} For $G = S^1, B\bG_a$, and $B\bG_a \rtimes \bG_m$, the stacks $BG$ are not compactly generated by perfect objects, i.e. $\QCoh(BG) \not\simeq \Ind(\Perf(BG))$, and the operations of ind-completion and $G$-invariants do not commute.  We will explain why we prefer to take the $G$-invariants of small categories and then ind-complete (which we call the \emph{category of compactly renormalized (weak) $G$-invariants}), rather than the other way around in Section \ref{S1 ren sec}.

\medskip

\subsubsection{Linearization and affinization of circle actions}\label{deeq sec}   Let $G$ be an affine algebraic group acting on a scheme $X$.  There are two module categories for different monoidal categories one can attach to this set-up.
\begin{enumerate}
\item The category $\QCoh(G)$ is monoidal under pushforward along group multiplication $m: G \times G \rightarrow G$, and acts on $\QCoh(X)$ via pushforward along the action map $a: G \times X \rightarrow X$.  Alternatively (and more naturally), we may view $\QCoh(G)$ as a comonoidal category under pullback, and consider comodule categories.
\item The category $\QCoh(BG)$ is monoidal under tensor product, and acts on $\QCoh(X/G)$ via pullback and tensoring.
\end{enumerate}
Furthermore, these two categories are related via the invariants and coinvariants constructions:
$$\QCoh(X)^{\QCoh(G)} \simeq \QCoh(X/G), \;\;\;\;\;\;\;\;\;\; \QCoh(X/G) \otimes_{\QCoh(BG)} \cat{Vect}_k \simeq \QCoh(X).$$
These operations are sometimes referred to as \emph{equivariantization} and \emph{de-\hspace{0ex}equivariantization} respectively. 

\medskip

In fact, these functors can be made sense of in a general context, where they arise as the adjoint invariants and reconstruction functors from Section 10.2 of \cite{1affine}
$$\begin{tikzcd}[column sep=30ex]
 \cat{Mod}(\QCoh(BG)) \arrow[r ,shift left=0.9ex, "{\cat{D} \mapsto \cat{D} \otimes_{\cat{Rep}(G)} \cat{Vect}_k}"] & \cat{Comod}(\QCoh(G)) \arrow[l, shift left=0.9ex, "\cat{C} \mapsto \cat{C}^{\QCoh(G)}"]
\end{tikzcd}$$
One can pass from the less familiar comodules to modules by the canonoical equivalence $\cat{Comod}(\QCoh(G)) \simeq \cat{Mod}(\QCoh(G)^\vee)$; if $\QCoh(G)$ is self-dual, then in addition we have $\cat{Mod}(\QCoh(G)^\vee) \simeq \cat{Mod}(\QCoh(G))$.\footnote{In many situations, $\QCoh(G)$ is self-dual, e.g. by D.1.2 of \cite{1affine} when $\QCoh(G)$ is rigid for any monoidal structure, which is satisfied by any perfect stack by Proposition 3.4.2 of \cite{GR}).}  We are interested in conditions under which these adjoints are in fact equivalences.  Following the discussion in \emph{loc. cit.}, this occurs when $BG$ is \emph{1-affine}, and by Proposition 10.4.4 of \cite{1affine}, $BG$ is 1-affine when $G$ is an affine algebraic group.  This is by no means the only case; by Theorem 2.5.7 of \emph{op. cit.} we see that $BG$ is 1-affine (and self-dual) when $G = B\bG_a, B\bG_a \rtimes \bG_m$.

\medskip

As mentioned, $B^2\bG_a$ is 1-affine, and we will now see that $BS^1$ is not.  We are interested in the monoidal category $\QCoh(S^1)$, where the monoidal structure is given by pushforward along multiplication $m: S^1 \times S^1 \rightarrow S^1$ and the unit is the skyscraper at the identity.  We recall the following standard calculation.
\begin{prop}
The category $\QCoh(S^1)$ is equivalent to the (derived) category of chain complex valued local systems on $S^1$.  Under Cartier duality we have monoidal equivalences and identifications
$$\begin{tikzcd}
& \cat{Vect} & \\
\QCoh(S^1) \arrow[r, "\simeq"] \arrow[ur, "p^*"] & \Mod(k\bZ) \arrow[r, "\simeq"] \arrow[u] & \arrow[ul, "q_*"'] \QCoh(\bG_m)
\end{tikzcd}$$
where $p: \Spec k \rightarrow B\bZ = S^1$ and $q: \bG_m \rightarrow \Spec k$, and where we take for monoidal structure the tensor product on $\QCoh(S^1)$ and $\QCoh(\bG_m)$ and the convolution product on $\Mod(k\bZ)$.\footnote{This is not important for us, but note that $\QCoh(S^1)^{\omega} \supsetneq \Perf(S^1)$, so $S^1$ is not perfect.}
\end{prop}
%\begin{proof}
%In general, for any homotopy type $X$, the category $\QCoh(X)$ is the derived category of chain complex valued local systems on $X$.  Formally, one can write $X = \lim_X \mathrm{pt}$, i.e. the $\infty$-categorical limit with diagram $X$, so that $\QCoh(X) = \lim_X \cat{Vect}_k$.  
%\end{proof}

\begin{rmk}\label{categorical S1 def}
To add to the confusion, there is another possible notion of what one might mean by an $S^1$-action on a category.  Namely, any topological group $G$ may be made internal to $\infty$-categories, and thus one may formulate in entirely abstract terms what it means for $G$ to act on an $\infty$-category.\footnote{E.g. we may define a category $\cat{B}G$ consisting of one object whose morphisms are $G$ and define a $G$-action on $\cat{C}$ to be a functor $\cat{B}G \rightarrow \cat{Cat}_\infty$ sending the unique object to $\cat{C}$ (the ``straightened'' version), or we may define a $G$-module category to be a coCartesian fibration $\cat{C}^\otimes \rightarrow \cat{B}G$ (the ``unstraightened'' version).}  In the $k$-linear setting these notions are equivalent.  In the case where $G = S^1$, such an action is given by a map $\bZ \rightarrow HH^\bullet(\cat{C})$ to the Hochschild cohomology complex, which is equivalent to a $\cO(\bG_m)$-linear structure on $\cat{C}$.  We refer the reader to Section 6.1 of \cite{toly} for details.
\end{rmk}

\medskip

We now make calculations on the equivariantized side, beginning with the 1-affine $BG = B^2\bG_a$.  We choose $u \in \cO(B^2\bG_a)$, thus identifying once and for all $k[u] \simeq \cO(B^2\bG_a)$, and let $\bB = \{0\} \times_{\A^1} \{0\}$.
\begin{prop}\label{BGa calc}
The stack $B^2\bG_a$ is 1-affine and self-dual, thus the equivariantization and de-equivariantization functors are equivalences.  Furthermore, letting $p: \Spec k \rightarrow B^2\bG_a$, we have monoidal identifications
$$\begin{tikzcd}[column sep=huge]
& \cat{Vect} & \\
\QCoh(B^2\bG_a) \arrow[ur, "p^*"] \arrow[r, "\simeq"] & \QCoh(\bB) \arrow[u] \arrow[r, "{- \tens{\cO_{\bB}} k}"', "\simeq"] &  \Mod_{u\mathrm{-tors}}(k[u]) \arrow[ul, "{ R\Hom_{k[u]}(k, -)}"'] 
\end{tikzcd}$$
for the tensor monoidal structure on $\QCoh(B^2\bG_a)$, the convolution structure on $\QCoh(\bB)$, and the $!$-tensor product on $\Mod_{u\mathrm{-tors}}(k[u])$, the full subcategory of locally $u$-torsion modules.  In particular, $\QCoh(B^2\bG_a)^\omega  \subsetneq \Perf(B^2\bG_a)$, and $B^2\bG_a$ is not perfect.\footnote{E.g. the module $k[u,u^{-1}]/uk[u] \in \Mod_u(\Sym V^*[-2])$ is not compact but is perfect.}  Similar statements hold for the group $G = B\bG_a \rtimes \bG_m$.
\end{prop}
\begin{proof}
Let us, for the reader's sake, give a sense of how the calculation of $\QCoh(B^2\bG_a)$ goes (following \cite{loops and conns}).  First, by the usual descent arguments, one has that $\QCoh(B^2\bG_a)$ is equivalent to comodules for the coalgebra $\cO(B\bG_a) \simeq \Ext^\bullet_{\bG_a}(k, k) \simeq C^\bullet(S^1; k)$, with comultiplication given by pullback along group multiplication.  By dualizing, this is equivalent to modules for $C_\bullet(S^1; k)$.  By Koszul duality, this is equivalent to locally nilpotent modules for $C^\bullet(BS^1; k)$ (with multiplication by cup product), i.e. the augmentation module $k \in \Mod(C^\bullet(BS^1; k))$ is a compact generator, and $\End_{C^\bullet(BS^1; k)}(k, k) \simeq C_\bullet(S^1; k)$.  
\end{proof}

The case of $BG = BS^1$ is similar (see Lemma 3.10 of \cite{loops and conns}), and in fact there is no difference between the categories of quasi-coherent sheaves on $B^2\bG_a$ and $BS^1$.  In particular, $B^2\bG_a$ is the 1-affinization of $BS^1$.
%In particular, the affinization map $S^1 \rightarrow B\bG_a$ induces the fully faithful map $\QCoh(B\bG_a) \rightarrow \QCoh(S^1)$ in Example \ref{S1 not 1affine}, while the affinization $BS^1 \rightarrow B^2\bG_a$ induces an equivalence $\QCoh(B^2\bG_a) \simeq \QCoh(BS^1)$, i.e. $B\bG_a$ is the 0-affinization of $S^1$ while $B^2\bG_a$ is the 1-affinization (and 0-affinization) of $BS^1$.\todo{work back in}
\begin{prop}\label{BS1 1-affinization}
The affinization map $a: BS^1 \rightarrow B^2\bG_a$ induces a monoidal equivalence
$$a^*: \begin{tikzcd} \QCoh(B^2\bG_a)  \arrow[r, "\simeq"] & \QCoh(BS^1).\end{tikzcd}$$
\end{prop}
%\begin{proof}
%On the other hand, by Theorem 2.5.7 of \cite{1affine} the stack $B^2\bG_a$ is 1-affine (and one can prove a similar result for the graded version $B^2\bG_a \rtimes \bG_m$), i.e. one can pass freely between categories with a $B\bG_a$-action and $B\bG_a$-equivariant categories (see, for example, Theorem 2.3.7 of \cite{beraldo}).\todo{write}
%\end{proof}

\medskip

Passing through Cartier duality, we have equivalences of monoidal categories on the de-equivariantized side
$$\QCoh(B\bG_a) \simeq \QCoh(\wh{\bG}_a), \;\;\;\;\;\;\;\;\;\; \QCoh(S^1) = \QCoh(B\bZ) \simeq \QCoh(B\bG_m).$$
These are evidently different; letting $a: S^1 \rightarrow B\bG_a$ denote the affinization map and $\iota: \wh{\bG}_a \hookrightarrow \bG_m$, we have an identification of the adjoint pair $(a^*, a_*)$ with the adjoint pair $(\iota_*, \iota^!)$ under Cartier duality.  Since the corresponding categories on the equivariantized side are equivalent by Proposition \ref{BS1 1-affinization}, we may conclude that $BS^1$ is not 1-affine and that equivariantization and de-equivariantization are not inverse equivalences.
\begin{exmp}[$BS^1$ is not 1-affine]\label{S1 not 1affine}
Consider the regular $\QCoh(S^1)$-comodule category $\QCoh(S^1)$.  There is a tautological identification of its invariants with the augmentation $\QCoh(BS^1)$-module category $\cat{Vect}_k$, and by Proposition \ref{BS1 1-affinization} we have $\cat{Vect}_k \otimes_{\QCoh(BS^1)} \cat{Vect}_k \simeq \QCoh(B\bG_a)$, thus the functor
$$\QCoh(S^1)^{\QCoh(S^1)} \tens{\QCoh(BS^1)} \cat{Vect}_k \simeq \cat{Vect}_k \tens{\QCoh(BS^1)} \cat{Vect}_k \simeq \QCoh(B\bG_a) \hookrightarrow \QCoh(S^1)$$
is fully faithful and induced by pullback along the affinization map $a: S^1 \rightarrow B\bG_a$, with essential image the full subcategory of locally nilpotent $k\bZ$-modules.  In particular, this functor is not an equivalence, so $BS^1$ is not 1-affine.%In particular, $\QCoh(B\bG_a)$ is compactly generated by the trivial representation, so the essential image of the functor must be as well.
\end{exmp}

\medskip

The full faithfulness of the functor in the above examples holds in greater generality; we call the essential image of the functor in the next proposition the full subcategory of \emph{$S^1$-invariant} (or $S^1$-equivariantizable) \emph{objects}.  We note that this proposition has a renormalized counterpart in Proposition \ref{equivariant ff ren}.
\begin{prop}\label{equivariant ff}
Let $\cat{C}$ be a $\QCoh(S^1)$-comodule category.  Then,
$$\cat{C}^{\QCoh(S^1)} \otimes_{\QCoh(BS^1)} \cat{Vect} \hookrightarrow \cat{C}$$
is fully faithful.
\end{prop}
\begin{proof}
The case of the right regular representation $\cat{C} = \QCoh(S^1)$ was established in Example \ref{S1 not 1affine}.  In general, tensor the above case with $\cat{C}$:
$$\cat{C} \otimes_{\QCoh(S^1)} \QCoh(S^1)^{S^1} \otimes_{\QCoh(BS^1)} \cat{Vect} \rightarrow \cat{C}.$$
The operation $\cat{C} \otimes_{\QCoh(S^1)} -$ can be expressed as a limit by passing to right adjoints, thus commutes with the $S^1$-invariants operation and preserves full faithfulness.
\end{proof}

\medskip

\subsubsection{Renormalization of circle and $B\bG_a$-actions}\label{S1 ren sec}

We now discuss the need to renormalize the invariants operation (as raised in \cite{toly indcoh}).  We summarize the discussion of this subsection as follows.
\begin{enumerate}
\item A $\QCoh(S^1)$-module category $\cat{C}$ equivariantizes to a cyclic deformation $\cat{C}^{S^1}$ over the 2-shifted formal affine line $\wh{\A}^1[2]$, whose special fiber is the full subcategory of $\cat{C}$ generated by $S^1$-invariant objects.
\item A $\QCoh(B\bG_a)$-module category $\cat{C}$ equivariantizes to a cyclic deformation $\cat{C}^{B\bG_a}$ over the 2-shifted formal affine line $\wh{\A}^1[2]$, whose special fiber recovers $\cat{C}$.
\item A $\QCoh(B\bG_a \rtimes \bG_m)$-module category $\cat{C}$ equivariantizes to a cyclic deformation $\cat{C}^{B\bG_a \rtimes \bG_m}$ over the graded formal affine line $\wh{\A}^1/\bG_m$, whose special fiber recovers $\cat{C}$.
\end{enumerate}
The generic fiber in all the cases above vanishes.  The goal of renormalization is to replace the formal affine line with the usual affine line, thus allowing for a meaningful generic fiber.

\medskip

We begin with some generalities.  Let $G$ be a group prestack acting on a prestack $X$.  Then, there is a tautological $\Perf(G)$-coaction on $\Perf(X)$, as well as a $\QCoh(G)$-coaction on $\QCoh(X)$.  Furthermore, we have
$$\Perf(X)^{\Perf(G)} \simeq \Perf(X/G) \subset \QCoh(X/G) \simeq \QCoh(X)^{\QCoh(G)}.$$
If $X/G$ is a perfect stack in the sense of \cite{BFN}, i.e. $\QCoh(X/G)$ is compactly generated by $\Perf(X/G)$, then taking $G$-invariants of small categories and then ind-completing is the same as taking $G$-invaraints of large categories.  By Corollary 3.22 of \emph{op. cit.} this occurs, for example, in characteristic zero when $X$ is quasi-projective and $G$ is an affine algebraic group.  However, our groups $G = S^1, B\bG_a$ take us outside of this setting, since $BG$ is not perfect as we have seen in Example \ref{S1 not 1affine}. 
Since the stacks $BS^1$ and $B^2\bG_a$ are not perfect, the following equivariantizations give different answers:
\begin{enumerate}
\item taking $\QCoh(G)$-invariants of large categories, as we have done in the previous section, and
\item taking $\Perf(G)$-invariants of small categories, and then ind-completing to obtain a large category.
\end{enumerate}

We now take the second approach and define the following renormalized equivariantization.
\begin{defn}\label{renormalized invariants}
Let $G$ be a group object in prestacks.  We will denote the $\QCoh(G)$-invariants of a comodule category by $\cat{C}^G$. Suppose that $\cat{C} = \Ind(\cat{C}_0)$ is a compactly generated category such that $\cat{C}_0$ has a $\Perf(G)$-coaction.  We define the \emph{category of compactly-renormalized $G$-invariants} to be the $\IndPerf(BG)$-module category
$$\cat{C}^{\omega G} := \Ind(\cat{C}_0^{\Perf(G)}).$$
There is a tautological functor $\cat{C}^{\omega G} \rightarrow \cat{C}^G$ induced by the functor $\cat{C}_0^{\Perf(G)} \rightarrow \cat{C}^{\QCoh(G)}$.
\end{defn}

\medskip

Then, the compactly renormalized equivariantization in our cases of interest $G = S^1, B\bG_a, B\bG_a \rtimes \bG_m$ take on the following forms. 
\begin{enumerate}
\item An $\IndPerf(S^1)$-module category $\cat{C}$ compactly equivariantizes to a cyclic deformation $\cat{C}^{\omega S^1}$ over the 2-shifted affine line $\A^1[2]$.
\item An $\IndPerf(B\bG_a) = \QCoh(B\bG_a)$-module category $\cat{C}$ compactly equivariantizes to a cyclic deformation $\cat{C}^{\omega B\bG_a}$ over the 2-shifted affine line $\A^1[2]$.
\item An $\IndPerf(B\bG_a \rtimes \bG_m) = \QCoh(B\bG_a \rtimes \bG_m)$-module category $\cat{C}$ compactly equivariantizes to a cyclic deformation $\cat{C}^{\omega B\bG_a \rtimes \bG_m}$ over the graded affine line $\A^1/\bG_m$.
\end{enumerate}

In light of this calculation we can make the following definition.
\begin{defn}\label{defn tate}
Let $\cat{C} = \Ind(\cat{C}_0)$ be a compactly generated category with a $\QCoh(S^1)$-action restricting to compact objects.  We define the \emph{Tate construction} to be the 2-periodic category
$$\cat{C}^{\Tate} := \cat{C}^{\omega S^1} \otimes_{\Mod(k[u])} \Mod(k[u, u^{-1}]).$$
One can make a similar definition for a category with a $\QCoh(B\bG_a)$-action.  If the category has a $\QCoh(B\bG_a \rtimes \bG_m)$-action, we may define the \emph{graded Tate construction}
$$\cat{C}^{\Tate \rtimes \bG_m} := \cat{C}^{\omega B\bG_a \rtimes \bG_m} \otimes_{\QCoh(\A^1/\bG_m)} \QCoh(\bG_m/\bG_m)$$
which is a (non-periodic) $k$-linear category.
\end{defn}

\medskip

In \cite{BCHN2} we show that the compactly renormalized invariant category is compatible with the usual invariants at the $u=0$ fiber.
%Finally, we state (without proof) a renormalized analogue of Proposition \ref{equivariant ff}; in particular the notion of $S^1$-equivariantizability is independent of whether we take renormalized invariants or not.
\begin{prop}\label{equivariant ff ren}
Let $\cat{C}_0$ be a small stable $k$-linear $\infty$-category with a $\Perf(S^1)$-coaction, and let $\cat{C} = \Ind(\cat{C}_0)$.  The canonical functor
$$\begin{tikzcd} \cat{C}^{\omega S^1} \otimes_{k[u]} \cat{Vect}_k \arrow[r, "\simeq"] & \cat{C}^{S^1} \otimes_{k[u]} \cat{Vect}_k\end{tikzcd}$$
is an equivalence.  In particular, by Proposition \ref{equivariant ff} the functor
$$\begin{tikzcd}
\cat{C}^{\omega S^1} \otimes_{k[u]} \cat{Vect} \arrow[r, hook] & \cat{C}\end{tikzcd}$$
is fully faithful.  Similar statements hold for the compactly renormalized $B\bG_a$ and $B\bG_a \rtimes \bG_m$-invariants as well.  %Likewise, if $\cat{C}_0$ has a $\Perf(B\bG_a)$-coaction, then the de-\hspace{0pt}equivariantization functor
%$$\cat{C}^{\omega B\bG_a} \otimes_{k[u]} \cat{Vect} \simeq \cat{C}$$
%is an equivalence.
\end{prop}

\medskip

\subsubsection{Koszul duality and graded lifts}\label{graded lifts} As we have seen in Definition \ref{defn tate}, the Tate construction gives rise to a 2-periodic category, while the graded Tate construction provides a lift to a non-periodic category.  These phenomena are typical to Koszul duality \cite{BGS, BY} and we now give formalizations of these notions.

\begin{defn}\label{koszul duality def}
Let $\cat{C}$ be a $k$-linear category; a \emph{graded lift} of $\cat{C}$ is a $\QCoh(B\bG_m)$-module category $\cat{C}_{gr}$ and an equivalence $\cat{C} \simeq \cat{C}_{gr} \otimes_{\QCoh(B\bG_m)} \cat{Vect}_k$.  The category $\QCoh(B\bG_m)$ of $\bZ$-graded vector spaces has an automorphism by \emph{Tate shearing}:
$$M^\shear := \bigoplus_{n \in \bZ} M_n[-2n], \;\;\;\;\;\;\;\;\;\; M^{\unshear} := \bigoplus_{n \in \bZ} M_n[2n].$$
For any $\QCoh(B\bG_m)$-module category $\cat{C}$ we denote by $\cat{C}^\shear$ (resp. $\cat{C}^\unshear$) the $\QCoh(B\bG_m)$-module category where $\QCoh(B\bG_m)$ now acts through the Tate shearing (resp. unshearing).  For any $\QCoh(\bG_m)$-module category $\cat{C}$, we denote by $\cat{C}^\shear$ (resp. $\cat{C}^\unshear$) the category obtained by equivariantizing, shearing (resp. unshearing), and de-equivariantizing.  Finally, we say two categories are \emph{Koszul equivalent} 
$$\cat{C} \kos \cat{D}$$
if they admit graded lifts $\cat{C}_{gr}, \cat{D}_{gr}$ and an equivalence of categories $\cat{C}_{gr} \simeq \cat{D}^{\shear}_{gr}$.
\end{defn}

It is sometimes convenient to work with the a coarser notion of Koszul equivalence induced by the refined version above by passing to 2-periodic categories.
\begin{defn}\label{2per}
Let $\cat{C}$ be a $k$-linear category; the \emph{2-periodicization} of $\cat{C}$ is the 2-periodic category
$$\cat{C}^{per} := \cat{C} \otimes_{\Mod(k)} \Mod(k[u,u^{-1}])$$
where $|u| = 2$.  There is a canonical functor $(-)^{per}: \cat{C} \rightarrow \cat{C}^{per}$ such that $\Hom_{\cat{C}^{per}}(X^{per}, Y^{per})$ is the 2-periodicization of the chain complex $\Hom_{\cat{C}}(X, Y)$.
\end{defn}
Since the Tate shearing is the identity in 2-periodic categories, it becomes clear that a Koszul equivalence induces an equivalence on 2-periodic categories
$$\cat{C}^{per} \simeq \cat{D}^{per}.$$

\begin{exmp}
The 2-periodicization operation can introduce many new objects into a category.  For example, take the category $\rnD(BT)$ of ind-coherent $\cD$-modules on a torus, which is equivalent to the category $\Mod(H^\bullet(BT; k))$ of modules for the cohomology of $BT$, in turn equivalent to $\Mod(\Sym \mf{t}^*[-2])$.  Roughly speaking, this category only has a skyscraper object at 0.  However, the 2-periodicization is equivalent to the category $\Mod((\Sym \mf{t}^*[-2])[u,u^{-1}])$, which has skyscraper objects for every point $t \in \mf{t}$ by the evaluation maps $u^{-1}\mathrm{ev}_t: \mf{t}^*[-2]u^{-1} \simeq \mf{t}^* \rightarrow k$.  

We also observe that the standard graded lift of this category gives $\mf{t}^*$ weight $-1$, which the Tate shearing moves from degree 2 to degree 0; by contrast, the Rees parameter $t$ has weight 1, which the Tate shearing turns into the degree 2 $S^1$-deformation parameter $u$.

\end{exmp}

\medskip

\subsubsection{Examples} We present the following very computable examples of $S^1$-equivariant categories.  Namely, let $G$ be an abelian affine algebraic (possibly formal) group with Cartier dual $\wh{G}$.  Then, passing through Cartier duality, we have identifications
$$\Coh(\cL(BG)) = \Coh(G \times BG) \simeq \Coh(G \times \wh{G}).$$
Furthermore, Cartier duality provides a monoidal equivalence
$$(\QCoh(B\bZ), \circ) \simeq (\QCoh(\bG_m), \otimes)$$
i.e. an $S^1$-action on a category may be described as a $\QCoh(\bG_m)$-linear structure.  In this setting, both arise geometrically, i.e. come from pullback along the canonical pairing
$$G \times \wh{G} \rightarrow \bG_m, \;\;\;\;\;\;\;\;\;\; (g, \chi) \mapsto \chi(g)$$
Applying the discusison of \cite[\textsection 3]{toly}, we have
$$\Coh(\cL(BG))^{S^1} \simeq \Coh((G \times \wh{G}) \times_{\bG_m} \{1\}), \;\;\;\;\;\;\;\;\;\; \Coh(\cL(BG))^{\Tate} \simeq \mathrm{MF}(G \times \wh{G}, f)$$
where $\mathrm{MF}$ denotes the 2-periodic category of matrix factorizations.  We work out the above in two specific examples.

\begin{exmp}\label{exmp BT}
Let us take $X = BT$ where $T$ is an algebraic torus; then $\cL(BT) \simeq T \times BT$.  We have an identification of the category
$$\Coh(\cL(BT)) \simeq \Coh(T \times BT) \simeq \Coh(T \times X^\bullet(T)) \simeq  \bigoplus_{\lambda \in X^\bullet(T)} \Coh(T),$$
% \;\;\;\;\;\;\;\;\;\;\; \Coh(\wh{\cL}(BT)) \simeq \bigoplus_{\lambda \in X^\bullet(T)} \Coh(\wh{T}) \otimes k_\lambda$$
i.e. a decomposition into isotypic components for $\Rep(T)$.  % where $\wh{T}$ is the completion at the identity.  
The $S^1$-action is given (see Remark \ref{categorical S1 def}) by multiplication by $t^\lambda$ at $(t, \lambda) \in T \times X^\bullet(T)$, i.e. multiplication by the monomial $t^\lambda \in \cO(T)$ on $T \times \{\lambda\}$.
%where $\epsilon$ is a degree $-1$ variable with differential $d(\epsilon) = 1 - t^\lambda$; when $\lambda \ne 1$ this dg algebra is quasi-isomorphic to $\cO(T)/(1 - t^\lambda)$.  %In this example the $S^1$-action on each block may be described by the automorphism of the identity functor (see Lemma 6.1.1 of \cite{toly}) given by multiplication by $t^\lambda \in \cO(T)$, and 
We now compute the $S^1$-equivariant, Tate localization, and $S^1$-invariant objects (see Proposition \ref{equivariant ff}).  Let $T_\lambda$ denote the derived zero locus of the polynomial $1 - t^\lambda$.  Explicitly,
$$T_\lambda = (\cO(T)[\epsilon], d(\epsilon) = 1 - t^\lambda) \simeq \begin{cases} \cO(T)[\epsilon] & \lambda = 1, \\ \cO(T)/t^\lambda - 1 & \lambda \ne 1. \end{cases}$$
By Remark \ref{categorical S1 def} the $S^1$-invariants for an $S^1$-action given by automorphism $\alpha$ may be described as imposing the equation $1 - \alpha$ in a derived way.  Thus, we have
$$\Coh(\cL(BT))^{S^1} \simeq \bigoplus_{\lambda \in X^\bullet(T)} \Coh(T_\lambda).$$
Koszul dually \cite[Con. 3.1.5]{toly}, we may view $\Coh(\cL(BT))^{S^1}$ as a $k[u]$-module category where $|u| = 2$ acts by cohomological operators.  The derived zero locus $T_\lambda$ is smooth when $\lambda \ne 1$, thus any coherent sheaf has a finite resolution, thus $u$ acts by torsion on $\Coh(T_\lambda)$, i.e. $\Coh(T_\lambda) \otimes_{k[u]} k(u) \simeq 0$.  When $\lambda = 1$ then $u$ acts freely, and we have (see also \cite[Prop. 3.4.1]{toly}):
$$\Coh(\cL(BT))^{\Tate} \simeq \Coh(T) \otimes \Mod(k[u, u^{-1}]).$$
In particular, we have
$$\Coh(\wh{\cL}(BT))^{\Tate} = \Coh(\cL^u(BT))^{\Tate} \subsetneq \Coh(\cL(BT))^{\Tate}$$
where the first equality arises since $T$ has no unipotent elements, i.e. $\cL^u(BT) = \wh{\cL}(BT)$.  Specializing at $u=0$ we have instead \cite[Cor. 3.2.4]{toly}
$$\Coh(\cL(BT))^{S^1} \otimes_{k[u]} k \simeq \bigoplus_{\lambda \in X^\bullet(T)} \Coh_{T_\lambda}(T).$$
In particular, not every object of $\Coh(\cL(BT))$ is $S^1$-equivariantiazble.
\end{exmp}

\begin{exmp}\label{exmp BGa}
We consider a somewhat orthogonal example.  Consider $X = B\bG_a$.  Then, we have via Cartier duality
$$\Coh(\cL(B\bG_a)) \simeq \Coh(\bG_a \times B\bG_a) \simeq \Coh(\bG_a \times \wh{\bG}_a).$$
Fix coordinates $x, y$ on the two $\bG_a$ respectively; we view $x$ as coming from the action of $\cO(\bG_a)$, and $y$ as the action of $1 \in \mathrm{Lie}(\bG_a)$, i.e. the action map on  $V \in \Coh(B\bG_a)$ is $t \mapsto \exp(ty)$.  Then, the $S^1$-action is given by the automorphism of the identity $\exp(xy)$.  Taking $S^1$-invariants imposes the equation $\exp(xy) = 1$, i.e. $xy = 0$ (since $y$ is in an infinitesimal neighborhood of 0, we can ignore other logarithms of 1), in a derived way.  We have
$$\Coh(\cL(B\bG_a))^{S^1} \simeq \Mod_{\mathrm{f.g.}, y\mh\mathrm{nil}}(k[x, y]/xy).$$
The Tate localization is the category of matrix factorizations on $k[x, y]$ for the equation $xy = 0$ (where $y$ acts nilpotently).  We have an equivalences
$$\Coh(\wh{\cL}(B\bG_a))^{\Tate} \simeq \Coh(\cL^u(B\bG_a))^{\Tate} = \Coh(\cL(B\bG_a))^{\Tate} \simeq \Mod(k(u))$$
where the first isomorphism arises since the equation $xy = 0$ has singular locus at the origin (thus completing at $x=0$ doesn't change the category of matrix factorizations), and the second arises geometrically (i.e. $\cL^u(B\bG_a) = \cL(B\bG_a)$ since $\bG_a$ consists only of unipotent elements).  The rightmost identification of the categories is a calculation, see e.g. \cite[Sec. 5.3]{toby}.  Alternatively, one can pass through Koszul duality (Theorem \ref{kd stacks thm}), where $\Coh(\wh{\cL}(B\bG_a))^{\Tate} \simeq \cD(B\bG_a) \otimes \Mod(k(u)) \simeq \Mod(k(u))$.  All objects of $\Coh(\cL(B\bG_a))$ are $S^1$-invariant, i.e.
$$\Coh(\cL(B\bG_a))^{S^1} \otimes_{k[u]} k = \Coh(\cL(B\bG_a)).$$
Since all loops are unipotent, the $S^1$-action factors through a $B\bG_a$-action, which is 1-affine.  This example is easily generalized to the case where $X = BV$, where $V$ is a commutative additive group.
\end{exmp}

\subsection{Koszul duality and equivariant localization}

We now turn our attention to various technical details that arise in the discussion of Koszul duality and equivariant localization.

\medskip

\subsubsection{Renormalized categories}\label{sec ren}

The category $\cD(X)$ of $\D$-modules on a smooth stack $X$ is defined via descent. Let $F\cD(X)$ denote the category of filtered $\cD$-modules on a stack $X$, also defined via descent.  We let $F\cD_c(X)$ denote the full subcategory of coherent $\cD$-modules on $X$ (with good filtrations).  Note that these are not the compact objects in $F\cD(X)$ in general; we let $F\breve{\cD}(X) = \Ind(F\cD_c(X))$ denote the category of \emph{ind-coherent filtered $D$-modules}, i.e. we renormalize the category with respect to coherent $D$-modules.

\medskip

We require a similar renormalization for the odd tangent bundle; see Section 2 of \cite{CD} for a discussion.  For any closed substack $Z \subset X$ we define the category $\hCoh(\wh{X}_Z)$ (sometimes denoted $\hCoh_Z(X)$) of \emph{ind-continuous coherent sheaves} to be the full subcategory consisting of objects $\cF \in \IndCoh(\wh{X}_Z) \simeq \IndCoh_Z(X)$ such that $\cF$ is $t$-bounded and almost $!$-perfect, i.e. such that for any closed substack $i: Z' \hookrightarrow X$ set-theoretically supported on $Z \subset X$, the $!$-restriction $i^! \cF$ has coherent cohomology.  This category has the following properties (see \cite{koszul}).
\begin{enumerate}[leftmargin=5ex]
\item It contains the usual category of coherent sheaves supported along $Z \subset X$, i.e. $\Coh_Z(X) \subset \wh{\Coh}_Z(X)$.  In particular, we have $\Coh(\wh{\cL} X) = \Coh_X(\cL X) \subset \hCoh(\wh{\cL} X).$
\item Letting $\wh{i}: \wh{X}_Z \hookrightarrow X$ be the inclusion of the formal completion, the functor $\wh{i}^!: \IndCoh(X) \rightarrow \IndCoh(\wh{X}_Z)$ takes $\Coh(X)$ to $\wh{\Coh}(\wh{X}_Z)$.  In particular, taking $\wh{z}: \wh{\cL} X \rightarrow \cL X$ to be the inclusion of formal loops, the functor $\wh{z}^!: \IndCoh(\cL X) \rightarrow \IndCoh(\wh{\cL} X)$ restricts to a functor $\Coh(\cL X) \rightarrow \hCoh(\wh{\cL} X)$.
\item If $Z \subset X$ is defined by a nilpotent ideal, then $\Coh(\wh{X}_Z) = \wh{\Coh}(\wh{X}_Z) = \Coh(X)$.  In particular, if $X$ is a scheme, then $\Coh(\wh{\cL} X) = \hCoh(\wh{\cL} X)  = \Coh(\cL X)$.
\item If $Z \subset X$ is a derived local complete intersection, then $\cF \in \wh{\Coh}(\wh{X}_Z)$ if and only if its $!$-restriction to $Z$ is in $\Coh(Z)$.  In particular, if $p: U \rightarrow X$ is a smooth atlas, we have that $\cF \in \hCoh(\wh{\cL} X)$ if and only if $\wh{\cL}p^! \cF \in \Coh(\cL U)$.
\end{enumerate}
We denote its ind-completion by $\wh{\IndCoh}(\wh{X}_Z) := \Ind(\wh{\Coh}(\wh{X}_Z))$.    This renormalization is necessary to have a good restriction functor to completions at parameters.  In a general setting, if $F: \cat{C} \rightarrow \cat{D}$ is a colimit-preserving functor which preserves compact objects (i.e. a left adjoint with a colimit-preserving right adjoint) between compactly generated categories, then for $X \in \cat{C}$ compact we have a commuting diagram:
$$\begin{tikzcd}
\End(X)^{op}\dmod \arrow[r, hook] \arrow[d, "{- \otimes_{\End(X)} \End(F(X))}"'] & \cat{C} \arrow[d, "F"] \\
\End(F(X))^{op}\dmod \arrow[r, hook] & \cat{D}.
\end{tikzcd}$$
Commutativity follows by checking the tautological commutativity of right adjoints, while compactness of $X$ guarantees that the left adjoint to $\Hom(X, -)$ is fully faithful\footnote{I.e. the unit of the adjunction $M \rightarrow \Hom(X, X \otimes_{\End(X)} M)$ is an equivalence when $\Hom(X, -)$ commutes with colimits.} (and similarly for $F(X)$).  Unfortunately, the $!$-restriction on ind-coherent sheaves is a right adjoint and does not preserve compact objects.  On the other hand, under renormalization $\widehat{\ell}_\alpha^!: \IndCoh(\cL(X/G)) \rightarrow \wh{\IndCoh}(\LLf_\alpha(X(\alpha)/G(\alpha)))$ becomes a \emph{left adjoint}, and compact objects by construction.  In particular, we have a commuting square:
$$\begin{tikzcd}
\Mod(\cH) \simeq \langle \cS \rangle \arrow[d, "{- \otimes_{\cH} \cH(\alpha)}"'] \arrow[r, hook] & \IndCoh(\cL(X/G)) \arrow[d, "{\wh{\ell}_\alpha^!}"] \\
\Mod(\cH(\alpha)) \simeq \langle \cS(\alpha) \rangle \arrow[r, hook] & \wh{\IndCoh}(\LLf_\alpha(X(\alpha)/G(\alpha))).
\end{tikzcd}$$

\medskip

%\subsection{Equivariant localization}\label{sec eqloc}

%In the previous section we described a Koszul duality identifying $S^1$-equivariant sheaves on formal loop spaces $\wh{\cL} X$ with $\cD$-modules on $X$.  In this section, we will see how equivariant localization describes sheaves on the loop space of a global quotient stack $X/G$ in a similar fashion.

%\medskip

\subsubsection{Homotopy fixed points}\label{sec homotopy fixed}

Our equivariant localization statement in Theorem \ref{LocThm} uses a notion of homotopy fixed points; we give a brief overview of this notion and refer the reader to the appendix of \cite{AKLPR} for details.
\begin{defn}
Let $G$ be a group prestack acting on a prestack $X$.  We define the \emph{homotopy fixed points} $X^{G}$ by the (derived) fiber product
$$\begin{tikzcd}
X^G \arrow[r] \arrow[d] & \Map(BG, X/G) \arrow[d]  \\
\{\mathrm{id}_{BG}\} \arrow[r] & \Map(BG, BG).
\end{tikzcd}$$
In other words, $X^G$ is the prestack of sections of the map $X/G \rightarrow BG$.
\end{defn}

We discuss examples of this construction, since it has a different flavor for different inputs $G$.
\begin{enumerate}[leftmargin=5ex]
%\item When $G$ acts trivially on $X$, this is just $\Map(BG, X)$.  When $G = \bZ$ (where $\bZ$ is viewed as a locally constant stack, and in particular is not algebraic), this is just the derived loop space $\cL X$.  When $G$ is affine algebraic, $\Map(BG, X)^{cl} = X$.
\item When $G$ is linearly reductive and $X$ a locally Noetherian Artin stack, then the cotangent complex of $X^G$ has the same lower bound on Tor amplitude (i.e. ``level of singularity'') as the cotangent complex of $X$ by Corollary A.35 of \cite{AKLPR}. In particular, if $X$ is a smooth scheme, then $X^G \subset X$ is a smooth closed subscheme by Proposition A.23 of \emph{op. cit.} and thus $X^G$ is the classical $G$-fixed points of $X$.  This observation, combined with the observation that $(-)^{hG}$ commutes with Zariski localization and fiber products, gives a method for computing homotopy fixed points for derived schemes with given presentations.
%\item By the above, if $G$ is linearly reductive acting trivially on a locally Noetherian Artin stack $X$, then $X^G = X$.  In particular, for the trivial $G$-action if $X$ is classical then $X^G$ is classical.
\item If $G$ is not reductive, $X^G$ may have derived structure.  For example, taking $G = \bG_a$ acting trivially on $X$, we have $X^G = \cL^u X$ is the unipotent loop space, which is just the usual loop space $\cL X$ when $X$ is a scheme.
\item Assuming that $G$ is reductive, but without the assumption that $X$ is smooth, it is possible for $X^G$ to have nontrivial derived structure.  For example, the nilpotent cone $\cN = \mf{g} \times_{\mf{g}//G} \{0\}$ is a non-derived complete intersection.  However, taking $G$-invariants we have $\cN^G = \mf{z}(\mf{g}) \times_{\mf{g}//G} \{0\}$, which has derived structure unless $\mf{g}$ is commutative, i.e. $G$ is a torus.  When $G$ is a torus, one can argue that $X^G$ is classical if $X$ is quasi-smooth and classical.\footnote{One may adapt the argument in Theorem \ref{LocThm} to reduce to the case where $X$ is a fiber product of smooth $G$-schemes, which one may then reduce to the case of a derived intersection of smooth $G$-schemes, and then argue via tangent spaces.}
\item When $G$ is a topological group acting trivially on $X$, the homotopy fixed points $X^G = \Map(BG, X) = \Loc_G(X)$ is the derived moduli stack of $G$-local systems on $X$.  This stack often has nontrivial derived structure, even when $X$ does not.
\item When $G = \bZ$ acts on $X$ via an automorphism $\phi: X \rightarrow X$, then we obtain the derived fixed points $\cL_\phi X$ of $\phi$ on $X$.  In particular, even when $X$ is a smooth scheme this will be derived unless the fixed points have dimension 0.
\item More generally, if $G$ is a discrete cyclic (thus commutative) group with generator $\eta \in G$, then there is a canonical equivalence $X^G \simeq \cL_\eta(X/G)$, e.g. by observing that the outer and left squares are Cartesian in the diagram:
$$\begin{tikzcd}
X^G \arrow[r] \arrow[d] & \Map(BG, X/G) \arrow[r] \arrow[d] & \cL(X/G) \arrow[d] \\
\{e\} \arrow[r] & \Hom(BG, BG) \arrow[r] & \cL(BG)
\end{tikzcd}$$
and noting that the bottom composition is exactly the map $\{\eta\} \rightarrow G/G$.
\item Suppose $G = \bZ/n\bZ$ acts on a scheme $X$ trivially.  Then, $X^G$ is the relative kernel (over $X$) of the map $\bT_X[-1] \rightarrow \bT_X[-1]$ induced by the degree $n$ map on $S^1$, i.e. multiplication by $n$, which is an isomorphism if $n$ is invertible in $k$.  Note that $G = \bZ/n\bZ$ is linearly reductive over $k$ if and only if $n$ is invertible in $k$.
\end{enumerate}

\medskip

\subsubsection{Central shifting and twisting}\label{sec central twist}

The localization maps in Theorem \ref{LocThm} reduce the study of the loop space $\cL(X/G)$ over $[z] \in G//G$ to the study of $\cL(X(\alpha)/G(\alpha))$ over $[\alpha] \in G(\alpha)//G(\alpha)$.  The advantage of the latter is that $\alpha \in G(\alpha)$ is central and acts on $X(\alpha)$ trivially; thus we can perform a central shifting that translates $\alpha$ to the identity.  This shifting is \emph{not} $S^1$-equivariant for the loop rotations.\footnote{For example, the loop rotation on $\cL(BG) = G/G$ acts over $g \in G/G$ via the automorphism which conjuates by $g$, and this is not preserved by central shifting.}  Our goal in this section is to describe this shifting and a twisted $S^1$-action for which it is equivariant.

\medskip

We begin by describing the shifting.  The localization maps constructed in Theorem \ref{LocThm} factor into two steps:
$$\cL(X(\alpha)/G(\alpha)) \longrightarrow \cL(X/G(\alpha)) \longrightarrow \cL(X/G)$$
with the first map exhibiting the localization, and the second map a simple \'{e}tale base change in a neighborhood of $\alpha \in G//G$.  We will assume we have already performed this base change, focusing exclusively on the first map, and so may assume that $\alpha \in Z(G)$ is a central element.  We need to introduce the notion of $\alpha$-trivializations.
\begin{defn}
Let $G$ be a linear algebraic group with reductive neutral component acting on a prestack $X$, and $H \subset G$ be a normal subgroup, and denote $G' = G/H$.  An \emph{$H$-trivialization} of the $G$ action on $X$ is a $G'$-action on $X$ and an equivalence
$$X/G \simeq X/G' \times_{BG'} BG.$$
When $H = \langle \alpha \rangle$ for $\alpha \in Z(G)$ central, we call an $H$-trivialization an \emph{$\alpha$-trivialization}.
\end{defn}

\begin{prop}
Let $G$ act on a derived scheme $X$, and let $A \subset G$ be a central subgroup.  The homotopy fixed points $X^A$ comes equipped with a canonical $A$-trivialization.
\end{prop}
\begin{proof}
First, note that letting $i: BA \rightarrow BG$ denote the map defined by the inclusion of $A$ into $G$, we have a canonical equivalence
$$X^A \simeq \{i\} \times_{\Map(BA, BG)} \Map(BA, X/G).$$
On the right, since $A$ is central and both are reductive, the inclusion $\{i\} \hookrightarrow \Map(BA, BG) \simeq \Hom_{grp}(A, G)/G$ lifts to a map $\{i\}/G \hookrightarrow \Map(BG, BG)$.  Since $A$ acts trivially on $\Map(BA, BG)$, this map also descends to a map $\{i\}/G' \rightarrow \Hom_{grp}(A, G)/G$.  This data furnishes $X^A$ with the desired structure.
\end{proof}

\medskip

Again letting $\alpha \in G$ be central, and taking $A = \langle \alpha \rangle$, suppose that $X$ is a derived scheme equipped with an $A$-trivialization.  We now define a \emph{shift by $\alpha$} map on $\cL(X/G)$
$$sh_\alpha: \cL(X/G) \rightarrow \cL(X/G)$$
compatible with the multiplication by $\alpha$ map $\mu_\alpha: \cL(BG) = G/G \rightarrow \cL(BG) = G/G$ as follows.  % sending $g \mapsto gz = zg$, as follows.  %induces a \emph{shift by $z$} map $sh_z: \cL(X^z/G) \rightarrow \cL(X^z/G)$.  This shifting map is defined as follows.  
The canonical $\alpha$-trivialization $X^\alpha/G \simeq X^\alpha/G' \times_{BG'} BG$ induces a canonical equivalence $\cL(X^\alpha/G) \simeq \cL(X^\alpha/G') \times_{G'/G'} G/G$, and the automorphism is given on the right-hand side by the multiplication by $z$ map on $G/G$ and the identity elsewhere.  In the setting of Theorem~\ref{LocThm}, we have a composition which we call the (unipotent) \emph{$\alpha$-shifted localization map}
$$
s\ell_\alpha^u: \begin{tikzcd}[column sep=large]
\LL^u(X(\alpha)/G(\alpha)) \arrow[r, "{sh_\alpha}", "\simeq"'] &  \LL^u_\alpha(X(\alpha)/G(\alpha)) \arrow[r, "{\ell_\alpha^u}", "\simeq"'] &  \LL_\alpha^u(X/G).
\end{tikzcd}$$

\medskip

The shifting map on the left is \emph{not} $S^1$-equivariant for the loop rotation; we introduce the twisting of the $S^1$ action now.  For $\alpha \in Z(G)$, the multplication by $\alpha$ defines a map of groups $\bZ \times G \rightarrow G$, which gives rise to an action map $S^1 \times BG \rightarrow BG$.  If $G$ acts on an $\alpha$-trivialized scheme $X$, then this $S^1$-action gives rise to an $S^1$-action on $X/G \simeq X/G' \times_{BG'} BG$, where we take $G' = G/\langle \alpha \rangle$ and the trivial $S^1$-action on $X/G'$ and $BG'$.  We call the resulting $S^1$-action the \emph{$\alpha$-twisting $S^1$-action} $\sigma_\alpha$ on $X/G$.  Since the $S^1$-actions $\cL(\sigma_\alpha)$ and the loop rotation $\rho$ commute, we define the \emph{$\alpha$-shifted loop rotation}, denoted $\rho(\alpha)$ or $S^1(\alpha)$, to be the diagonal to the $S^1 \times S^1$-action $\rho \times \cL(\sigma_\alpha)$.  The shifted localization map is equivariant with respect to this twisting.
\begin{corollary}
The shifted localization map defines an equivalence  $$\begin{tikzcd}s\ell^u_\alpha : \LL^u(X(\alpha)/G(\alpha)) \arrow[r, "\simeq"]& \LL_\alpha^u(X/G)\end{tikzcd}$$ which is $S^1$-equivariant with respect to $\rho(\alpha)$ on the source and $\rho$ on the target, and likewise for the shifts of the completed and specialized localization maps $s\ell^\wedge_\alpha$ and $s\ell^\prime_\alpha$.
\end{corollary}

%\subsubsection{Twisted Koszul duality}
\medskip

\subsubsection{Trivial blocks for twisted actions}\label{sec triv block}

In order to apply the Koszul duality discussed in Section \ref{equiv sheaves}, which involves $S^1$-equivariance for the untwisted action, we are interested in identifying a subcategory of sheaves on derived loop spaces over semisimple parameter $\alpha$ on which the $\alpha$-twisting is trivial.  This is useful since the $\rho$ circle action on unipotent loop spaces factors through an action of $B\G_a$, but the twisted $\rho(\alpha)$ action does not (since it has nontrivial semisimple part).  This problem is an obstacle to applying Koszul duality to obtain an identification of $\Coh(\LLf_\alpha(X/G))^{S^1}$ with some kind of category of $\D$-modules as is done in Section \ref{shvs via eq loc}.

\medskip

%This subcategory is given by the \
%We avoid this obstacle by focusing only on the $z$-trivial block.  
To define this subcategory, we give a categorical interpretation of the geometric $\alpha$-twisting $S^1$-action $\sigma(\alpha)$ discussed above.
\begin{defn}\label{z trivial}
Let $G$ be an affine algebraic group, $A \subset Z(G)$ a central subgroup with $G'=G/A$, and $\cat{C}$ a $\QCoh(BG)$-module category.  An \emph{$A$-trivialization} of $\cat{C}$ consists of the data of a $\QCoh(BG')$-module category $\cat{C}'$ along with an identification $\cat{C} \simeq \cat{C}' \otimes_{\QCoh(BG')} \QCoh(BG)$.  We define the subcategory of \emph{$A$-trivial objects} $\cat{C}_A \subset \cat{C}$ to be the essential image of the natural functor $\cat{C}' \rightarrow \cat{C}$.  When $A = \langle \alpha \rangle$ for $\alpha \in Z(G)$, we write $\cat{C}_\alpha$.
\end{defn}

\medskip

Furthermore if $G$ is reductive, then the $\alpha$-trivial objects form a summand (not just a subcategory) of $\cat{C}$.  Since $G$ is reductive, the center is semisimple and letting $X^\bullet(A)$ denote the group of characters of $A$, we have a decomposition $\QCoh(BG) = \Rep(G) = \bigoplus_{\chi \in X^\bullet(A)} \Rep(G)_\chi$, with the trivial isotype corresponding to $\Rep(G')$.  Thus an $\alpha$-trivialization of $\cat{C}$ defines a direct sum decomposition 
$$\cat{C} \simeq \bigoplus_{\chi \in X^\bullet(A)} \cat{C}_\chi \simeq \bigoplus_{\chi \in X^\bullet(A)} \cat{C}' \otimes_{\Rep(G')} \Rep(G)_\chi.$$
We will consider the following $\alpha$-trivialized categories which arise in nature.
\begin{enumerate}
\item If $X/G$ is a global quotient stack, then an $\alpha$-trivialization of the stack gives rise to a $z$-trivialization of the category $\QCoh(X)$, $\IndCoh(X)$, et cetera.
\item If $\alpha \in G$ is central, then there is a canonical $\alpha$-trivialization of $\cL(BG) = G/G$.
\item Combining the two above, we have an $\alpha$-trivialization on $\cL(X/G)$.  
\end{enumerate}

\begin{exmp}
We return to Example \ref{exmp BT}.  Recall that $T_\lambda = \ker(\lambda)$ is the (derived) kernel of $\lambda \in X^\bullet(T)$ for $\alpha \in T$.  We have descriptions
$$\Coh({\cL}_\alpha(BT))^{S^1} = \bigoplus_{\lambda \in X^\bullet(T)} \Coh(T_\lambda),$$ 
$$\Coh({\cL}(BT))^{S^1(\alpha)} = \bigoplus_{\lambda \in X^\bullet(T)} \Coh(\alpha^{-1}T_{\lambda})$$
and the $\alpha$-trivial blocks correspond to those where $\alpha \in T_\lambda$.  In particular, applying the Tate construction, we have that the inclusion of the $\alpha$-trivial block induces an equivalence
$$\Coh({\cL}(BT))^{\Tate} \simeq \Coh({\cL}(BT))^{\Tate(\alpha)}$$
for all twists $\alpha$.
\end{exmp}

\subsubsection{Localization for non-reductive groups}\label{loc BB}

The equivariant localization in Theorem \ref{LocThm} can be extended to non-\hspace{0pt}reductive groups $K$ as follows: given a quotient stack $X/K$, we can consider instead the quotient stack $(X \times^K G)/G$, and do equivariant localization over $G//G$.  In general, the map $K//K \rightarrow G//G$ is neither injective nor surjective (e.g. when $K=U$ is a unipotent subgroup).

We pay special attention to the case where $K = B \subset G$ is a Borel subgroup of $G$.  In this case, we have the diagram
$$\begin{tikzcd}
\cL(X/B) = \cL((X \times^B G)/G) \arrow[r] & \cL(BB) = B/B \arrow[r]\arrow[d]  & \cL(BG) = G/G \arrow[d] \\
& B//B = H \arrow[r, "\nu"'] & G//G = H//W.
\end{tikzcd}$$
For given $\alpha \in H$ and $[\alpha] = \nu(\alpha) \in H//W$, we can have an $W_{G(\alpha)}$-equivariant identification of the $\alpha$-fixed points of $G/B$ with $\nu^{-1}([\alpha]) \times G(\alpha)/B(\alpha)$, where here $G(\alpha)$ is the centralizer of $\alpha$ and $B(\alpha) \subset G(\alpha)$ is the Borel subgroup.  Thus,
$$\cL_{[\alpha]}(BB) = \cL_{[\alpha]}(G \bs G/B) = \cL_{[\alpha]}(G(\alpha) \bs (G/B)^\alpha) = \coprod_{\nu^{-1}([\alpha])} \cL_{[\alpha]}(G(\alpha) \bs G(\alpha)/B(\alpha)).$$
We define $\cL_\alpha(BB)$ to be the connected component corresponding to $\alpha \in \nu^{-1}([\alpha])$, and then may define for $\alpha \in B//B = H$
$$\cL_\alpha(X/B) := \cL_{[\alpha]}((X \times^B G)/G) \times_{\cL_{[\alpha]}(BB)} \cL_\alpha(BB).$$
For details, see Example 1.0.6 of \cite{Ch}.

%\bibliographystyle{utphys}
%\bibliographystyle{alphaurl}
%\bibliography{BZSVpaper}
\bibliographystyle{amsplain}

\end{document}